\documentclass{article}
\usepackage{amsmath,amssymb,titling,amsthm,enumitem,mathrsfs,euscript,upgreek,wasysym}
\usepackage[utf8]{inputenc}
\usepackage[margin=1in]{geometry}

\theoremstyle{plain}
\newtheorem{thm}{Theorem}[section]

\newenvironment{axioms}
 {\enumerate[label=\textbf{A\arabic*.}, ref=A\arabic*]}
 {\endenumerate}
\makeatletter
\newcommand\varitem[1]{\item[\textbf{A\arabic{enumi}\rlap{$#1$}.}]%
  \edef\@currentlabel{A\arabic{enumi}{$#1$}}}
\makeatother

\newtheorem{innercustomthm}{Theorem}
\newenvironment{customthm}[1]
  {\renewcommand\theinnercustomthm{#1}\innercustomthm}
  {\endinnercustomthm}

\newtheorem{innercustomdefn}{Definition}
\newenvironment{customdefn}[1]
  {\renewcommand\theinnercustomdefn{#1}\innercustomdefn}
  {\endinnercustomdefn}
\newtheorem{innercustomcor}{Corollary}
\newenvironment{customcor}[1]
  {\renewcommand\theinnercustomcor{#1}\innercustomcor}
  {\endinnercustomcor}
\newtheorem{innercustomlem}{Lemma}
\newenvironment{customlem}[1]
  {\renewcommand\theinnercustomlem{#1}\innercustomlem}
  {\endinnercustomlem}
\newtheorem{innercustomexmp}{Example}
\newenvironment{customexmp}[1]
  {\renewcommand\theinnercustomexmp{#1}\innercustomexmp}
  {\endinnercustomexmp}

\theoremstyle{definition}

\theoremstyle{remark}

\newtheorem*{note}{Note}

\begin{document}
{\centering\scshape\Large\textsc{The Ordinals as a Consummate Abstraction of Number Systems} \par}
{\centering\scshape\textsc{Alec Rhea} \par}

\vspace{5mm}

\begin{abstract}
In the course of many mathematical developments involving 'number systems' like $\mathbb{N}, \mathbb{Z}, \mathbb{Q}, \mathbb{R}, \mathbb {C}$ etc., it sometimes becomes necessary to abstract away and study certain properties of the number system in question so that we may better understand objects having these properties in a more general setting -- a relevant example for this paper would be Hausdorffs $\eta_\varepsilon$-fields, objects abstracted from the property that there are no two disjoint intervals in $\mathbb{R}$ of cardinality $\alpha \leq \aleph_0$ whose union is $\mathbb{R}$. \\ 

My goal will be to reverse this process, so to speak -- I will begin in one of the most abstract mathematical settings possible, using only the undefined notions and axioms of MK class theory to define and subsequently add structure to the class of all ordinals, $O_n$. I proceed in this fashion until the construction (or a subclass thereof) is structurally isomorphic to whichever 'number system' we wish to consider. In the course of doing so, I construct a proper class of set-sized non-isomorphic discretely ordered rings and a proper class-sized discretely ordered ring, a proper class of non-isomorphic set-sized densely ordered fields that are not real-closed and a proper class sized densely ordered field that is not real-closed, a proper class of non-isomorphic real-closed fields, and a proper class of non-isomorphic algebraically closed fields.  I also propose a new definition for the Surreal and Surcomplex numbers, using a 'final culmination' of the process used to construct real-closed fields and their algebraic closures.\\ 

This process yields a straightforward and logically satisfying method for moving back and forth between well-studied number systems like $\mathbb{R}$ and very abstract number systems, such as an $\eta_{_{\omega^\omega}}$-field. \\
\end{abstract}

\section{Introduction}

\subsection{Undefined Notions}
The undefined notions in this development are those of a {\bf class}, denoted by a capital hollow letter $\mathbb{A},\mathbb{B},\mathbb{C},...$ and {\bf membership}, denoted by $\in$.  Membership may either hold between two classes $\mathbb{A}$ and $\mathbb{B}$, denoted $\mathbb{A}\in\mathbb{B}$, or not hold, denoted by $\mathbb{A}\notin\mathbb{B}$.\\

$\mathbb{A}$ is a {\bf set} iff there is a $\mathbb{B}$ such that $\mathbb{A}\in\mathbb{B}$. $\mathbb{A}$ is a {\bf proper class} iff $\mathbb{A}$ is not a set. Lowercase italic letters {\it a, b, c, ... , x, y, z} are used to denote sets unless otherwise stated.\\

Note that, accordingly, all classes only have other classes (that are sets) as members, and those members also only have other classes as members, so on and so forth down to the empty class.  Although this may seem strange at first blush, it is sufficient for the development of all modern mathematics and greatly simplifies the development of most branches therein when compared to a set theory that allows the existence of non-class based objects, often called {\it urelemente}. \\

\subsection{Axioms}

Note that this section primarily exists for posterity, and for any readers who may not be familiar with the feel of an axiomatic development of the most primitive notions in mathematics.  For readers who are comfortable with such, all of the notions introduced here will be perfectly familiar and this section may be skipped in favor of section 2. There will be a summary list of the axioms at the end of this section if the reader wishes to glance at them before moving on. \\

In order to present the axioms in their most succinct form, it is useful to flesh out a few definitions to use when expressing them. This is done legitimately within the framework provided thus far using a combination of simple definitions and theorems, but all we are really doing is introducing some useful symbolic notation to abbreviate assertions that we wish to express about classes and membership. To wit:\\

\begin{itemize}
\item {\bf A1 -- Axiom of Extensionality} $$\forall \mathbb{A} \forall \mathbb{B} \big[\forall \mathbb{C}\big(\mathbb{C} \in \mathbb{A} \iff \mathbb{C} \in \mathbb{B}\big) \iff \mathbb{A}=\mathbb{B}\big].$$ \\
\end{itemize}

Thus we see that two classes $\mathbb{A}$ and $\mathbb{B}$ are the same class if and only if they have the same members. \\

\begin{customdefn}{1.1}
The expressions $$\mathbb{A}=\mathbb{A}, \ \mathbb{A}=\mathbb{B}, \ \mathbb{A}=\mathbb{C}, \dots, \mathbb{B}=\mathbb{A}, \ \mathbb{B}=\mathbb{B}, \ \mathbb{B}=\mathbb{C}, \dots, \mathbb{C}=\mathbb{A}, \ \mathbb{C}=\mathbb{B}, \ \mathbb{C}=\mathbb{C}, \dots,$$ are all {\bf class-theoretical formulas}, as are $$\mathbb{A}\in\mathbb{A}, \ \mathbb{A}\in\mathbb{B}, \ \mathbb{A}\in\mathbb{C}, \dots, \mathbb{B}\in\mathbb{A}, \ \mathbb{B}\in\mathbb{B}, \ \mathbb{B}\in\mathbb{C}, \dots, \mathbb{C}\in\mathbb{A}, \ \mathbb{C}\in\mathbb{B}, \ \mathbb{C}\in\mathbb{C}, \dots.$$ If $\phi$ and $\psi$ are class-theoretical formulas, then so are $$\neg \phi, \ (\phi \vee \psi), \ (\phi \wedge \psi), \ (\phi \Rightarrow \psi), \ (\phi \iff \psi), \ \exists \mathbb{A}\phi, \ \exists \mathbb{B}\phi, \dots, \forall\mathbb{A}\phi, \ \forall\mathbb{B}\phi, \dots .$$ Class-theoretical formulas can only be obtained by finitely many applications of the processes just mentioned. \\
\end{customdefn}

Class-theoretical formulas are accordingly statements that we make about relationships between classes using equality, membership, the logical relation symbols and the logical quantification symbols. All theorems, definitions and axioms are class-theoretical formulas. For those more familiar with formal logic, they play the role of a {\it predicate}. \\

Note that, under this definition, expressions involving lowercase italic letters {\it a, b, c, \dots} are not technically class-theoretical formulas as they use non-hollow letters. This is acceptable, however, since any expression involving a lowercase italic letter could be replaced by a class-theoretical formula that signified exactly the same claim. For example: \\

\begin{customthm}{1.2}
$\forall a \exists \mathbb{B} \big(a \in \mathbb{B}\big).$
\end{customthm}
\begin{proof}
Suppose that $\exists a\forall \mathbb{B}\big(a\notin \mathbb{B}\big)$. Then $a$ is a proper class, a contradiction. Consequently, $\forall a \exists \mathbb{B} \big(a \in \mathbb{B}\big)$. \\
\end{proof}

This theorem is equivalent, on the basis of Definition 1.1, to: \\

\begin{customthm}{1.2'}
$\forall \mathbb{A}\big[ \exists\mathbb{C}\big(\mathbb{A}\in\mathbb{C}\big) \Rightarrow \exists \mathbb{B} \big(\mathbb{A} \in \mathbb{B}\big)\big].$ \\
\end{customthm}

Thusly, we see that any expression involving lowercase italic letters is really just a convienient way of wrapping up the claim that the classes in question are a priori known to be members of other classes. \\ 

\begin{customthm}{1.3}
$\forall x\big(x\in\mathbb{A} \iff x\in\mathbb{B}\big)\Rightarrow\mathbb{A}=\mathbb{B}.$
\end{customthm}
\begin{proof}
Assume that $\forall x\big(x\in\mathbb{A} \iff x\in\mathbb{B}\big)$. Let $\mathbb{C}\in\mathbb{A}$; then $\mathbb{C}$ is a set, so $\mathbb{C}\in\mathbb{B}$. Now let $\mathbb{D}\in\mathbb{B}$; then $\mathbb{D}$ is a set, so $\mathbb{D}\in\mathbb{A}$. Consequently we have that $\forall \mathbb{C}\big(\mathbb{C} \in \mathbb{A} \iff \mathbb{C} \in \mathbb{B}\big) \Rightarrow \mathbb{A}=\mathbb{B}$ by the extensionality axiom, completing the proof.\\
\end{proof}

\begin{itemize}
\item {\bf A2 -- Class-Building Axioms} \\ If $\phi$ is a class-theoretical formula not involving $\mathbb{A}$, then the following is an axiom: $$\exists \mathbb{A} \forall \mathbb{X} \big( \mathbb{X} \in \mathbb{A} \iff \mathbb{X} \ \text{is a set} \wedge \phi(\mathbb{X})\big).$$ Similarly, if $\phi$ does not involve $\mathbb{B}$, then the following is an axiom: $$\exists \mathbb{B} \forall \mathbb{X} \big( \mathbb{X} \in \mathbb{B} \iff \mathbb{X} \ \text{is a set} \wedge \phi(\mathbb{X})\big),$$ and so on for other letters. Letters other than $\mathbb{X}$ may also be used. \\
\end{itemize}

\begin{customcor}{1.4}
If $\phi$ is a class-theoretical formula not involving $\mathbb{A}$ or $\mathbb{B}$, and if $\forall \mathbb{X} \big( \mathbb{X} \in \mathbb{A} \iff \mathbb{X} \ \text{is a set} \wedge \phi(\mathbb{X})\big)$ and $\forall \mathbb{X} \big( \mathbb{X} \in \mathbb{B} \iff \mathbb{X} \ \text{is a set} \wedge \phi(\mathbb{X})\big)$, then $\mathbb{A} = \mathbb{B}$.  Similarly for formulas not involving either $\mathbb{A}$ or $\mathbb{C}, \ \mathbb{B}$ or $\mathbb{D}$, etc.
\end{customcor}
\begin{proof}
Suppose that $\mathbb{C}\in\mathbb{A}$; then $\mathbb{C}$ is a set and $\phi(\mathbb{C})$, thus $\mathbb{C}\in\mathbb{B}$. Now suppose that $\mathbb{D}\in\mathbb{B}$; then $\mathbb{D}$ is a set and $\phi(\mathbb{D})$, thus $\mathbb{D}\in\mathbb{A}$. Consequently we have that $\forall \mathbb{C}\big(\mathbb{C} \in \mathbb{A} \iff \mathbb{C} \in \mathbb{B}\big) \Rightarrow \mathbb{A}=\mathbb{B}$ by the extensionality axiom, completing the proof.\\
\end{proof}

\begin{customdefn}{1.5}
For any class-theoretical formula $\phi(\mathbb{X})$ not involving $\mathbb{A}$, let $\{\mathbb{X}: \phi(\mathbb{X}) \}$ be the unique class $\mathbb{A}$ such that $\forall\mathbb{X}\big(\mathbb{X}\in\mathbb{A} \iff \mathbb{X} \ \text{is a set} \ \wedge \phi(\mathbb{X})\big)$. Similarly if $\phi(\mathbb{X})$ does not involve $\mathbb{B}, \ \mathbb{C},$ and so on. This definition is justified by Corollary 1.4. Again, letters other than $\mathbb{X}$ may be used, and even lowercase letters.\\
\end{customdefn}

The symbolism introduced in Definition 1.5 captures the entire force of the class-building axioms; $\{\mathbb{X}: \phi(\mathbb{X}) \}$ may be read as "the class of all sets $\mathbb{X}$ such that $\phi(\mathbb{X})$". In what follows, the class-building axioms will always be used simply by defining classes equal to $\{\mathbb{X}: \phi(\mathbb{X}) \}$ for some $\phi(\mathbb{X})$. \\ Note that every new definition of the form $\mathbb{A}=\{\mathbb{X}:\phi(\mathbb{X})\}$ is thusly one of the class-building axioms, and can be used as an axiom when constructing a proof.  Moreover, it will not be necessary to discuss class-theoretical formulas generally any more -- all class-theoretical formulas from here on in will be explicit. \\

\begin{customdefn}{1.6}
$\mathbb{A} \subseteq \mathbb{B} \iff \forall \mathbb{C}\big(\mathbb{C}\in\mathbb{A} \Rightarrow \mathbb{C}\in\mathbb{B}\big)$. $\mathbb{A}\subseteq\mathbb{B}$ is read as "$\mathbb{A}$ {\bf is included in} $\mathbb{B}$" or "$\mathbb{A}$ {\bf is contained in} $\mathbb{B}$"; $\subseteq$ is called {\bf inclusion}. We say that $\mathbb{A}$ is a {\bf subclass} of $\mathbb{B}$ and $\mathbb{B}$ is a {\bf superclass} of $\mathbb{A}$; if $\mathbb{A}$ is a set, then $\mathbb{A}$ is a {\bf subset} of $\mathbb{B}$; if $\mathbb{B}$ is a set, $\mathbb{B}$ is a {\bf superset} of $\mathbb{A}$.\\
\end{customdefn}

\begin{customcor}{1.7}
$\mathbb{A}\subseteq\mathbb{B} \iff \forall x \big(x \in \mathbb{A} \Rightarrow x \in \mathbb{B}\big).$
\end{customcor}
\begin{proof}
Suppose that $\mathbb{A}\subseteq\mathbb{B}$. Let $x\in\mathbb{A}$; then $x\in\mathbb{B}$. Consequently, $\forall x \big(x\in\mathbb{A} \Rightarrow x\in\mathbb{B}\big)$. Now, suppose that $\forall x \big(x\in\mathbb{A} \Rightarrow x\in\mathbb{B}\big)$. Let $\mathbb{C}\in\mathbb{A}$; then $\mathbb{C}$ is a set, so $\mathbb{C}\in\mathbb{B}$. Consequently, $\mathbb{A}\subseteq\mathbb{B}$. This completes the proof. \\
\end{proof} 

\begin{customcor}{1.8}
$\mathbb{A}=\mathbb{B}\iff\mathbb{A}\subseteq\mathbb{B}$ and $\mathbb{B}\subseteq\mathbb{A}$.
\end{customcor}
\begin{proof}
This follows immedately from the definition of $\subseteq$ and {\bf A1}. \\
\end{proof}

\begin{customcor}{1.9}
If $\mathbb{A}$ is a set and $\mathbb{B}\subseteq\mathbb{A}$, then $\mathbb{B}$ is a set.
\end{customcor}
\begin{proof}
Suppose $\mathbb{A}$ is a set and $\mathbb{B}\subseteq\mathbb{A}$; by the power set axiom we have that $\exists c\big(\mathbb{X}\subseteq\mathbb{A}\Rightarrow\mathbb{X}\in c\big)$, thus $\mathbb{B}\subseteq\mathbb{A}\Rightarrow\mathbb{B}\in c\Rightarrow\mathbb{B}$ is a set.  This completes the proof. \\
\end{proof}

Corollary 1.8 amounts to a useful recasting of the Axiom of Extensionality, now using the symbols and language of inclusion that we have defined. Much of what is to come will be comparable to this process -- we have very abstract and powerful statements that we are using as our axioms, and we will define new symbols and lanuguage that forces these abstract notions to match up with more familiar intuitive notions. \\

\begin{itemize}
\item {\bf A3 -- Power Set Axiom} $\forall a \exists b \forall \mathbb{C} \big(\mathbb{C} \subseteq a \Rightarrow \mathbb{C} \in b\big).$ \\
\end{itemize}

This axiom says intiutively that if a class $\mathbb{A}$ is 'small enough' to be a set, then there is another class $\mathbb{B}$ that is 'small enough' to be a set which contains every subclass of $\mathbb{A}$. We already knew that there was a class with this property, namely $\{\mathbb{X}:\mathbb{X}\subseteq \mathbb{A}\}$ -- this axiom assures us that there is a set with this property, and this is important in what follows. A similar  remark applies to almost all of the other axioms introduced.\\

\begin{itemize}
\item {\bf A4 -- Pairing Axiom} $\forall a \forall b \exists c \big(a \in c \wedge b \in c\big).$ \\
\item {\bf A5 -- Union Axiom} $\forall a \exists b \forall \mathbb{C} \big(\mathbb{C} \in a \Rightarrow \mathbb{C} \subseteq b\big).$\\
\end{itemize}

\begin{customdefn}{1.10}
\begin{enumerate}
\item $\mathbb{V} = \{x:x=x\}$. $\mathbb{V}$ is called the {\bf universe}.
\item $0 = \{x:x\neq x\}$. $0$ is called the {\bf empty class}. \\
\end{enumerate}
\end{customdefn}

\begin{customcor}{1.11}
$\forall \mathbb{X} \big(\mathbb{X}\subseteq\mathbb{V} \wedge \mathbb{X}\notin 0\big).$
\end{customcor}
\begin{proof}
Let $\mathbb{A}\in\mathbb{X}$; then $\mathbb{A}$ is a set and $\mathbb{A}=\mathbb{A}$, so $\mathbb{A}\in\mathbb{V}$, thus $\mathbb{X}\subseteq\mathbb{V}$. $\mathbb{X}$ being arbitrary, we have that $\forall\mathbb{X}\big(\mathbb{X}\subseteq\mathbb{V}\big)$. Now suppose that $\mathbb{A}\in 0$; then $\mathbb{A}$ is a set and $\mathbb{A}\neq\mathbb{A}$, a contradiction. Consequently $\forall \mathbb{X} \big(\mathbb{X}\notin 0\big).$ This completes the proof. \\
\end{proof}

\begin{customdefn}{1.12}
$\mathbb{A} \cap \mathbb{B} = \{x:x\in\mathbb{A}\wedge x\in\mathbb{B}\}$. $\mathbb{A} \cap \mathbb{B}$ is called the {\bf intersection} of $\mathbb{A}$ and $\mathbb{B}$. \\
\end{customdefn}

\begin{itemize}
\item {\bf A6 -- Regularity Axiom} $\forall \mathbb{A}\big[ \mathbb{A} \neq 0 \Rightarrow \exists \mathbb{X} \big(\mathbb{X} \in \mathbb{A} \wedge \mathbb{A} \cap \mathbb{X} = 0\big)\big].$ \\
\end{itemize}

\begin{customcor}{1.13}
$\nexists \mathbb{A} \big(\mathbb{A}\in\mathbb{A}\big).$
\end{customcor}
\begin{proof}
Suppose that $\mathbb{A}\in\mathbb{A}$, so $\mathbb{A}$ is a set, and let $\mathbb{D}=\{x:x=\mathbb{A}\}$; then $\mathbb{D}\neq 0$ and $\forall \mathbb{X} \big(\mathbb{X}\in\mathbb{D} \Rightarrow \mathbb{A}\in\mathbb{X} \cap \mathbb{D}\Rightarrow \mathbb{X} \cap \mathbb{D}\neq 0\big) \Rightarrow \nexists \mathbb{X}\big(\mathbb{X}\in\mathbb{D} \wedge \mathbb{X}\cap\mathbb{D}=0\big)$, a contradiction by the regularity axiom. Consequently $\nexists \mathbb{A} \big(\mathbb{A}\in\mathbb{A}\big)$, completing the proof. \\
\end{proof}

\begin{customdefn}{1.14}
$\mathcal{S}\mathbb{A}=\{x: x\in\mathbb{A} \vee x=\mathbb{A}\}.$ $\mathcal{S}\mathbb{A}$ is called the {\bf successor} of $\mathbb{A}$. \\
\end{customdefn}

\begin{customcor}{1.15}
$\forall x \big(x\in\mathcal{S}a \iff x\in a \vee x=a\big).$
\end{customcor}
\begin{proof}
Suppose $x\in\mathcal{S}a$; then $x\in a$ or $x=a$ by definition. Now, suppose that $x\in a$ or $x=a$. If $x=a$, then $x\in\mathcal{S}a$; if $x\in a$, then $x\in \mathcal{S}a$. Consequently, $x\in\mathcal{S}a \iff x\in a \vee x=a$; this completes the proof. \\
\end{proof}

\begin{customthm}{1.16}
\begin{enumerate}
\item $\mathbb{A}$ is a proper class $\iff \mathbb{A}=\mathcal{S}\mathbb{A}$. 
\item $\mathbb{A}$ is a set $\iff \mathbb{A}\in\mathcal{S}\mathbb{A}$.
\end{enumerate}
\end{customthm}
\begin{proof}
Let $\mathbb{A}$ be a proper class, so $\mathbb{A}\notin\mathcal{S}\mathbb{A}$. Let $\mathbb{X}\in\mathbb{A}$; then $\mathbb{X}$ is a set, so $\mathbb{X}\in\mathcal{S}\mathbb{A}$ by definition. Now let $\mathbb{Y}\in\mathcal{S}\mathbb{A}$, so $\mathbb{Y}\in\mathbb{A}$ or $\mathbb{Y}=\mathbb{A}$. If $\mathbb{Y}=\mathbb{A}$ then $\mathbb{A}\in\mathcal{S}\mathbb{A}$, a contradiction; thus $\mathbb{Y}\in\mathbb{A}$. Consequently $\forall \mathbb{C}\big(\mathbb{C} \in \mathbb{A} \iff \mathbb{C} \in \mathcal{S}\mathbb{A}\big) \Rightarrow \mathbb{A}=\mathcal{S}\mathbb{A}$ by the axiom of extensionality. Now suppose that $\mathbb{A}=\mathcal{S}\mathbb{A}$.  If $\mathbb{A}$ is a set then $\mathbb{A}\in\mathcal{S}\mathbb{A}$ by definition, thus $\mathbb{A}=\mathcal{S}\mathbb{A}\Rightarrow\mathcal{S}\mathbb{A}\in\mathcal{S}\mathbb{A}$, a contradiction by Corollary 1.13; thus we conclude that $\mathbb{A}$ is not a set, making it a proper class. This completes the proof of {\it 1}. \\ For {\it 2}, suppose that $\mathbb{A}$ is a set; then $\mathbb{A}\in\mathcal{S}\mathbb{A}$ by definition. Now suppose that $\mathbb{A}\in\mathcal{S}\mathbb{A}$; then $\mathbb{A}$ is a set. This completes the proof. \\
\end{proof}

Thus we see that the logical notion of being a set vs being a proper class can be entirely recast in terms of classes and their successors, using equality and membership. \\

\begin{itemize}
\item {\bf A7 -- Infinity Axiom} $\exists a \big[0 \in a \wedge \forall \mathbb{X} \big( \mathbb{X} \in a \Rightarrow \mathcal{S}\mathbb{X} \in a \big)\big].$ \\
\end{itemize}

The Infinity Axiom derives its name from the fact that the set $a$ asserted to exist clearly has an infinite number of elements. Using the power set axiom we may obtain even larger sets from this axiom, however we are still unable to create sets that are as 'large' as we would like. For this purpose we must develop the notion of an ordered pair, a relation and a function. \\

\begin{customdefn}{1.17}
$\{\mathbb{A},\mathbb{B}\}=\{x:x=\mathbb{A}\vee x=\mathbb{B}\}$. $\{\mathbb{A},\mathbb{B}\}$ is called the {\bf doubleton $\mathbb{A},\mathbb{B}$} or the {\bf unordered pair $\mathbb{A},\mathbb{B}$}. \\
\end{customdefn}

\begin{customthm}{1.18}
$\{a,b\}$ is a set.
\end{customthm}
\begin{proof}
By the pairing axiom $\exists c \big(a\in c \wedge b\in c \big)\Rightarrow\{a,b\}\subseteq c$. By the power set axiom $\exists d \big(x\subseteq c \Rightarrow x \in d\big)$, so $\{a,b\}\subseteq c \Rightarrow \{a,b\}\in d\Rightarrow\{a,b\}$ is a set. This completes the proof. \\
\end{proof}

\begin{customcor}{1.19}
\begin{enumerate} 
\item $\mathbb{X}\in \{a,b\} \iff \mathbb{X}=a \vee \mathbb{X}=b.$ 
\item $a\in\{a,b\}$.
\item $b\in\{a,b\}$.
\end{enumerate}
\end{customcor}
\begin{proof}
Suppose that $\mathbb{X}\in\{a,b\}$; then $\mathbb{X}=a$ or $\mathbb{X}=b$ by definition. Suppose $\mathbb{X}=a$; then $\mathbb{X}\in\{a,b\}$. Suppose $\mathbb{X}=b$; then $\mathbb{X}\in\{a,b\}$. This proves {\it 1}. Since $a=a$ and $a$ is a set and $\{a,b\}=\{x:x=a\vee x=b\}$, $a\in\{a,b\}$ by the class building axioms, proving {\it 2}. The argument for {\it 3} is identical. This completes the proof. \\
\end{proof}

\begin{customdefn}{1.20}
$\{\mathbb{A}\}=\{\mathbb{A},\mathbb{A}\}$. $\{\mathbb{A}\}$ is called the {\bf singleton $\mathbb{A}$}. \\
\end{customdefn}

\begin{customcor}{1.21}
$\{a\}$ is a set. \\
\end{customcor}

\begin{customcor}{1.22}
$\mathbb{X}\in \{a\} \iff \mathbb{X}=a$. \\
\end{customcor}

\begin{customthm}{1.23}
$\{a,b\}=\{c,d\} \Rightarrow \big(a=c \wedge b=d\big) \vee \big(a=d \wedge b=c\big)$. \\
\end{customthm}
\begin{proof}
Assume that $\{a,b\}=\{c,d\}$. Then $a\in\{a,b\}\Rightarrow a\in\{c,d\} \Rightarrow \big(a=c \vee a=d\big)$. Similarly, $b\in\{a,b\}\Rightarrow b\in\{c,d\}\Rightarrow \big(b=c \vee b=d\big)$. If $\big(a=c \wedge b=d\big)$ or $\big(a=d \wedge b=c\big)$ we are satisfied. Suppose that $a=c$ and $b=c$; then $d\in\{c,d\}\Rightarrow d\in\{a,b\}\Rightarrow \big(d=a \vee d=b\big)\Rightarrow d=c$, thus $a=b=c=d$ and we are again satisfied. The argument if $a=d$ and $b=d$ is identical.  This completes the proof. \\
\end{proof}

\begin{customdefn}{1.24}
$(\mathbb{A},\mathbb{B})=\{\{\mathbb{A}\},\{\mathbb{A},\mathbb{B}\}\}$. $(\mathbb{A},\mathbb{B})$ is called the {\bf ordered pair $\mathbb{A},\mathbb{B}$} with {\bf first coordinate} $\mathbb{A}$ and {\bf second coordinate} $\mathbb{B}$. \\
\end{customdefn}

\begin{customcor}{1.25}
$(a,b)$ is a set.
\end{customcor}
\begin{proof}
By Theorem 1.16, $\{a,b\}$ is a set for any $a$ and $b$.  By Corollary 1.19, $\{a\}$ is a set for any $a$. Letting $c=\{a\}$ and $d=\{a,b\}$ we obtain $(a,b)=\{c,d\}$, thus $(a,b)$ is a set.  This completes the proof. \\
\end{proof}

\vspace{20mm}

\begin{customthm}{1.26}
$(a,b)=(c,d) \Rightarrow \big(a=c \wedge b=d\big)$.
\end{customthm}
\begin{proof}
Suppose $(a,b)=(c,d)$, so $\{\{a\},\{a,b\}\}=\{\{c\},\{c,d\}\}$.  By Theorem {\it 1.21} we have two cases to work with, namely {\bf (1)} $\{a\}=\{c\}$ and $\{a,b\}=\{c,d\}$ or {\bf (2)} $\{a\}=\{c,d\}$ and $\{a,b\}=\{c\}$.  \\ {\it Case 1} $\mathbb{X}\in\{a\}\Rightarrow\big(\mathbb{X}\in\{c\}\wedge\mathbb{X}=a\big)$, and $\mathbb{X}\in\{c\}\Rightarrow\mathbb{X}=c\Rightarrow a=c.$ We then have that $b\in\{a,b\}\Rightarrow b\in\{c,d\}\Rightarrow\big(b=c \vee b=d\big)$. If $b=d$ we are satisfied; suppose that $b=c$. Then $a=b$, so $\{a,b\}=\{a\}\Rightarrow\{c,d\}=\{a\}=\{c\}\Rightarrow c=d \Rightarrow b=d$, and we are again satisfied. \\ {\it Case 2} $\big( c\in\{c,d\}\wedge d\in\{c,d\}\big)\Rightarrow\big( c\in\{a\} \wedge d\in\{a\}\big) \Rightarrow c=a=d.$ Similarly, $b\in\{a,b\}\Rightarrow b\in\{c\}\Rightarrow b=c \Rightarrow b=d$, and we are satisfied.  This completes the proof. \\
\end{proof}

Theorem 1.26 presents the only class-theoretical formula that we really wished to establish using the ordered pair structure. \\

\begin{customdefn}{1.27}
\begin{enumerate}
\item $\mathbb{X}$ is a {\bf relation} iff $\forall \mathbb{A}\big[\mathbb{A}\in\mathbb{X} \Rightarrow \exists c \exists d \big(\mathbb{A}=(c,d)\big)\big].$
\item $dmn\mathbb{X}=\{x:\exists y \big[(x,y)\in\mathbb{X}\big]\}$. $dmn\mathbb{X}$ is called the {\bf domain} of $\mathbb{X}$.
\item $rng\mathbb{X}=\{y:\exists x \big[(x,y)\in\mathbb{X}\big]\}$. $rng\mathbb{X}$ is called the {\bf range} of $\mathbb{X}$.
\item $\mathbb{F}$ is a {\bf function} iff $\mathbb{F}$ is a relation and $\forall x\forall y\forall z \big[(x,y)\in\mathbb{F}\wedge (x,z)\in\mathbb{F} \Rightarrow y=z\big]$. \\
\end{enumerate}
\end{customdefn}

Thus we see that a relation is simply a class of ordered pairs, with a function being a class of ordered pairs such that no set $x$ is the first coordinate of two unique ordered pairs in the class. We are now prepared to succinctly formulate our last two axioms. \\

\begin{itemize}
\item {\bf A8 -- Axiom of Substitution} If $\mathbb{F}$ is a function and $dmn\mathbb{F}$ is a set, then $rng\mathbb{F}$ is a set. \\
\item{\bf A9 -- Relational Axiom of Choice} If $\mathbb{R}$ is a relation, there exists a function $\mathbb{F}$ such that $\mathbb{F}\subseteq\mathbb{R}$ and $dmn\mathbb{R}=dmn\mathbb{F}$. \\
\end{itemize}

This completes our discussion of the axioms of MK set theory, which will serve as the foundation for the remainder of this paper. A succinct list of the nine axiom schemata presented here is given below. \\

\subsubsection{Axiom List}

\begin{axioms}
\item $\forall \mathbb{A} \forall \mathbb{B} \big[\forall \mathbb{C}\big(\mathbb{C} \in \mathbb{A} \iff \mathbb{C} \in \mathbb{B}\big) \Rightarrow \mathbb{A}=\mathbb{B}\big].$
\item $\exists \mathbb{A} \forall \mathbb{X} \big( \mathbb{X} \in \mathbb{A} \iff \mathbb{X} \ \text{is a set} \wedge \phi(\mathbb{X})\big), \ \text{where $\mathbb{A}$ does not occur in $\phi(\mathbb{X})$}$.
\item $\forall a \exists b \forall \mathbb{C} \big(\mathbb{C} \subseteq a \Rightarrow \mathbb{C} \in b\big).$
\item $\forall a \forall b \exists c \big(a \in c \wedge b \in c\big).$
\item $\forall a \exists b \forall \mathbb{C} \big(\mathbb{C} \in a \Rightarrow \mathbb{C} \subseteq b\big).$
\item $\forall \mathbb{A}\big[ \mathbb{A} \neq 0 \Rightarrow \exists \mathbb{X} \big(\mathbb{X} \in \mathbb{A} \wedge \mathbb{A} \cap \mathbb{X} = 0\big)\big].$
\item $\exists a \big[0 \in a \wedge \forall \mathbb{X} \big( \mathbb{X} \in a \Rightarrow \mathcal{S}\mathbb{X} \in a \big)\big].$
\item $\big(\mathbb{F} \ \text{is a function} \ \wedge \ dmn\mathbb{F} \ \text{is a set}\big) \Rightarrow rng\mathbb{F} \ \text{is a set.}$
\item $\mathbb{R} \ \text{is a relation} \Rightarrow \exists \mathbb{F}\big(\mathbb{F} \ \text{is a function} \wedge \mathbb{F}\subseteq\mathbb{R}\wedge dmn\mathbb{R}=dmn\mathbb{F}\big).$ \\
\end{axioms}

Up until now, we have very carefully been explicit in defining any symbols we use prior to their usage, regardless of commonly understood interpretations -- this was to help give a feel for the proper development of mathematics from its foundations, and because we were formulating the axioms which will serve as a basis for the remainder of the paper and as such it behooved us to leave absolutely no logical ambiguity. To speed the development from here on in (and to avoid being overly pedantic), we will allow ourselves to use symbols with well-understood common meanings, like the binary union $\cup$, without first defining them explicitly using the class-buiding axioms. For symbols that are perhaps less generally used or are closely tied to undefined notions, like the infinitary union $\bigcup$ or proper inclusion $\subset$, we will still provide explicit definitions. \\	

It is worth noting explicitly that MK class theory as presented above is sufficiently powerful to predicate on proper classes, albeit in a much more restricted fashion than sets -- this is not the case in ZFC set theory. Accordingly, we have been and will continue to refer to MK as a class theory while we refer to ZFC as a set theory. Further, proper classes are the unique logical objects that lead to paradoxes when approached using naive thought, and set theories that are sufficiently rigorous to delineate between proper classes and sets provide a context in which all known mathematical paradoxes are resolved; for example, Russel's paradox reduces to the class of all sets that are not members of themselves, which is not a member of itself and not a set. \\

As a final note, much of the material presented above closely follows the first chapter of Introduction to Set Theory (1969) by J. Donald Monk [1] with slight alterations to proofs and the addition of Theorem 1.16, and I offer much thanks to him for his clear and concise exposition on the subject. \\

\section{The Ordinals $O_n$}

As children, most of us had some experience with an argument along the lines of: "I love you." "Well, I love you $\times 2$." ... "Well I love you $\times \infty$!" "Well I love you $\times (\infty + 1)$!" ... "Well I love you $\times (\infty^{\infty})$!" etc.  The ordinals show us, completely rigorously, that we were correct in our intuition as children that there is no good reason to stop at the first $\infty$-like number if it helps your argument not to. \\

\begin{thm}
\begin{enumerate}
\item $0$ is a set.
\item $\mathbb{V}$ is a proper class.
\end{enumerate}
\end{thm}
\begin{proof}
{\it 1} follows immedately from {\bf A7}. For {\it 2}, suppose $\mathbb{V}$ is a set. Then $\mathbb{V}=\mathbb{V}$ and $\mathbb{V}$ is a set, thus $\mathbb{V}\in\mathbb{V}$ by the class building axioms, a contradiction by Corollary 1.13. We thusly conclude that $\mathbb{V}$ is not a set, completing the proof.\\
\end{proof}

\begin{customdefn}{2.2}
\begin{enumerate}
\item $\mathbb{A}$ is {\bf membership transitive} iff for all $x$ and $y$, $x\in y\in\mathbb{A} \Rightarrow x\in\mathbb{A}$.
\item $\mathbb{A}$ is an {\bf ordinal} or an {\bf ordering number} iff $\mathbb{A}$ is membership transitive and every member of $\mathbb{A}$ is membership transitive.
\item $O_n=\{x:x \ \text{is an ordinal}\}.$ $O_n$ is called the {\bf Ordinals}, or the {\bf ordering numbers}. Lowercase greek letters $\alpha, \beta, \gamma, \dots$ will be used to denote members of $O_n$ unless otherwise stated.\\
\end{enumerate}
\end{customdefn}

Note that the definition of an ordinal is \underline{very} closely tied to the undefined notions in MK class theory. \\

We will see in the next few pages that this notion of 'membership transitivity' for a class and all of its members provides a very natural context in which we can define objects that behave like the natural numbers, as well as the class of all natural numbers. Further, in this context the class of all natural numbers is also an ordinal, specifically the first infinite ordinal that we encounter in our ordering when beginning at $0$ and moving upwards. \\

\begin{customthm}{2.3}
$0\in O_n$.
\end{customthm}
\begin{proof}
By virtue of the vacuous implication, $0$ is membership transitive. Since $0$ has no members, all members of $0$ are also membership transitive by virtue of the vacuous implication; thus we see that $0$ is an ordinal. By Theorem {\it 2.1} $0$ is also a set, thus $0\in O_n$ by {\bf A2} and Definition {\it 2.2}. This completes the proof. \\
\end{proof}

\begin{customthm}{2.4}
If $\alpha$ is an ordinal, then $\mathcal{S}\alpha$ is an ordinal.
\end{customthm}
\begin{proof}
Suppose that $\alpha$ is an ordinal, so $\alpha$ is membership transitive and all of its members  are as well. Consequently, by the definition of $\mathcal{S}\alpha$ we have that all members of $\mathcal{S}\alpha$ are membership transitive, so we only need to show that $\mathcal{S}\alpha$ is membership transitive. Let $x\in y\in\mathcal{S}\alpha$, so $y=\alpha$ or $y\in \alpha$. Suppose $y=\alpha$, so $\mathcal{S}\alpha=\mathcal{S}y$, thus $x\in y \Rightarrow x \in \mathcal{S}y\Rightarrow x \in \mathcal{S}\alpha$. Now suppose $y\in \alpha$, so $x \in \alpha$ since $\alpha$ is membership transitive, thus $x\in\mathcal{S}\alpha$. Thusly we have that $x\in y \in \mathcal{S}\alpha \Rightarrow x\in\mathcal{S}\alpha$, completing the proof. \\
\end{proof}

Thus we immedately see that there are a profusion of ordinals, $0,\mathcal{S}0,\mathcal{S}\mathcal{S}0,\dots$. The next definition and theorem will allow us to obtain even more members of $O_n$. \\

\begin{customdefn}{2.5}
\begin{enumerate}
\item $\bigcup\mathbb{A} = \{x:\exists y\big[x\in y \in \mathbb{A}\big]\}$. $\bigcup\mathbb{A}$ is called the {\bf union} of $\mathbb{A}$.
\item $\bigcap\mathbb{A} = \{x:\forall y\big[y\in\mathbb{A}\Rightarrow x\in y\big]\}$. $\bigcap\mathbb{A}$ is called the {\bf intersection} of $\mathbb{A}$.  \\
\end{enumerate}
\end{customdefn}

\begin{customthm}{2.6}
If $\mathbb{A}\subseteq O_n$, then $\bigcup \mathbb{A}$ is an ordinal.
\end{customthm}
\begin{proof}
Suppose that $\mathbb{A}\subseteq O_n$, and let $x\in y\in \bigcup\mathbb{A}$. Then $y\in \alpha$ for some $\alpha\in \mathbb{A}$, and since $\alpha$ is an ordinal $x\in y\in \alpha \Rightarrow x\in \alpha \Rightarrow x\in\bigcup\mathbb{A}$, thus we see that $\bigcup\mathbb{A}$ is membership transitive. Now, suppose that $y \in \bigcup\mathbb{A}$. Since $y\in \bigcup\mathbb{A}$, $y\in \alpha$ for some $\alpha\in\mathbb{A}$ and consequently $y$ is membership transitive. This completes the proof. \\
\end{proof}

We have now proven the existence of all ordinals that are sets.  We will define a structure among them shortly, but first we must demonstrate the existence of the unique proper class ordinal. \\

\begin{customthm}{2.7}
If $\mathbb{A}$ is an ordinal, then $\mathbb{A}\subseteq O_n$.
\end{customthm}
\begin{proof}
Suppose $\mathbb{A}$ is an ordinal, so we know that all of its members are membership transitive -- we wish to show that the members of $\bigcup\mathbb{A}$ are also membership transitive. Let $x\in y\in \mathbb{A}$, so $x\in\bigcup\mathbb{A}$.  Since $\mathbb{A}$ is an ordinal and thusly membership transitive we have $x\in y\in \mathbb{A} \Rightarrow x\in \mathbb{A} \Rightarrow x$ is membership transitive. This completes the proof. \\
\end{proof}

\begin{customthm}{2.8}
$O_n$ is an ordinal.
\end{customthm}
\begin{proof}
Since all members of $O_n$ are ordinals by definition, they are all membership transitive; we wish to show that $O_n$ is also membership transitive.  Suppose that $x\in y\in O_n$. Then $y\in O_n \Rightarrow y\subseteq O_n$ by the previous theorem, so $x\in O_n$. Thusly $O_n$ is membership transitive, making it an ordinal.  This completes the proof. \\
\end{proof}

\begin{customthm}{2.9}
$O_n$ is a proper class.
\end{customthm}
\begin{proof}
If $O_n$ were a set, we would have $O_n \in O_n$ by Definition 2.2 and Theorem 2.8, a contradiction by Corollary 1.13. Thusly we conclude that $O_n$ is not a set, completing the proof. \\
\end{proof}

As a result of the previous three theorems we see that $O_n$ is a proper class ordinal. Theorem 2.11 following the next lemma is useful for much of the remaining paper, essentially establishing that membership forms a total ordering on $O_n$ -- another immediate consequence of it is the uniqueness of $O_n$ as a proper class ordinal. \\

\begin{customlem}{2.10}
If $\mathbb{A}$ is an ordinal, then $0\in\mathbb{A}$ or $0=\mathbb{A}$.
\end{customlem}
\begin{proof}
Let $\mathbb{A}$ be an ordinal; if $\mathbb{A}=0$, we are satisfied. Suppose $\mathbb{A}\neq 0$, and by the regularity axiom choose $a\in\mathbb{A}$ such that $a\cap\mathbb{A}=0$. If $a=0$, we are satisfied; suppose that $a\neq0$, so $\exists x\big(x\in a\big)$. But $\mathbb{A}$ is an ordinal and consequently membership transitive, thus $x\in a\in \mathbb{A} \Rightarrow x\in\mathbb{A}\Rightarrow a\cap\mathbb{A}=\{x\}\neq 0$, a contradiction. Thusly we conclude that $a=0$, so $0\in\mathbb{A}$. This completes the proof.  \\
\end{proof}

\begin{customthm}{2.11}
If $x$ and $y$ are ordinals, then $x=y$ or $x\in y$ or $y\in x$.
\end{customthm}
\begin{proof}
Let $\mathbb{A}=\{x:x\in O_n \wedge \exists y\big(y\in O_n \wedge y\neq x \wedge y\notin x \wedge x\notin y\big)\}$. We wish to show that $\mathbb{A}=0$.  Suppose that $A\neq 0$, and using the regularity axiom let $a\in\mathbb{A}$ such that $a\cap\mathbb{A}=0$. Since $a\in\mathbb{A}$, we have that $\mathbb{B}=\{y:y\in O_n \wedge y\neq a \wedge y\notin a \wedge a\notin y\}\neq 0$. Again using the regularity axiom, let $b\in\mathbb{B}$ such that $b\cap\mathbb{B}$=0. Note that $b\in\mathbb{B}$ gives us $b\neq a$, $b\notin a$ and $a\notin b$.  \\ Suppose $b=0$; then since $a$ is an ordinal we have $b\in a$ or $b=a$ by Theorem 2.10, a contradiction. Consequently, $b\neq 0$.  Let $x\in b$, thus $x=a$ or $a\in x$ or $x\in a$ since $b\cap\mathbb{B}=0$. If $x=a$ then $a\in b$, a contradiction. If $a\in x$, then since $b$ is an ordinal we have $a\in x\in b \Rightarrow a\in b$, another contradiction.  Thusly we conclude that $x\in a$, and consequently that $b\subseteq a$. \\ Suppose $a=0$; then since $b$ is an ordinal we have $a\in b$ or $a=b$ by Theorem 2.10, a contradiction. Consequently, $a\neq 0$. Let $y\in a$, so $y$ is an ordinal by Theorem 2.7, and suppose that $y\neq b$ and $y\notin b$ and $b\notin y$. Then $y\in\mathbb{A}$, thus $a\cap\mathbb{A}=\{y\}\neq 0$, a contradiction; consequently we have that $y=b$ or $y\in b$ or $b\in y$. If $y\in b$ then $a\subseteq b$ and consequently $a=b$, a contradiction. If $y=b$ then $b\in a$, a contradiction. If $b\in y$, then since $a$ is an ordinal we have that $b\in y \in a \Rightarrow b\in a$, another contradiction. We thusly conclude that $\mathbb{A}=0$, completing the proof. \\
\end{proof}

The previous method of proof, while somewhat more involved than the other proofs presented thus far, is actually an incredibly versatile approach. We define a class $\mathbb{A}$ full of sets $x$ that are also in a particular class $\mathbb{B}$ (in the above case $\mathbb{B}=O_n$), and require that the sets $x$ satisfy some class-theoretical formula $\phi(x)$ as a result of being in $\mathbb{A}$. This class-theoretical formula $\phi(x)$ is the negation of a class-theoretical formula $\psi(x)$ $\big(\phi(x)\equiv\neg\psi(x)\big)$, and we wish to establish that $\psi(x)$ is true for all $x\in\mathbb{B}$. We do so by proving that $\mathbb{A}$ is empty, usually by observing that we arrive at a contradiction if $\mathbb{A}\neq 0$, and since $\mathbb{A}$ is the subclass of $\mathbb{B}$ where $\phi(x)$ is true this establishes that $\neg\phi(x)\equiv\psi(x)$ is true for all $x\in\mathbb{B}$ as desired. \\ It is worth noting that this method is only tractable to prove class-theoretical formulas about {\it sets}, or classes that are already known to be in some other class where a predicate can arbitrarily hold or not.  The above theorem does not, a priori, apply to ordinals that are classes (which was implicit in the lowercase letters used in the statement of the theorem). This is a good first example of what we mean when we say that proper classes are more logically restricted objects than sets -- we can't arbitrarily check predicates on them in the above fashion. Luckily for us, the Ordinals have sufficient structure for us to circumvent this issue. \\

\begin{customthm}{2.12}
$O_n$ is the unique proper class ordinal.
\end{customthm}
\begin{proof}
Let $\mathbb{A}$ be an ordinal that is a proper class; by Theorem 2.7 we have $\mathbb{A}\subseteq O_n$.  If $\mathbb{A}=O_n$ we are satisfied; suppose $O_n\sim\mathbb{A}\neq0$, and let $x\in O_n\sim\mathbb{A}$ and $z\in\mathbb{A}$, so $x$ is an ordinal by definition and $z$ is an ordinal since $\mathbb{A}\subseteq O_n$. We then have that $z\in x$ or $z=x$ or $x\in z$ by Theorem 2.11, since $x$ and $z$ are sets. If $z\in x$ then $z\in x\in O_n\sim\mathbb{A}\Rightarrow z\in O_n\sim\mathbb{A}$ by Theorem 2.8, a contradiction since $z\in\mathbb{A}$. If $z=x$ then $z\in O_n\sim\mathbb{A}$ again, a contradiction. If $x\in z$ then $x\in z \in \mathbb{A} \Rightarrow x\in\mathbb{A}$ since $\mathbb{A}$ is an ordinal, a contradiction since $x\in O_n\sim\mathbb{A}$. We thusly conclude that $O_n\sim\mathbb{A}=0$, so $\mathbb{A}=O_n$, completing the proof. \\
\end{proof}

We have now demonstrated the existence of all types of ordinals that exist -- $0$, successors of ordinals, infinitary unions of classes full of sets that satisfy a class-theoretical formula making them ordinals, and $O_n$. {\it Cardinals} are special types of ordinals usually defined in terms of equipotence and power classes, however they are immaterial to our development at this juncture. \\

Note that by Theorem 1.16 and Theorem 2.9 we have that $\mathcal{S}O_n=O_n$ and $\mathbb{A}\in O_n\Rightarrow \mathcal{S}\mathbb{A}\neq\mathbb{A}$, and by Theorem 2.12 we have that any ordinal $\mathbb{A}$ such that $\mathbb{A}\neq O_n$ satisfies $\mathbb{A}\in O_n$, thus $O_n$ represents the only logically natural point at which we are unable to create new unique ordinals by looking at the object formed when taking the successor of a class that is an ordinal.  There is one last interesting fact regarding $O_n$ in this respect that can be established quickly using the next lemma. \\

\begin{customlem}{2.13}
$\mathbb{A}$ is a set $\iff$ $\mathcal{S}\mathbb{A}$ is a set.
\end{customlem}
\begin{proof}
Suppose $\mathbb{A}$ is a set. By the pairing axiom we have that $\exists b\big(\mathbb{A}\in b\big)$, so $\{\mathbb{A}\}\subseteq b$ and consequently $\{\mathbb{A}\}$ is a set by Corollary 1.9. Again by the pairing axiom we then have that $\exists c\big(\mathbb{A}\in c \wedge \{\mathbb{A}\}\in c\big)$, and then by the union axiom we have that $\exists d\big(x\in c \Rightarrow x\subseteq d\big)$. Suppose that $\mathbb{X}\in\mathcal{S}\mathbb{A}$, so $\mathbb{X}=\mathbb{A}$ or $\mathbb{X}\in\mathbb{A}$. If $\mathbb{X}=\mathbb{A}$ then $\mathbb{X}\in d$, and if $\mathbb{X}\in\mathbb{A}$ then $\mathbb{X}\in d$, thus $\mathbb{X}\in\mathcal{S}\mathbb{A}\Rightarrow\mathbb{X}\in d\Rightarrow\mathcal{S}\mathbb{A}\subseteq d$, and consequently $\mathcal{S}\mathbb{A}$ is a set by Corollary 1.9. \\ Now suppose that $\mathcal{S}\mathbb{A}$ is a set. If $\mathbb{A}$ is a proper class then $\mathbb{A}=\mathcal{S}\mathbb{A}$ by Theorem 1.16, thus $\mathcal{S}\mathbb{A}$ is a proper class as well, a contradiction.  We thusly conclude that $\mathbb{A}$ is a set, completing the proof. \\
\end{proof} 

Note that we have only used the axioms of MK class theory and theorems from Section 1 to prove the previous lemma; it is independent of the structure of the ordinals. \\

\begin{customcor}{2.14}
$\nexists x\big(\mathcal{S}x=O_n\big)$.
\end{customcor}
\begin{proof}
Suppose that $\mathcal{S}x=O_n$; then $x$ is a set $\Rightarrow\mathcal{S}x$ is a set by Lemma 2.13, thus $O_n$ is a set, a contradiction by Theorem 2.9. Consequently $\mathcal{S}x\neq O_n$, so we have that $x$ is a set $\Rightarrow \mathcal{S}x\neq O_n$ and we thusly conclude that $\nexists x\big(\mathcal{S}x=O_n\big)$, completing the proof. \\
\end{proof}

\begin{customcor}{2.15}
There is no ordinal $\mathbb{A}\neq O_n$ such that $\mathcal{S}\mathbb{A}=O_n$.
\end{customcor}
\begin{proof}
Let $\mathbb{A}$ be an ordinal and suppose that $\mathbb{A}\neq O_n$; then $\mathbb{A}$ is a set by Theorem 2.12, thus $\mathcal{S}\mathbb{A}\neq O_n$ by Corollary 2.14.  This completes the proof. \\
\end{proof}

We thusly see that it is impossible to obtain $O_n$ as the successor of any other ordinal, and that this fact is a nigh direct result of the locigal context we are working in. Once we have defined a satisfying notion of finite vs. infinite ordinals we will see that this directly leads to $O_n$ being the most primitive and largest infinity, primitive meaning that it's 'infinite' properties are more closely tied to the undefined notions and axioms of MK class theory than any other infinite ordinal, and largest meaning that it would be ordered above all other infinite ordinals in our ordering if it could be legitimately placed in an ordering (which it can't be as a proper class). \\

The next few definitions and theorems will establish the remaining useful facts regarding ordinals that we wish to express in this somewhat abstract symbolic form, and we will then define some new symbols that are hopefully more familiar to our intuition with numbers, like $\leq$, and finally gather together the last of the properties we care about pertaining to $O_n$ in a more familiar looking symbolism. \\

\begin{customdefn}{2.16}
$\mathbb{A}\subset\mathbb{B}\iff\mathbb{A}\subseteq\mathbb{B}$ and $\mathbb{A}\neq\mathbb{B}$. $\subset$ is called {\bf proper inclusion}. Note that we could have required that $\mathbb{A}\subseteq\mathbb{B}$ and $\exists x\big(x\in\mathbb{B}\wedge x\notin\mathbb{A}\big)$, but this definition is equivalent and generally more elegant in proofs. \\
\end{customdefn}

\begin{customthm}{2.17}
If $\mathbb{A}$ and $\mathbb{B}$ are ordinals, then $\mathbb{A}\subset\mathbb{B}\iff\mathbb{A}\in\mathbb{B}$.
\end{customthm}
\begin{proof}
Let $\mathbb{A}$ and $\mathbb{B}$ be ordinals, and suppose that $\mathbb{A}\subset\mathbb{B}$. If $\mathbb{A}=O_n$ then $\mathbb{A}\subset\mathbb{B}\Rightarrow\mathbb{B}$ is a proper class and $\mathbb{B}\neq O_n$, a contradiction by Theorem 2.12. We thusly conclude that $\mathbb{A}\neq O_n$, so $\mathbb{A}\in O_n$ again by Theorem 2.12. Now suppose that $\mathbb{B}=O_n$; then $\mathbb{A}\in\mathbb{B}$ and we are satisfied. \\  Now suppose that $\mathbb{B}\neq O_n$, so $\mathbb{B}$ is a set, and thusly by Theorem 2.11 we have that $\mathbb{A}\in\mathbb{B}$ or $\mathbb{A}=\mathbb{B}$ or $\mathbb{B}\in\mathbb{A}$. If $\mathbb{A}=\mathbb{B}$ we have a contradiction since $\mathbb{A}\subset\mathbb{B}$ by assumption. If $\mathbb{B}\in\mathbb{A}$ then $\mathbb{A}\subset\mathbb{B}\Rightarrow\mathbb{B}\in\mathbb{B}$, a contradiction by Corollary 1.13. Thusly we conclude that $\mathbb{A}\in\mathbb{B}$, completing half of the proof. \\ Now suppose that $\mathbb{A}\in\mathbb{B}$, so $x\in \mathbb{A}\in\mathbb{B}\Rightarrow x\in\mathbb{B}$ since $\mathbb{B}$ is an ordinal. Since $x$ was arbitrary in $\mathbb{A}$, we have\ $\mathbb{A}\subseteq\mathbb{B}$. If $\mathbb{A}=\mathbb{B}$ then $\mathbb{B}\in\mathbb{B}$, a contradiction by Corollary 1.13.  Thusly we conclude that $\mathbb{A}\subset\mathbb{B}$, completing the proof.\\
\end{proof}

\begin{customthm}{2.18}
If $\mathbb{A}$ is a non-empty class of ordinals, then $\bigcap\mathbb{A}$ is an ordinal and $\bigcap\mathbb{A}\in\mathbb{A}$.
\end{customthm}
\begin{proof}
Let $\mathbb{A}\subseteq O_n$ with $\mathbb{A}\neq 0$. If $\bigcap\mathbb{A}=0$ then $\bigcap\mathbb{A}$ is an ordinal by Theorem 2.3 and $\mathbb{A}\neq0\Rightarrow\bigcap\mathbb{A}\in\mathbb{A}$ by Theorem 2.11, thus we are satisfied. Suppose $\bigcap\mathbb{A}\neq0$, and let $x\in y\in \bigcap\mathbb{A}$, so $y\in a$ for all $a\in\mathbb{A}$. Let $z\in\mathbb{A}$, so $z$ is an ordinal and $y\in z$, thus $x\in y\in z\Rightarrow x\in z$, and since $z$ was arbitrary in $\mathbb{A}$ we have $x\in \bigcap\mathbb{A}$, thus $\mathbb{A}$ is membership transitive. Further, by Theorem 2.8, $y\in z\Rightarrow y$ is an ordinal, and consequently $x\in y\Rightarrow x$ is an ordinal, thus $x$ is membership transitive. Since $y$ was arbitrary in $\bigcap\mathbb{A}$ and $x$ was arbitrary in $y$, this completes the proof that $\bigcap\mathbb{A}$ is an ordinal. \\
By definition we have that $x\in\mathbb{A}\Rightarrow\bigcap\mathbb{A}\subseteq x$. If $\bigcap\mathbb{A}=x$ then $\bigcap\mathbb{A}\in\mathbb{A}$ and we are satisfied. If $\bigcap\mathbb{A}\subset x$, then since $\bigcap\mathbb{A}$ and $x$ are ordinals $\bigcap\mathbb{A}\in x$ by Theorem 2.15, and since $x$ was arbitrary in $\mathbb{A}$ we have that $\bigcap\mathbb{A}\in\bigcap\mathbb{A}$, a contradiction by Corollary 1.13. Thusly we conclude that $\bigcap\mathbb{A}\in\mathbb{A}$, completing the proof. \\
\end{proof}

Throughout the remainder of this section, we will state theorems on the properties of ordinals that are known to be true without proofs included -- for a proof of any of the following theorems in this section, see D. Monk [1]. \\

\begin{customthm}{2.19}
Let $\mathbb{A}$ and $\mathbb{B}$ be ordinals. Then
\begin{enumerate}
\item If $\mathbb{A}\in\mathbb{B}$, then $\mathcal{S}\mathbb{A}=\mathbb{B}$ or $\mathcal{S}\mathbb{A}\in\mathbb{B}$.
\item If $\mathbb{C}\subseteq\mathbb{A}$, then $\bigcup\mathbb{C}=\mathbb{A}$ or $\bigcup\mathbb{C}\in\mathbb{A}$. \\
\end{enumerate}
\end{customthm}

We now collect together the remaining necessary symbolic definitions regarding relations we need prior to our discussion of ordering and ordinal arithmetic. \\

\begin{customdefn}{2.20}
Let $\mathbb{R}$ be a relation.
\begin{enumerate}
\item We write $x\mathbb{R}y$ iff $(x,y)\in\mathbb{R}$. Similarly, we write $x\mathbb{R}y\mathbb{R}z$ iff $(x,y)\in\mathbb{R}$ and $(y,z)\in\mathbb{R}$.
\item $\mathbb{R}^{-1}=\{x:x=(a,b)\wedge b\mathbb{R}a\}$. $\mathbb{R}^{-1}$ is called the {\bf inverse} of $\mathbb{R}$.
\item $\mathbb{R}\upharpoonleft\mathbb{X}=\{x:\exists(a,b)\big[x=(a,b)\wedge a\mathbb{R}b\wedge a\in\mathbb{X}\big]\}$. $\mathbb{R}\upharpoonleft\mathbb{X}$ is called the {\bf domain restriction of $\mathbb{R}$ to $\mathbb{X}$}.
\item $\mathbb{R}\upharpoonright\mathbb{Y}=\{x:\exists(a,b)\big[x=(a,b)\wedge a\mathbb{R}b\wedge b\in\mathbb{Y}\big]\}$. $\mathbb{R}\upharpoonright\mathbb{Y}$ is called the {\bf range restriction of $\mathbb{R}$ to $\mathbb{Y}$}.
\item $\mathbb{R}\upharpoonleft\upharpoonright\mathbb{Z}=\{x:\exists(a,b)\big[x=(a,b)\wedge a\mathbb{R}b\wedge a\in\mathbb{Z}\wedge b\in\mathbb{Z}\big]\}$. $\mathbb{R}\upharpoonleft\upharpoonright\mathbb{Z}$ is called the {\bf restriction of $\mathbb{R}$ to $\mathbb{Z}$}.
\item $^\mathbb{A}\mathbb{B}=\{f:f\ \text{is a function} \ \wedge \ dmnf=\mathbb{A} \ \wedge \ rngf\subseteq\mathbb{B}\}$. $^\mathbb{A}\mathbb{B}$ is called the {\bf function space from $\mathbb{A}$ into $\mathbb{B}$}.
\item Let $\mathbb{A}$ be a relation. $\mathbb{R}|\mathbb{A}=\{x:\exists a\exists b \exists c\big[a\mathbb{R}b \wedge b\mathbb{A}c\wedge x=(a,c)\big]\}$. $\mathbb{R}|\mathbb{A}$ is called the {\bf relative product of $\mathbb{R}$ with respect to $\mathbb{A}$}. If $(x,z)\in\mathbb{R}|\mathbb{A}$ we will write $x\mathbb{R}\mathbb{A}z$. Note that, in general, $\mathbb{R}|\mathbb{A}\neq(\mathbb{A}|\mathbb{R})^{-1}$.
\item $\mathbb{R}\circ\mathbb{A}=\mathbb{A}|\mathbb{R}$. $\mathbb{R}\circ\mathbb{A}$ is called the {\bf composition of $\mathbb{R}$ with $\mathbb{A}$.}
\item $Fld\mathbb{R}=dmn\mathbb{R}\cup rng\mathbb{R}$. $Fld\mathbb{R}$ is called the {\bf field} of $\mathbb{R}$. $\mathbb{R}$ is said to be {\bf on $\mathbb{A}$} iff $Fld\mathbb{R}=\mathbb{A}$. \\
\end{enumerate}
\end{customdefn}

We will not discuss any of the above definitions in ther own right beyond this point, however the notions introduced are hopefully familiar and intuitive. \\

\begin{customdefn}{2.21} Let $\mathbb{R}$ be a relation.
\begin{enumerate}
\item $\mathbb{R}$ is {\bf symmetric} $\iff \forall x \forall y\big[x\mathbb{R}y\Rightarrow y\mathbb{R}x\big]$.
\item $\mathbb{R}$ is {\bf transitive} $\iff \forall x\forall y\forall z\big[x\mathbb{R}y\mathbb{R}z\Rightarrow x\mathbb{R}z\big]$.
\item $\mathbb{R}$ is an {\bf equivalence relation} $\iff \mathbb{R}$ is transitive and symmetric.
\item $x/\mathbb{R}=\{y:x\mathbb{R}y\}.$  If $\mathbb{R}$ is an equivalence relation and $x\in Fld{\mathbb{R}}$, then $x/\mathbb{R}$ is called an {\bf equivalence class}, and $x$ is a {\bf representative} of $x/\mathbb{R}$.
\item $\mathbb{R}$ is {\bf reflexive on $\mathbb{A}$} $\iff \forall x\big[x\in\mathbb{A}\Rightarrow x\mathbb{R}x\big]$.
\item $\mathbb{R}$ is {\bf antisymmetric} $\iff \forall x\forall y\big[x\mathbb{R}y\mathbb{R}x\Rightarrow x=y\big]$.
\item $\mathbb{R}$ is a {\bf partial ordering} $\iff \mathbb{R}$ is reflexive on $Fld\mathbb{R}$ and $\mathbb{R}$ is transitive and $\mathbb{R}$ is antisymmetric.
\item $\mathbb{A}$ is {\bf partially ordered by $\mathbb{R}$} $\iff (\mathbb{A}\times\mathbb{A})\cap\mathbb{R}=\mathbb{R}'$ is a partial ordering and $Fld\mathbb{R}'=\mathbb{A}$.
\item $\mathbb{R}$ is a {\bf total ordering} $\iff \mathbb{R}$ is a partial ordering and \\ $\forall x\forall y\big[\big(x\in Fld\mathbb{R} \wedge y\in Fld\mathbb{R}\big)\Rightarrow \big(x\mathbb{R}y\vee y\mathbb{R}x\big)\big]$.
\item $\mathbb{R}$ is {\bf well-founded} $\iff \mathbb{A}\subseteq\mathbb{R}$ and $\mathbb{A}\neq 0$ imply that $\exists x\big(\exists y\big[ (x,y)\in\mathbb{A} \wedge \nexists z\big[(z,x)\in\mathbb{A}\big]\big]\big)$.
\item $\mathbb{R}$ is a {\bf well-ordering} $\iff \mathbb{R}$ is a total ordering and $\mathbb{R}$ is well-founded. \\
\end{enumerate}
\end{customdefn}

Although it may seem odd that we did not require an equivalence relation to be reflexive, note that if $\mathbb{R}$ is an equivalence relation on $\mathbb{A}$ then $\mathbb{R}$ is reflexive on $\mathbb{A}$ according to Definition 2.20, thus our intuition regarding equivalence relations is satisfied. We have now provided sufficient context to define a natural well-ordering on $O_n$ -- we first briefly expand our symbolic ability to define classes full of certain sets. \\
 
\begin{customdefn}{2.22}
\begin{enumerate}
\item Any capital hollow letter is a {\bf class-theoretical term}, as are $\mathbb{V}$ and $0$.
\item If $\sigma$ and $\tau$ are class-theoretical terms, then so are $\mathcal{S}\sigma,\{\sigma\},\{\sigma,\tau\},(\sigma,\tau),dmn\sigma,rng\sigma,Fld\sigma, \sigma\upharpoonright\tau, \\ \sigma|\tau,\sigma^{-1},\sigma\cap\tau,\sigma\cup\tau,\sigma\sim\tau,\sigma\times\tau,\sigma ^* \tau,\sigma(\tau),\sigma\tau,\sigma_{\tau},^{\tau}\sigma,$ and $\sigma\circ\tau$.
\item If $\sigma(\mathbb{X})$ is a class-theoretical term and $\phi(\mathbb{X})$ is a class-theoretical formula, then $$\{\sigma(\mathbb{X}):\phi(\mathbb{X})\}=\{x: \exists y\big[\phi(y) \wedge x=\sigma(y)\big]\}.$$ Similarly, we use the notations $\{\sigma(\mathbb{X},\mathbb{Y}):\phi(\mathbb{X},\mathbb{Y})\}$, etc.
\item $<\sigma(\mathbb{X}):\mathbb{X}\in \mathbb{I}>$ denotes the function $\{(\mathbb{X},\sigma(\mathbb{X})):\mathbb{X}\in \mathbb{I}\}$ for any term $\sigma(\mathbb{X})$ such that $\sigma(\mathbb{X})$ is a set for all $\mathbb{X}\in\mathbb{I}$. \\
\end{enumerate}
\end{customdefn}

To illustrate the utility of this notational expansion, we restate three of the more apparently burdensome definitions from Definition $2.20$. \\

\begin{customdefn}{2.20'}
\begin{enumerate}
\item $\mathbb{R}^{-1}=\{(a,b):b\mathbb{R}a\}$. $\mathbb{R}^{-1}$ is called the {\bf inverse} of $\mathbb{R}$.
\item $\mathbb{R}\upharpoonleft\mathbb{X}=\{(a,b): a\mathbb{R}b \wedge a\in\mathbb{X}\big]\}$. $\mathbb{R}\upharpoonright\mathbb{X}$ is called the {\bf domain restriction of $\mathbb{R}$ to $\mathbb{X}$}.
\item $\mathbb{R}|\mathbb{A}=\{(a,c):\exists b\big(a\mathbb{R}b\wedge b\mathbb{A}c\big)\}$. $\mathbb{R}|\mathbb{A}$ is called the {\bf relative product of $\mathbb{R}$ with respect to $\mathbb{A}$}. \\
\end{enumerate}
\end{customdefn}

Thusly, the reader can hopefully see that this expansion is quite useful when trying to express definitions concisely. \\

\begin{customdefn}{2.23}
\begin{enumerate}
\item Lowercase greek letters $\alpha,\beta,\gamma,\dots$ will be used to denote ordinals that are sets unless otherwise specified.
\item $\leq = \{(\alpha,\beta):\alpha\in\beta\vee\alpha=\beta\}.$ We write $\alpha\leq\beta$ iff $(\alpha,\beta)\in\leq$.
\item $< = \{(\alpha,\beta):\alpha\leq\beta \wedge \alpha\neq\beta\}.$ We write $\alpha<\beta$ iff $(\alpha,\beta)\in <$. \\
\end{enumerate}
\end{customdefn}

Thus we see that $\alpha<\beta$ is logically equivalent to $\alpha\in\beta$, and this will be tacitly assumed in much of what follows. We now present all of the remaining facts we will use regarding ordinals. \\

\begin{customthm}{2.24}
\begin{enumerate}
$\leq$ is a well-ordering with field $O_n$. Further:
\item For any non-empty class $\mathbb{A}\subseteq O_n$, $\bigcap\mathbb{A}$ is the $\leq$-least member of $\mathbb{A}$.
\item $0$ is the $\leq$-least member of $O_n$.
\item $\alpha=\{\beta:\beta<\alpha\}$.
\item $\alpha<\beta \iff \mathcal{S}\alpha\leq\beta$.
\item $\alpha<\mathcal{S}\beta \iff \alpha\leq\beta$. 
\item $\nexists \beta\big(\alpha<\beta<\mathcal{S}\alpha\big)$.
\item $\alpha\leq\beta \iff \alpha\subseteq\beta$.
\item $0\neq\mathbb{A}\subset O_n \Rightarrow \bigcup\mathbb{A}$ is the $\leq$-least upper bound of $\mathbb{A}$.
\item $\mathcal{S}\alpha=\mathcal{S}\beta \iff \alpha=\beta$.
\item $\mathcal{S}\alpha<\mathcal{S}\beta \iff \alpha<\beta$.
\item $\alpha=\bigcup\mathcal{S}\alpha$.
\item $\alpha=\mathcal{S}\bigcup\alpha$ or $\alpha=\bigcup\alpha$. \\
\end{enumerate}
\end{customthm}

This concludes our discussion of the basic properties of ordinals; the material presented above will be ubiquitous throughout the remainder of the paper. \\

\vspace{30mm}

\section{Transfinite Induction}

Here we provide the standard theorems concerning induction over ordinal classes, within the context of MK class theory. The outline for the following material follows closely to the exposition given in D. Monk [1]. \\

\begin{customthm}{3.1 -- Generalized Principle of Transfinite Induction}
If $\mathbb{A}$ is an ordinal and $\mathbb{X}$ is a class such that for every $\alpha\in\mathbb{A}$, $$\forall\beta\big(\beta<\alpha\implies\beta\in\mathbb{X}\big)\implies\alpha\in\mathbb{X},$$ then $\mathbb{A}\subseteq\mathbb{X}$.
\end{customthm}
\begin{proof}
Suppose $\mathbb{A}$ is an ordinal, and $\mathbb{X}$ has the above property, and suppose that $\mathbb{A}\nsubseteq\mathbb{X}$.  Then $\mathbb{A}\sim\mathbb{X}\neq0$; let $\alpha$ be the $\leq$-least member of $\mathbb{A}\sim\mathbb{X}$. We then have that $$\forall\beta\big(\beta<\alpha\implies\beta\in\mathbb{X}),$$ thus $\alpha\in\mathbb{X}$ by assumption, a contradiction. We thusly conclude that $\mathbb{A}\subseteq\mathbb{X}$, completing the proof. \\
\end{proof}

This deceptively simple theorem is actually the most general and powerful induction principle over ordinal classes in this context -- induction over $\omega$ or larger transfinite ordinal classes will all follow as a consequence of Theorem {\it 3.1}. We now expand our vocabulary to delineate between two distinct types of ordinals. \\

\begin{customdefn}{3.2} \
\begin{enumerate}
\item $\alpha$ is a {\bf successor ordinal} iff there exists some $\beta$ such that $\alpha=\mathcal{S}\beta$. We will then write $\alpha=\beta-1$.
\item $\alpha$ is a {\bf limit ordinal} iff $\alpha$ is not a successor ordinal. If $\alpha\neq0$, we may refer to $\alpha$ as a {\bf non-zero limit ordinal}. \\
\end{enumerate}
\end{customdefn}

Note that, by definition, every ordinal is either a succcessor ordinal or a limit ordinal. Although it may not immedately be clear to the reader that there are any limit ordinals besides $0$, which is a limit ordinal vacuously, we will prove very shortly that a profusion of them do indeed exist -- this is essentially a consequence of Theorem {\it 2.6}. First, we provide one refinement of Theorem {\it 3.1} using the language of successor and limit ordinals. \\

\begin{customthm}{3.3} \
\begin{enumerate}
\item $\alpha$ is a successor ordinal iff $\bigcup\alpha<\alpha$.
\item $\alpha$ is a limit ordinal iff $\bigcup\alpha=\alpha$.
\item $\alpha$ is a limit ordinal iff $\forall\beta\big[\beta<\alpha\implies\exists\gamma\big(\beta<\gamma<\alpha)\big]$.
\end{enumerate}
\end{customthm}
\begin{proof}
For {\it (1)}, suppose $\alpha$ us a successor ordinal, so $\alpha=\mathcal{S}\beta$ for some $\beta$. We then observe that $$\bigcup\alpha=\bigcup\mathcal{S}\beta=\beta$$ by Theorem {\it 2.24 (11)}, thus $$\bigcup\alpha=\beta<\mathcal{S}\beta=\alpha,$$ completing the proof of {\it (1)} one way. For the other direction, suppose $\bigcup\alpha<\alpha$. Then by Theorem {\it 2.24 (12)} we have that $$\alpha\neq\bigcup\alpha\implies\alpha=\mathcal{S}\bigcup\alpha,$$ thus $\alpha$ is a successor ordinal. This completes the proof of {\it (1)}. \\ For {\it (2)}, suppose $\alpha$ is a limit ordinal. Then by Theorem {\it 2.24 (12)}, $$\nexists\beta\big(\alpha=\mathcal{S}\beta\big)\implies\alpha\neq\mathcal{S}\bigcup\alpha\implies\alpha=\bigcup\alpha,$$ completing the proof of {\it (2)} one way. For the other direction, suppose $\alpha=\bigcup\alpha$, and suppose that $\alpha=\mathcal{S}\beta$ for some $\beta$. Then we have that $$\bigcup\mathcal{S}\beta=\bigcup\alpha=\alpha=\mathcal{S}\beta,$$ but by Theorem {\it 2.24 (11)} we also have that $\bigcup\mathcal{S}\beta=\beta$, thus $\beta=\mathcal{S}\beta\implies\beta\in\beta,$ a contradiction by Corollary {\it 1.13}. We thusly conclude that there is no $\beta$ such that $\alpha=\mathcal{S}\beta$, thus $\alpha$ is a limit ordinal.  This completes the proof of {\it (2)}. \\ For {\it (3)}, suppose again that $\alpha$ is a limit ordinal. If there exists a $\beta<\alpha$ such that there is no $\gamma$ satisfying $\beta<\gamma<\alpha$, then by Theorem 2.24 {\it (6)} we have that $$\beta<\alpha\leq\mathcal{S}\beta\implies\alpha=\mathcal{S}\beta,$$ a contradiction since $\alpha$ is a limit ordinal; this completes the proof one way. For the other direction, suppose that for all $\beta<\alpha$ there exists a $\gamma$ such that $\beta<\gamma<\alpha$, and that $\alpha=\mathcal{S}\beta$ for some $\beta$. We then have that $$\beta<\alpha\implies\exists\gamma\big(\beta<\gamma<\alpha\big)\implies\beta<\gamma<\mathcal{S}\beta,$$ a contradiction by Theorem 2.24 {\it (6)} again. We thusly conclude that there is no $\beta$ such that $\alpha=\mathcal{S}\beta$, thus $\alpha$ is a limit ordinal. This completes the proof of {\it (3)}, completing the proof overall. \\
\end{proof}

We thusly see that being a successor ordinal or limit ordinal can be recast in terms of equality with infinitary unions and successor functions, and that any limit ordinal will always have an infinite number of other ordinals 'between' itself and any of its members under $\leq$. We now provide the aforementioned refinement of Theorem {\it 3.1}. \\

\begin{customthm}{3.4}
Suppose $\mathbb{A}$ is an ordinal, and $\mathbb{X}$ is a class with the following properties:
\begin{enumerate}
\item $0\in\mathbb{X}$.
\item For all $\alpha\in\mathbb{A}$, if $\alpha\in\mathbb{X}$ and $\mathcal{S}\alpha\in\mathbb{A}$, then $\mathcal{S}a\in\mathbb{X}$.
\item For all $\alpha\in\mathbb{A}$, if $\alpha$ is a limit ordinal and $\beta\in\mathbb{X}$ for all $\beta<\alpha$, then $\alpha\in\mathbb{X}$.
\end{enumerate}
Then $\mathbb{A}\subseteq\mathbb{X}$.
\end{customthm}
\begin{proof}
Let $\mathbb{A}$ be an ordinal, and $\mathbb{X}$ be a class with the properties listed above. By Definition {\it 3.2}, for any $\alpha\in\mathbb{A}$ we have two possible options:
\begin{enumerate}
\item $\alpha$ is a successor ordinal.
\item $\alpha$ is a limit ordinal.
\end{enumerate}
We will use Theorem {\it 3.1}; suppose that $\alpha\in\mathbb{A}$ and $\beta\in\mathbb{X}$ for every $\beta<\alpha$. If $\alpha$ is a successor ordinal, then $\alpha=\mathcal{S}\gamma$ for some $\gamma$, and $$\gamma\in\alpha\in\mathbb{A}\implies\gamma\in\mathbb{A},$$ since $\mathbb{A}$ is an ordinal and consequently membership transitive. Further, we have that $\gamma<\alpha\implies\gamma\in\mathbb{X}$ by assumption, thus $\mathcal{S}\gamma=\alpha\in\mathbb{X}$ by property {\it (2)} above. If $\alpha$ is a limit ordinal and $\alpha=0$ then $\alpha\in\mathbb{X}$ by {\it (1)}; if $\alpha$ is a non-zero limit ordinal, then $\alpha\in\mathbb{X}$ by property {\it (3)} above and our assumption. Consequently, $\mathbb{X}$ is a class satisfying the conditions of Theorem {\it 3.1} for $\mathbb{A}$, thus $\mathbb{A}\subseteq\mathbb{X}$. This completes the proof.  \\
\end{proof}

\section{Finite and Infinite Ordinals}

Here we provide our definition of a finite vs infinite ordinal, along with some brief discussion of the meaning behind the delineation. The first infinite ordinal we will introduce in this section is $\omega$, the countable infinity.  In many ways, the countable infinity serves as the fundamental transfinite object of interest -- many fascinating properties of transfinite ordinal numbers can be understood in terms of the countable infinity, particularly in combination with a powerful axiom of choice.  \\

\begin{customdefn}{4.1}
$\omega = \bigcap\{\alpha:0\in\alpha\wedge\forall\beta\big(\beta\in\alpha\Rightarrow\mathcal{S}\beta\in\alpha\big)\}$. $\omega$ is called the {\bf countable infinity}, or the {\bf set of all natural numbers}. Members of $\omega$ are called {\bf natural numbers}, and natural numbers will be written using lowercase italic letters $l,m,n,i$ unless otherwise specified.\\
\end{customdefn}

\begin{customthm}{4.2}
$\omega$ is the smallest non-zero limit ordinal.
\end{customthm}
\begin{proof}
That $\omega$ is a set follows immedately from the infinity axiom. Let $\mathbb{A}=\{\alpha:\alpha\in\omega\wedge\alpha\subseteq\omega\}.$  We then observe that $0\in\mathbb{A}$, and if $\beta\in\mathbb{A}$ then $\beta\in\omega$ and consequently $\mathcal{S}\beta\in\omega$, and $\mathcal{S}\beta=\beta\cup\{\beta\}\subseteq\omega$, thus $\mathcal{S}\beta\in\mathbb{A}$. We thusly have that $\omega\subseteq\mathbb{A}$, since $\mathbb{A}$ is a class satisfying the class-theoretical formula involved in Definition {\it 4.1} and $\omega$ is the intersection of all such classes. Every member of $\omega$ is an ordinal by definition; suppose $\alpha\in\beta\in\omega$. Then since $\beta\subseteq\omega$ by virtue of the fact that $\omega\subseteq\mathbb{A}$, we have $\alpha\in\omega$, thus $\omega$ is also membership transitive and consequently an ordinal. \\ Suppose that $\omega$ is a succeessor ordinal, so $\omega=\mathcal{S}\zeta$ for some $\zeta\in O_n$. We then have that $\zeta\in\omega\implies\mathcal{S}\zeta\in\omega\implies\omega\in\omega$, a contradiction by Corollary {\it 1.16}. We thusly conclude that $\omega$ is not a successor ordinal, so $\omega$ is a limit ordinal. \\ Suppose that $\gamma$ is also a non-zero limit ordinal, so $0\in\gamma$ by Lemma 2.10. Suppose that $\alpha\in\gamma$; we then have that $\alpha<\mathcal{S}\alpha\leq\gamma$ and since $\gamma$ is a limit ordinal by definition we conclude that $\mathcal{S}\alpha<\gamma$. Consequently $\gamma$ is another set satisfying the class-theoretical formula involved in Definition {\it 4.1}, thus $\omega\subseteq\gamma\implies\omega\leq\gamma$ by Theorem {\it 2.17}, thus $\omega$ is the smallest limit ordinal. This completes the proof. \\
\end{proof}

\begin{customdefn}{4.3} \
\begin{enumerate}
\item $\alpha$ is {\bf finite} iff $\alpha<\omega$.
\item $\gamma$ is {\bf infinite} iff $\omega\leq\gamma$. We may also refer to $\gamma$ as {\bf transfinite}. \\
\end{enumerate}
\end{customdefn}

$\omega$ is the point in $O_n$ at which counter-intuitive behaviour begins to really manifest itself. For any ordinal smaller than $\omega$, taking a successor produces a set that actually contains more members in the sense that there is no injective, surjective function between any finite ordinal and its successor. This stops being true at $\omega$, since $\mathbb{F}=\langle\mathcal{S}i:i\in\omega\rangle\cup{(0,\omega)}$ is an injective, surjective function mapping $\mathcal{S}\omega$ onto $\omega$, and in general for any $\gamma\geq\omega$ we recall that $\mathbb{I}=\{(x,x):x\in\mathbb{V}\}$ and observe that $$\mathbb{H}=\langle \mathcal{S}i:i\in\omega\rangle\cup\big[\mathbb{I}\upharpoonright(\gamma\sim\omega)\big]\cup\{(\gamma,0)\}$$ is surjective and injective, and maps $\mathcal{S}\gamma$ onto $\gamma$ -- the profusion of transfinite sets which are equipotent with $\omega$ is the phenomenon known as being {\it countable}, and it is a consequence of this counter-intuitive behaviour. Note that for any finite ordinal $n$, if we redefine $\mathbb{H}$ to have domain $n$ we have that $\mathbb{H}(n)=n$, so this function no longer maps successors onto their predecessors in a finite setting. We now briefly define the notion of a Cardinal number, mainly for the sake of discussion. \\

\begin{customdefn}{4.4}
For any ordinal $\alpha$, 
\begin{enumerate}
\item $\alpha$ is a {\bf Cardinal number} iff there is no one-one function between $\alpha$ and any $\beta\in\alpha$. 
\item $Card=\{\alpha:\alpha\ \text{is a Cardinal number}\}$.
\end{enumerate}
\end{customdefn}

It is well known that all natural numbers are cardinal numbers, $\omega$ is a cardinal number, and $Card$ is a proper class. After $\omega$ however, we do not encounter another Cardinal number for a {\it long} time -- we denote this next Cardinal number by $\omega^+$. This is really at the heart of much apparently strange behaviour in $O_n$. What this phenomenon tells us is fascinating, despite its strange manifestation -- for any finite set, there is only one way (up to isomorphism) to well-order a set having that number of elements.  Once we have an infinite number of elements in our class however, there are actually many non-isomorphic ways to well-order such a class without adding any additional elements -- larger order-types allow for more complicated structures, however. A good example of this is the fact that $\mathbb{C}$ is isomorphic to $\mathbb{R}^2$ with some additional structure defined and consequently has the same number of members as $\mathbb{R}$ (namely $\omega^+$ members, assuming the Generalized Continuum Hypothesis), however any one-one mapping between $\mathbb{C}$ and $\mathbb{R}$ will necessarily destroy some of that structure. This is because $\mathbb{C}$ has a much larger {\it order-type} than $\mathbb{R}$, despite having the same number of elements -- we have defined much more structure among the elements of $\mathbb{C}$ than we have amongst those of $\mathbb{R}$. \\  Thusly, we see that the delineation between 'finite' and 'infinite' can be understood in terms of the ordering behaviour of the classes involved -- finite classes can only be well ordered in one way up to isomorphism once you know their cardinality, however transfinite classes may be well-ordered in a profusion of non-isomorphic ways that all have the same cardinality. This is why the finite cardinals coincide completely with the finite ordinals, however there are many transfinite ordinals between each transfinite cardinal number -- the ordinals between each pair of transfinite cardinal numbers each represent a non-isomorphic way to well-order a set with the smaller of the two transfinite cardinalities. \\

\section{Recursion}

Here we present the General Recursion Principle -- it is the strongest recursion principle in the context of MK class theory, allowing for recursion over well-founded relations.  This will allow for recursion over $O_n$ since $\leq$ is a total ordering on $O_n$, and for recursion over $O_n^n$ since it can be well ordered lexicographically -- we will use the former to define ordinal arithmetic, and we will use recursion over $O_n\times O_n=O_n^2$ to define the Transfinite Hyperoperation sequence. \\

\begin{customthm}{5.1 -- General Recursion Principle}
Let $\mathbb{R}$ be a well-founded relation such that for all $x\in Fld\mathbb{R}$, $\{y:y\mathbb{R}x\}$ is a set, and let $\mathbb{F}$ be a function with $dmn\mathbb{F}=Fld\mathbb{R}\times\mathbb{V}$. Then there is a unique function $\mathbb{G}$ such that $dmn\mathbb{G}=Fld\mathbb{R}$ and for all $x\in Fld\mathbb{R}$, $$\mathbb{G}(x)=\mathbb{F}(x,\mathbb{G}\upharpoonright\{y:y\mathbb{R}x\}).$$ \\
\end{customthm}

A proof of this theorem occupies roughly two pages in D. Monk [1], and will not be reproduced here as it is rather tangential to the main purpose of this paper. Despite this, all definitions for recursive ordinal arithmetic in the next section will depend explicitly on this theorem -- if the reader is not familiar with recursion, it may well be worth taking a moment to consider how Theorem {\it 5.1} applies to the definitions in the next section. \\

\vspace{50mm}

\section{Recursive Ordinal Arithmetic}

Here we present the standard recursive definitions for addition, multiplication and exponentiation on the ordinals that emulate the behavior of the intuitive addition, multplication and exponentiation defined on the natural numbers. All of the following theorems and definitions are justified by the General Recursion Principle. Note that although we will note prove it here, all of these operations are associative; for more information on recursive ordinal arithmetic in the context of MK class theory, see D. Monk [1] or W. Sierpinski [2]. \\

\begin{customthm}{6.1}
There exists a unique function $\dot+$ such that $dmn\dot+=(O_n\times O_n)$ and $rng\dot+=O_n$, and
\begin{enumerate}
\item $\dot+(\alpha,0)=\alpha$,
\item $\dot+(\alpha,\mathcal{S}\beta)=\mathcal{S}\dot+(\alpha,\beta)$,
\item $\dot+(\alpha,\gamma)=\bigcup_{\delta<\gamma}\dot+(\alpha,\delta)$ if $0\neq\gamma=\bigcup\gamma$. \\
\end{enumerate}
\end{customthm}

\begin{customdefn}{6.2}
$\dot+$ is the unique function asserted to exist in Theorem 6.1; it is called {\bf recursive addition}. Further, we will write $\alpha\dot+\beta=\gamma$ iff $\dot+(\alpha,\beta)=\gamma$. $\alpha\dot+\beta$ is called the {\bf recursive sum of $\alpha$ and $\beta$}. \\
\end{customdefn}

\begin{customthm}{6.3}
For any ordinals $\alpha,\beta$ such that $\alpha<\beta$, there exists a unique $\gamma$ such that $\alpha\dot+\gamma=\beta$. \\
\end{customthm}

We include a dot over the recursive addition symbol to delineate it from the natural addition symbol $+$ we will define shortly, which will actually play the role of addition for us; $\dot+$ is non-commutative for most ordinals $\alpha$ such that $\omega\leq\alpha$. A similar comment applies to recursive multiplication. Note that we now have satisfying algebraic relationships like $\mathcal{S}\alpha=\alpha\dot+1$, $\mathcal{S}\mathcal{S}\alpha=\alpha\dot+2,\dots$. \\

\begin{customthm}{6.4}
There exists a unique function $\dot\times$ such that $dmn\dot\times=(O_n\times O_n)$ and $rng\dot\times=O_n$, and
\begin{enumerate}
\item $\dot\times(\alpha,0)=0$,
\item $\dot\times(\alpha,\mathcal{S}\beta)=\dot\times(\alpha,\beta)\dot+\alpha$,
\item $\dot\times(\alpha,\gamma)=\bigcup_{\delta<\gamma}\dot\times(\alpha,\delta)$ if $0\neq\gamma=\bigcup\gamma$. \\
\end{enumerate}
\end{customthm}

\begin{customdefn}{6.5}
$\dot\times$ is the unique function asserted to exist in Theorem 6.3; it is called {\bf recursive multiplication}. Further, we will write $\alpha\dot\times\beta=\gamma$ iff $\dot\times(\alpha,\beta)=\gamma$. $\alpha\dot\times\beta$ is called the {\bf recursive product of $\alpha$ and $\beta$}.\\
\end{customdefn}

\begin{customthm}{6.6}
There exists a unique function $\uparrow$ such that $dmn\uparrow=(O_n\times O_n)$ and $rng\uparrow=O_n$, and
\begin{enumerate}
\item $\uparrow(\alpha,0)=1$,
\item $\uparrow(\alpha,\mathcal{S}\beta)=\uparrow(\alpha,\beta)\dot\times\alpha$,
\item $\uparrow(\alpha,\gamma)=\bigcup_{\delta<\gamma}\uparrow(\alpha,\delta)$ if $0\neq\gamma=\bigcup\gamma$. \\
\end{enumerate}
\end{customthm}

\begin{customdefn}{6.7}
$\uparrow$ is the unique function asserted to exist in Theorem 6.5; it is called {\bf recursive exponentiation}. Further, we will write $\alpha^{\beta}=\gamma$ iff $\uparrow(\alpha,\beta)=\gamma$. \\
\end{customdefn}

Note that $\uparrow$ has no dot; it will be the only exponentiation we use. We now provide a theorem that will allow us to express any ordinal in a particular 'base' that we wish to choose. \\

\begin{customthm}{6.8}
If $0<\alpha$ and $1<\beta$, then there exist unique $\gamma,\delta,\varepsilon$ such that 
\begin{enumerate}
\item $\alpha=\beta^\gamma\dot\times\delta\dot+\varepsilon$,
\item $\gamma<\alpha$,
\item $0<\delta<\beta$,
\item $\varepsilon<\beta^\gamma$. \\
\end{enumerate}
\end{customthm}

In order to formulate our commutative addition structure $+$, dubbed natural addition after G. Hessenberg (1906), we must first formulate the notion of a recursive finite sum of ordinals. \\

\begin{customdefn}{6.9}
If $\mu\in^mO_n$, we define 
\begin{enumerate}
\item $\dot\Upsigma(\mu,0)=0$,
\item $\dot\Upsigma(\mu,n\dot+1)=\begin{cases} 
      \dot\Upsigma(\mu,n)\dot+\mu_n, & n<m \\
      0 ,& m\leq n. \\
   	  \end{cases}$ 
\end{enumerate}
We write $\dot\Upsigma_{_{i<n}}\mu_i$ instead of $\dot\Upsigma(\mu,n)$. \\
\end{customdefn}

\begin{customthm}{6.10 -- Associative Law for $\dot\Upsigma$}
If $\mu\in^mO_n$, then $$\dot\Upsigma_{_{i<m}}\mu_i = \dot\Upsigma_{_{i<n}}\mu_i \dot+ \dot\Upsigma_{_{n<i<m}}\mu_i.$$ \\
\end{customthm}

\begin{customthm}{6.11 -- Base Expansion Theorem}
Let $1<\beta$. Then for any ordinal $\alpha$ there exist unique $m,\zeta,\delta$ satisfying the following conditions.
\begin{enumerate}
\item $\zeta,\delta\in^mO_n$.
\item $\alpha=\dot\Upsigma_{_{i<m}}\beta^{\zeta_i}\dot\times\delta_i$.
\item For all $i<m$, $\zeta_i\leq\alpha$.
\item If $i\dot+1<m$, then $\zeta_{i\dot+1}<\zeta_i$.
\item For all $i<m$, $0<\delta_i<\beta$. \\
\end{enumerate}
\end{customthm}

Choosing $\beta=10$ leads to the usual school form base-10 representation of any finite ordinal. Of particular interest to us is the choice $\beta=\omega$, specifically because this together with condition {\it 5} forces all of the 'coefficients' $\delta_i$ in the expansion of any ordinal to be {\it natural numbers}.  We call this base-$\omega$ expansion the {\it Cantor Normal form} of an ordinal, and make this terming explicit in the next definition. \\

\begin{customdefn}{6.12 -- Cantor Normal form}
For any ordinal $\alpha$, we define the unique $\dot\Upsigma_{_{i<m}}\omega^{\zeta_i}\dot\times n_i=\alpha$ to be the {\bf Cantor Normal form} of $\alpha$, with $m\in\omega$, $\zeta\in ^mO_n$, and $n\in^m\omega$. This definition is justified by the Base Expansion Theorem. \\
\end{customdefn}

\section{Transfinite Hyperoperation Sequence}

In this section we first give the definition of some 'closure' ordinals with respect to $\dot+$, $\dot\times$, and $\uparrow$; we then generalize these recursive operations using a recursive sequence of binary operations indexed over $O_n$, and generalize the notion of a 'closure' ordinal to the $\alpha$-th hyperoperation. \\
We begin now by defining the aforementioned 'closure' ordinals with respect to the first three recursive binary operations on $O_n$, first explored by Cantor and named here as he named them. \\

\subsection{$\gamma$-numbers, $\delta$-numbers, and $\varepsilon$-numbers}

\begin{customdefn}{7.1} \
\begin{enumerate}
\item $\alpha$ is a {\bf $\gamma$-number} iff $\forall\beta\big(\beta<\alpha\Rightarrow\beta\dot+\alpha=\alpha\big)$.
\item $\alpha$ is a {\bf $\delta$-number} iff $\forall\beta\big(\beta<\alpha\Rightarrow\beta\dot\times\alpha=\alpha\big)$.
\item $\alpha$ is a {\bf $\varepsilon$-number} iff $\forall\beta\big(\beta<\alpha\Rightarrow\beta^\alpha=\alpha\big)$. \\
\end{enumerate}
\end{customdefn}

Thusly we see that $\gamma$-numbers, $\delta$-numbers, and $\varepsilon$-numbers represent 'absorbtion points' for the recursive binary operations we have defined on the ordinals.  It is clear that $\omega$ is a $\gamma$-number, a $\delta$-number, and a $\varepsilon$-number since any ordinal $n<\omega$ is finite, and consequently $n$ composed with $\omega$ on the right will be an infinitary union of $n$ composed with every natural number on the right, which is $\omega$. In searching for the smallest $\gamma$-number larger than $\omega$, we can see that $\omega^2$ is the smallest viable candidate since $\omega\dot\times n\dot+\omega^2=\bigcup_{\delta<\omega^2}\omega\dot\times n\dot+\delta=\omega^2$, but any ordinal less than $\omega^2$ is of the form $\dot\Upsigma_{i<m}\omega\dot\times n_i+j_i$ and consequently can be passed by ordinals less than it using recursive addition (besides $\omega$). Note that the first $\gamma$-number greater than $\omega$ is $\omega$ composed with itself using the next binary operation 'up' in our three-operation hierarchy, $\dot\times$. \\ 

Similarly, we see that the smallest $\delta$-number greater than $\omega$ is $\omega^{\omega}$, since $\omega^n\dot\times\omega^\omega=\bigcup_{\delta<\omega^\omega}\omega^n\dot\times\delta=\omega^\omega$ and any ordinal less than $\omega^\omega$ is of the form $\dot\Upsigma_{i<m}\omega^{l_{_i}}\dot\times n_i+j_i$ and consequently can be passed by ordinals less than it using recursive multiplication (besides $\omega$).  Note that, once again, the smallest $\delta$-number greater than $\omega$ is $\omega$ composed with itself using the next binary operation 'up' in our three-operation hierarchy, $\uparrow$. \\

Finally, we note that the smallest $\varepsilon$-number greater than $\omega$ is $\omega^{\omega^{\omega^{.^{.^{.}}}}}$. In order to more precisely formulate what we mean by the expression $\omega^{\omega^{\omega^{.^{.^{.}}}}}$, we now define the Transfinite Hyperoperation Sequence. \\

\subsection{The Transfinite Hyperoperation Sequence}

\begin{customthm}{7.2 -- The Transfinite Hyperoperation Sequence $\mercury$}
There exists a unique function $\mercury$ such that $dmn\mercury=\big(O_n\times(O_n\times O_n)\big)$ and $rng\mercury=O_n$ that is defined recursively as follows:
$$\mercury_{_\Omega}(\alpha,\beta)=\begin{cases} 
\mathcal{S}\alpha, & \text{if} \ \Omega=0. \\
\alpha, & \text{if} \ \Omega=1 \ \text{and} \ \beta=0. \\
\mathcal{S}\alpha,& \text{if} \ \Omega=1 \ \text{and} \ \beta=1. \\
0, & \text{if} \ \Omega=2 \ \text{and} \ \beta=0. \\
\alpha, & \text{if} \ \Omega=2 \ \text{and} \ \beta=1. \\
1, & \text{if} \ \Omega\geq3 \ \text{and} \ \beta=0. \\
\alpha, & \text{if} \ \Omega\geq3 \ \text{and} \ \beta=1. \\
\mercury_{_{\Omega-1}}\big(\alpha,\mercury_{_\Omega}(\alpha,\beta-1)\big), & \text{if} \ \Omega=\mathcal{S}\bigcup\Omega \ \text{and} \ 1<\beta=\mathcal{S}\bigcup\beta. \\
\bigcup_{\delta<\beta}\mercury_{_\Omega}(\alpha,\delta), & \text{if} \ \Omega=\mathcal{S}\bigcup\Omega \ \text{and} \ \ 1<\beta=\bigcup\beta . \\
\bigcup_{\rho<\Omega}\mercury_{_\rho}(\alpha,\beta), & \text{if} \ 0\neq\Omega=\bigcup\Omega \ \text{and} \ 1<\beta=\mathcal{S}\bigcup\beta. \\
\bigcup_{\rho<\Omega}\bigcup_{\delta<\beta}\mercury_{_\rho}(\alpha,\delta), & \text{if} \ 0\neq\Omega=\bigcup\Omega \ \text{and} \ 1<\beta=\bigcup\beta. \\
\end{cases}$$ \\
\end{customthm}
\begin{proof}
The proof of uniqueness for the first four operations in the sequence, where $\Omega=0,1,2,3$, can be found in any introductory set theory book, particularly D. Monk [1], observing that $\mercury_0(\alpha,\beta)=\mathcal{S}\alpha$, $\mercury_1(\alpha,\beta)=\alpha\dot+\beta$, $\mercury_2(\alpha,\beta)=\alpha\dot\times\beta$, and $\mercury_3(\alpha,\beta)=\alpha^\beta$. The uniqueness of all remaining functions follows in much the same fashion, using infinitary unions at every limit ordinal. \\
\end{proof}

\begin{note}
\begin{enumerate}
\item $\mercury_4(\alpha,4)=\alpha^{(\alpha^{(\alpha^\alpha)})}$,
\item $\mercury_i(m,n)<\omega$,
\item $\mercury_\omega(3,3)=\omega$. \\
\end{enumerate}
\end{note}

\begin{customdefn}{7.3}
 $\mercury$ is the unique function defined in Theorem 7.7; it is called the {\bf Transfinite Hyperoperation Sequence}. Further, $\mercury_{_\zeta}$ is the unique function defined by Theorem 7.7 in the case that $\Omega=\zeta$; it is called the {\bf $\zeta^{th}$ recursive hyperoperation}. Further, we will write $\alpha\mercury_{_\zeta}\beta=\gamma$ iff $\mercury_{_\zeta}(\alpha,\beta)=\gamma$. $\alpha\mercury_{_\zeta}\beta$ is called the {\bf $\zeta^{th}$ recursive hypercomposition of $\alpha$ and $\beta$}. \\
\end{customdefn}

\begin{customdefn}{7.4}
$\uparrow\uparrow$ is the unique function defined by Theorem 7.7 in the case that $\Omega=4$; it is called {\bf recursive tetration}. Further, we will write $^\alpha\beta=\gamma$ iff $\uparrow\uparrow(\alpha,\beta)=\gamma$. \\
\end{customdefn}

We can now formulate precisely that the first $\varepsilon$-number greater than $\omega$ is $\varepsilon_{_0}=\ ^\omega\omega=\omega^{\omega^{\omega^{.^{.^{.}}}}}$. $\varepsilon_{_0}$ is of great importance in many areas of mathematics, the primary reason being that any ordinal less than $\varepsilon_{_0}$ has a Cantor Normal form whose components can also be expressed in Cantor normal form, etc., and this process will only iterate a finite number of times. For ordinals greater than or equal to $\varepsilon_{_0}$ this process will iterate at least a countable number of times -- in a certain sense, any ordinal less than $\varepsilon_{_0}$ has a particular finite configuration of natural numbers combined with $\omega$ that can be used to uniquely express it, while for any ordinal greater than or equal to $\varepsilon_{_0}$ we need an infinite configuration. \\

Note that, once again, the smallest $\uparrow$-number greater than $\omega$ is $\omega$ composed with itself using the next hyperoperation up in $\mercury$. \\

We now present a standard theorem concerning $\gamma$, $\delta$, and $\varepsilon$ numbers. \\

\begin{customthm}{7.5} \
\begin{enumerate}
\item If $\alpha$ is a $\gamma$-number, then for all $\beta<\alpha$ and $\zeta<\alpha$, \\ $\beta\dot+\zeta<\alpha.$
\item If $\alpha$ is a $\delta$-number, then for all $\beta<\alpha$ and $\zeta<\alpha$, \\ $\beta\dot\times\zeta<\alpha.$
\item If $\alpha$ is a $\varepsilon$-number, then for all $\beta<\alpha$ and $\zeta<\alpha$, \\ $\beta^\zeta<\alpha.$ \\
\end{enumerate}
\end{customthm}
\begin{proof}
A proof of this theorem can be found in any introductory book on the theory of ordinals, particularly D. Monk [1]. \\
\end{proof}

Thusly we see that, in addition to 'absorbing' particular binary operations on the left, these numbers also represent possible 'closure points' for the recursive binary operations defined in $\mercury$.  Using the next definition, we now generalize the notion of $\gamma$-numbers, $\delta$-numbers, and $\varepsilon$-numbers to the $\alpha^{th}$ hyperoperation and tweak our language accordingly. \\

\begin{customdefn}{7.6}
For $n>0$, $\alpha$ is a {\bf $\mercury_{_\Omega}$-number} iff $\forall\beta\forall\gamma\big(\beta<\alpha\wedge\gamma<\alpha\Rightarrow\beta\mercury_{_\Omega}\gamma<\alpha\big)$. \\
\end{customdefn}

\begin{note}
Under the above definition we have that $\gamma$-numbers are $\dot+$-numbers, $\delta$-numbers are $\dot\times$-numbers, and $\varepsilon$-numbers are $\uparrow$-numbers. We also required that $n>0$ when referencing a $\mercury_n$-number in general, as there are no $\mercury_0$ numbers besides $O_n$, since the successor of any set ordinal is a new, unique ordinal by Corollary {\it 2.15}.\\
\end{note}

As the reader may have guessed, we now generalize the phenomenon we have observed thus far concerning numbers that 'absorb' a particular type of binary operation on the left. \\

\begin{customthm}{7.7}
\begin{enumerate}
\item $\omega$ is a $\mercury_n$-number for all $0<n<\omega$.
\item The smallest $\mercury_n$-number greater than $\omega$ is $\mercury_{n\dot+1}(\omega,\omega)$ for all $0<n<\omega$. 
\item If $\lambda$ is a $\mercury_n$-number, then the smallest $\mercury_n$-number greater than $\lambda$ is $\mercury_{n\dot+1}(\lambda,\omega)$ for all $0<n<\omega$.\\
\end{enumerate}
\end{customthm}
\begin{proof}
For {\it (1)}, we simply observe that any composition of a natural number with $\omega$ on the right using a finitely indexed hyperoperation will be the union of that natural number composed with every other natural number on the right, which is $\omega$. Explicitly, for any $0<n,m<\omega$, $$n\mercury_m\omega=\bigcup_{i\in\omega}n\mercury_m i=\omega.$$ This completes the proof of {\it (1)}. \\ For {\it (2)}, we observe that for any $\mercury_n$ if we can compose $\omega$ with itself then we can create expressions of the form $$\omega\mercury_n\omega\mercury_n\dots\mercury_n\omega=\omega\mercury_{n\dot+1}m$$ using only $\mercury_n$, where we composed $\omega$ with itself $m$ times on the left-hand side above -- this follows immedately from the definition of $\mercury_n$ and $\mercury_{n+1}$. We may then create expressions of the form $$(\omega\mercury_{n\dot+1}m)\mercury_n(\omega\mercury_{n\dot+1}p)=\omega\mercury_{n\dot+1}(m\dot+p)$$ using only $\mercury_n$, for any $m,p\in\omega$. We then observe that $$\omega<\alpha<\omega\mercury_{n\dot+1}\omega\implies\exists m\big(\omega\mercury_{n\dot+1}1<\alpha<\omega\mercury_{n\dot+1}m\big),$$ since $0<n$, thus $\omega<\alpha<\omega\mercury_{n\dot+1}\omega$ implies that $\alpha$ is not a $\mercury_n$-number. We finally observe that $$\nexists m\big(\omega\mercury_{n\dot+1}\omega<\omega\mercury_{n\dot+1}m\big)\implies\forall m\big(\omega\mercury_{n\dot+1}m<\omega\mercury_{n\dot+1}\omega\big),$$ thus $\omega\mercury_{n\dot+1}\omega=\mercury_{n\dot+1}(\omega,\omega)$ is a $\mercury_n$ number. {\it (2)} then follows by induction over $\omega$. \\ For {\it (3)}, suppose $\lambda$ is a $\mercury_n$-number. We then observe that $$\lambda<\alpha<\lambda\mercury_{n\dot+1}\omega\implies\exists m\big(\alpha<\lambda\mercury_{n+1}m\big),$$ thus $\lambda<\alpha<\lambda\mercury_{n\dot+1}\omega$ implies that $\alpha$ is not a $\mercury_n$-number.  We then observe that $$\nexists m\big(\lambda\mercury_{n\dot+1}\omega<\lambda\mercury_{n\dot+1}m\big)\implies\forall m\big(\lambda\mercury_{n\dot+1}m<\lambda\mercury_{n\dot+1}\omega\big),$$ thus $\lambda\mercury_{n\dot+1}\omega=\mercury_{n\dot+1}(\lambda,\omega)$ is a $\mercury_n$ number. {\it (3)} then follows by induction over $\omega$, observing that it is true for $n=1$. This completes the proof.\\ 
\end{proof}

Thus we see that, should we desire, we can always find ordinal sets that have successively 'larger' finitely indexed binary operations equipped on them, and that these ordinal sets will be closed under their 'large' binary operations. It is not immedately clear what happens past $\omega$, however I suspect we may be assuming the existence of an inaccesible cardinal by assuming the existence of, for example, a $\mercury_{_{\omega^\omega}}$-number. We now generalize a well-known theorem concerning transfinite $\dot+$, $\dot\times$ and $\uparrow$-numbers. \\

\begin{customthm}{7.8}
$\lambda$ is a transfinite $\mercury_n$-number $\iff \exists\zeta\big(\lambda=\omega\mercury_{n\dot+1}(\omega\mercury_{n\dot+1}\zeta)\big)$. \\
\end{customthm}

This theorem is rather tangential to the main course of our development here and as such I will not devote a large portion of time to proving it, however it represents a nice generalization of the representation theorems for $\delta$, $\gamma$ and $\varepsilon$-numbers. \\

\section{Natural Ordinal Arithmetic}

Here we present the aforementioned 'natural' definitions for addition and multiplication of ordinals (first given by Hessenberg in 1906, although defined in a different vocabulary then); these will serve as the basic binary operations in the main constructions of this paper. For more detailed information on natural ordinal arithmetic, see W. Sierpinski [2]. \\

\begin{customdefn}{8.1}
Suppose $\alpha,\beta\in O_n$, with $\zeta\in^nO_n$, $n'\in^n\omega$, $\gamma\in^mO_n$, and $m'\in^m\omega$, and let \\ $rng\zeta\cup rng\gamma=\mathbb{A}=\{\dot\mu_0,\dots,\dot\mu_{\ell-1}\}\subset O_n$ be ordered such that $\dot\mu_0>\dot\mu_1>\dots>\dot\mu_{\ell-1}$; then define a fifth function $\mu\in^{\ell}\mathbb{A}$ such that $\mu_i=\dot\mu_i$ for each $i<\ell$. Further, let $\alpha=\dot\Upsigma_{_{i<n}}\omega^{\zeta_i}\dot\times n'_i$ and $\beta=\dot\Upsigma_{_{i<m}}\omega^{\gamma_i}\dot\times m'_i$ be the Cantor Normal forms of $\alpha$ and $\beta$. We define $$\alpha+\beta=\dot\Upsigma_{_{i<\ell}}\omega^{\mu_i}\dot\times (n_i\dot+m_i),$$ where $$n_i=\begin{cases}n'_i,&\ \text{if}\ \mu_i\in rng\zeta \\ 0,&\ \text{if}\ \mu_i\notin rng\zeta\end{cases}$$ and $$m_i=\begin{cases}m'_i,&\ \text{if}\ \mu_i\in rng\gamma \\ 0,&\ \text{if}\ \mu_i\notin rng\gamma\end{cases}.$$ $+$ is called {\bf natural addition}, or simply {\bf addition}. $\alpha+\beta$ is called the {\bf natural sum of $\alpha$ and $\beta$}, or simply the {\bf sum of $\alpha$ and $\beta$}. \\
\end{customdefn}

Although this definition seems somewhat technical at a glance, it is actually a much nicer algebraic operation on the ordinals than $\dot+$. As we will see in the next theorem, it is commutative -- this usually comes at the cost of 'right continuity in the argument', in the sense that $1+x=\omega$ has no solution in the ordinals. Although this may seem like a critical failure, we will get around this problem with our 'Surinteger' extension of the ordinals. \\

It is worth noting further that $1\dot+\omega=\omega\neq\omega\dot+1=\omega+1=1+\omega$. This behavior generalizes in the intuitively expected way, and we thusly see that the apparent technicality of this definition is really just the technical way of 'forcing' commutativity on our addition of ordinals in the canonical way. It orders the sum in a consistent fashion for any binary pairing of ordinals according to their Cantor normal forms, then uses $\dotplus$ to combine the elements once they're all in order. \\

\begin{customthm}{8.2 -- Associative and Commutative Law for $+$}
 \
\begin{enumerate}
\item $\forall\alpha\forall\beta\forall\gamma\big(\alpha+(\beta+\gamma)=(\alpha+\beta)+\gamma\big)$.
\item $\forall\alpha\forall\beta\big(\alpha+\beta=\beta+\alpha\big)$.
\end{enumerate}
\end{customthm}
\begin{proof}
For {\it (1)}, suppose $\alpha,\beta,\gamma\in O_n$, with $\zeta\in^nO_n$, $n'\in^n\omega$, $\gamma\in^mO_n$, $m'\in^m\omega$, $\psi\in^p O_n$, $p'\in^p\omega$, and let $rng\zeta\cup rng\gamma=\mathbb{A}=\{\dot\mu_0,\dots,\dot\mu_{\ell-1}\}\subset O_n$ be ordered such that $\dot\mu_0>\dot\mu_1>\dots>\dot\mu_{\ell-1}$, and let $rng\gamma\cup rng\psi=\mathbb{B}=\{\dot\rho_0,\dots,\dot\rho_{q-1}\}$ be ordered such that $\dot\rho_0>\dots>\dot\rho_{q-1}$; then define a third function $\mu\in^{\ell}\mathbb{A}$ such that $\mu_i=\dot\mu_i$ for each $i<\ell$, and a fourth function $\rho\in^q\mathbb{B}$ such that $\rho_i=\dot\rho_i$ for all $i<q$. Finally, we let $rng \zeta\cup rng\gamma\cup rng\psi=\mathbb{A}\cup\mathbb{B}=\mathbb{C}=\{\dot\eta_0,\dots,\dot\eta_{r-1}\}$ be ordered such that $\dot\eta_0>\dots>\dot\eta_{r-1}$, and we then define a fifth function $\eta\in^r\mathbb{C}$ such that $\eta_i=\dot\eta_i$ for all $i<r$. Further, let $\alpha=\dot\Upsigma_{_{i<n}}\omega^{\zeta_i}\dot\times n'_i$, $\beta=\dot\Upsigma_{_{i<m}}\omega^{\gamma_i}\dot\times m'_i$, and $\gamma=\dot\Upsigma_{_{i<p}}\omega^{\psi_i}\dot\times p'_i$ be the Cantor Normal forms of $\alpha$, $\beta$, and $\gamma$. We then observe that $$\alpha+(\beta+\gamma)=\alpha+\dot\Upsigma_{_{i<q}}\omega^{\rho_i}\dot\times(m_i\dot+p_i)=\dot\Upsigma_{i<r}\omega^{\eta_i}\dot\times\big(n_i\dot+(m_i\dot+p_i)\big)$$ $$=\dot\Upsigma_{i<r}\omega^{\eta_i}\dot\times\big((n_i\dot+m_i)\dot+p_i\big)=\dot\Upsigma_{i<\ell}\omega^{\mu_i}\dot\times(n_i\dot+m_i)+\gamma=(\alpha+\beta)+\gamma,$$ since $\dot+$ is associative on all ordinals. This completes the proof of {\it (1)}. \\ For {\it (2)}, let $\alpha$ and $\beta$ be as above. We then have that $$\alpha+\beta=\dot\Upsigma_{_{i<\ell}}\omega^{\zeta_i}\dot\times (n_i\dot+m_i)=\dot\Upsigma_{_{i<\ell}}\omega^{\zeta_i}\dot\times (m_i\dot+n_i)=\beta+\alpha,$$ since $\dot+$ is commutative on natural numbers. This completes the proof. \\
\end{proof}

We now formulate the notion of a natural finite sum of ordinals. \\

\begin{customdefn}{8.3}
For all $\dot\mu\in^mO_n$, we define a second function $\mu\in^\omega O_n$ such that $\mu_i=\dot\mu_i$ for all $i<m$ and $\mu_i=0$ for all $i\geq m$. We then define $$\sum(\mu,n)=\mu_0+\mu_1+\dots+\mu_{n-1}.$$ We will write $\sum_{i<n}\mu_i$ instead of $\sum(\mu,n)$. \\
\end{customdefn}

\begin{customthm}{8.4}
$\sum_{i<n}\mu_i$ is commutative for all $\mu$, $m$ and $n$.
\end{customthm}
\begin{proof}
Let $\dot\mu'$ be any function with $dmn\dot\mu'=dmn\dot\mu=m$ and $rng\dot\mu'$ indexed as a permutation of $rng\dot\mu$.  Then $$\sum_{i<n}\mu_i=\sum_{i<n}\mu'_i,$$ since $+$ is commutative for all ordinals by Theorem {\it 8.2}. This completes the proof. \\
\end{proof}

Using definition 8.3, we now formulate natural multiplication. \\

\begin{customdefn}{8.5}
Suppose $\alpha,\beta\in O_n$, with $\alpha=\dot\Upsigma_{_{i<n}}\omega^{\zeta_i}\dot\times n'_i$ and $\beta=\dot\Upsigma_{_{i<m}}\omega^{\gamma_i}\dot\times m'_i$ the Cantor Normal forms of $\alpha$ and $\beta$. We define $$\alpha\times\beta=\sum_{_{i<n}}\sum_{_{j<m}}\omega^{\zeta_i+\gamma_j}n_i\dot\times m_j.$$ $\times$ is called {\bf natural multiplication}, or simply {\bf multiplication}. $\alpha\times\beta$ is called the {\bf natural product of $\alpha$ and $\beta$}, or simply the {\bf product of $\alpha$ and $\beta$}. \\
\end{customdefn}

Thus we see that the natural product of two ordinals behaves like polynomial multiplication with $\omega$ as the lone independent variable and the natural numbers playing the role of coefficients, where we take the natural sum of the exponents. As one would expect, this multiplicative structure is commutative. \\

\begin{customthm}{8.6 -- Associative and Commutative Law  for $\times$} \
\begin{enumerate}
\item $\forall\alpha\forall\beta\forall\gamma\big(\alpha\times(\beta\times\gamma)=(\alpha\times\beta)\times\gamma$.
\item $\forall\alpha\forall\beta\big(\alpha\times\beta=\beta\times\alpha\big)$.
\end{enumerate}
\end{customthm}
\begin{proof}
That $\times$ is associative follows immedately from the associativity of $+$ and $\dot\times$.  \\ For commutativity, suppose $\alpha,\beta\in O_n$, with $\alpha=\dot\Upsigma_{_{i<n}}\omega^{\zeta_i}\dot\times n'_i$ and $\beta=\dot\Upsigma_{_{i<m}}\omega^{\gamma_i}\dot\times m'_i$ the Cantor Normal forms of $\alpha$ and $\beta$. Then $$\alpha\times\beta=\sum_{_{i<n}}\sum_{_{j<m}}\omega^{\zeta_i+\gamma_j}n_i\dot\times m_j = \sum_{_{j<m}}\sum_{_{i<n}}\omega^{\gamma_j+\zeta_i}m_j\dot\times n_i = \beta\times\alpha,$$ since $\dot\times$ is commutative for natural numbers and $\sum(\mu,n)$ is commutative for all $\mu$ and $n$ by Theorem 6.4. This completes the proof. \\
\end{proof}

We now extend the notion of $\dot+$ and $\dot\times$-numbers to $+$ and $\times$. \\

\begin{customdefn}{8.7}
\begin{enumerate}
\item $\alpha$ is a {\bf $+$-number} iff for all $\beta<\alpha$ and $\gamma<\alpha$, \\ $\beta+\gamma<\alpha.$
\item $\alpha$ is a {\bf $\times$-number} iff for all $\beta<\alpha$ and $\gamma<\alpha$, \\ $\beta\times\gamma<\alpha.$ \\
\end{enumerate}
\end{customdefn}

\begin{customthm}{8.8}
\begin{enumerate}
\item Every $\dot+$-number is a $+$-number.
\item Every $\dot\times$-number is a $\times$-number.
\end{enumerate}
\end{customthm}
\begin{proof}
By Theorem 7.2 and Theorem 7.4. \\
\end{proof}

Thus we see that 'closing out' the first two natural arithmetic operations on $O_n$ is easy once we have located the 'absorbtion' points for the first two recursive arithmetic operations. We now give the last definition and theorem necessary to begin constructing number fields. \\

\begin{customdefn}{8.8}
\begin{enumerate}
\item $1^{st}x=\bigcap\bigcap x$. $1^{st}x$ is called the {\bf first coordinate} of $x$.
\item $2^{nd}x=\bigcap\bigcap\bigcap\{x\}^{-1}$. $2^{nd}x$ is called the {\bf second coordinate} of $x$.
\end{enumerate}
\end{customdefn}

These definitions are obviously only really of interest when $x$ is an ordered pair.

\begin{customthm}{8.9}
\begin{enumerate}
\item $1^{st}(a,b)=a$.
\item $2^{nd}(a,b)=b$. \\
\end{enumerate}
\end{customthm}

\section{The Surintegers}

We now define a proper-class sized, commutative, discretely ordered ring with a unique cyclic subring that is consequenty isomorphic to any construction of the Integers. We must first define an additive and multiplicative structure on $O_n\times O_n$ as follows. \\

\begin{customdefn}{9.1}
Let $a,b,c\in O_n\times O_n$. We define
\begin{enumerate}
\item $\ddot+=\{\big((a,b),c\big):(1^{st}a+1^{st}b,2^{nd}a+2^{nd}b)=c\}.$ We will write $a\ddot+ b=c$ iff $\big((a,b),c\big)\in\ddot+$.	
\item $\ddot\times=\{\big((a,b),c\big):(1^{st}a\times2^{nd}b+2^{nd}a\times1^{st}b,1^{st}a\times1^{st}b+2^{nd}a\times2^{nd}b)=c\}.$ We will write $a\ddot\times b=c$ iff $\big((a,b),c\big)\in\ddot\times$. \\
\end{enumerate}
\end{customdefn}

\begin{customthm}{9.2}
$\ddot+$ and $\ddot\times$ are both associative and commutative.
\end{customthm}
\begin{proof}
For the associativity of $\ddot+$, we simply observe that $$a\ddot+(b\ddot+c)=a\ddot+(1^{st}b+1^{st}c,2^{nd}b+2^{nd}c)$$ $$=\big(1^{st}a+(1^{st}b+1^{st}c),2^{nd}a+(2^{nd}b+2^{nd}c)\big)$$ $$=\big((1^{st}a+1^{st}b)+1^{st}c,(2^{nd}a+2^{nd}b)+2^{nd}c\big)$$ $$=(1^{st}a+1^{st}b,2^{nd}a+2^{nd}b)\ddot+c=(a\ddot+b)\ddot+c,$$ since $+$ is associative on all ordinals. For the commutativity of $\ddot+$, we observe that $$a\ddot+b=(1^{st}a+1^{st}b,2^{nd}a+2^{nd}b)=(1^{st}b+1^{st}a,2^{nd}b+2^{nd}a)=b\ddot+a$$ since $+$ is commutative on all ordinals. For the associativity of $\ddot\times$, we observe that $$a\ddot\times(b\ddot\times c)=a\ddot\times(1^{st}b\times2^{nd}c+2^{nd}b\times1^{st}c,1^{st}b\times1^{st}c+2^{nd}b\times2^{nd}c)$$ $$=\big(1^{st}a\times(1^{st}b\times1^{st}c+2^{nd}b\times2^{nd}c)+2^{nd}a\times(1^{st}b\times2^{nd}c+2^{nd}b\times1^{st}c),$$ $$1^{st}a\times(1^{st}b\times2^{nd}c+2^{nd}b\times1^{st}c)+2^{nd}a\times(1^{st}b\times1^{st}c+2^{nd}b\times2^{nd}c)\big)$$ $$=\big(1^{st}a\times1^{st}b\times1^{st}c+1^{st}a\times2^{nd}b\times2^{nd}c+2^{nd}a\times1^{st}b\times2^{nd}c+2^{nd}a\times2^{nd}b\times1^{st}c,$$ $$1^{st}a\times1^{st}b\times2^{nd}c+1^{st}a\times2^{nd}b\times1^{st}c+2^{nd}a\times1^{st}b\times1^{st}c+2^{nd}a\times2^{nd}b\times2^{nd}c\big)$$ $$=\big((1^{st}a\times2^{nd}b+2^{nd}a\times1^{st}b)\times2^{nd}c+(1^{st}a\times1^{st}b+2^{nd}a\times2^{nd}b)\times1^{st}c,$$ $$(1^{st}a\times2^{nd}b+2^{nd}a\times1^{st}b)\times1^{st}c+(1^{st}a\times1^{st}b+2^{nd}a\times2^{nd}b)\times2^{nd}c\big)$$ $$=(1^{st}a\times2^{nd}b+2^{nd}a\times1^{st}b,1^{st}a\times1^{st}b+2^{nd}a\times2^{nd}b)\ddot\times c = (a\ddot\times b)\ddot\times c,$$ since $\times$ is commutative and associative on ordinals. Finally, for the commutativity of $\ddot\times$ we simply observe that $$a\ddot\times b=(1^{st}a\times2^{nd}b+2^{nd}a\times1^{st}b,1^{st}a\times1^{st}b+2^{nd}a\times2^{nd}b)=(1^{st}b\times2^{nd}a+2^{nd}b\times1^{st}a,1^{st}b\times1^{st}a+2^{nd}b\times2^{nd}a)=b\ddot\times a,$$ since $+$ and $\times$ are both commutative on ordinals.  This completes the proof. \\
\end{proof}

We now define the notion of a finite sum in $O_n\times O_n$ under this additive structure. \\

\begin{customdefn}{9.3}
Suppose that $m\in\omega$ and $\mu\in^m(O_n\times O_n)$. We define $$\ddot\sum(\mu,n)=\mu'_0\ddot+\mu'_1\ddot+\dots\ddot+\mu'_n,$$ where $$\mu'_i=\begin{cases}\mu_i,& i<m, \\ 0,& i\leq m.\end{cases}$$ We will write $\ddot\sum_{i<n}\mu_i$ instead of $\ddot\sum(\mu,n)$. \\
\end{customdefn}

We are now technically prepared to define the Surintegers. \\

\begin{customdefn}{9.4}
Let $\dot\omega=(0,\omega)$, and $(0,\omega)^{(0,\alpha)}=(0,\omega^\alpha)$ for all $\alpha$. We then define
$$\mathbb{Z}_\infty=\{\ddot\sum_{i<n}\dot\omega^{\zeta_i}\ddot\times\mu_i:\zeta\in^n(\{0\}\times O_n)\wedge\zeta_0\neq\zeta_1\neq\dots\neq\zeta_{n-1}\wedge\mu\in^n(\omega\times\{0\}\cup\{0\}\times\omega)\}.$$ 
$\mathbb{Z}_\infty$ will be called the {\bf Surintegers}. Lowercase italic letters from the beginning of the alphabet $a,b,c,\dots$ will be used to denote members of the Surintegers, called {\bf surintegers}, unless otherwise noted. If $a=\ddot\sum_{i<n}\dot\omega^{\zeta_i}\ddot\times\mu_i$, we will call $\ddot\sum_{i<n}\dot\omega^{\zeta_i}\ddot\times\mu_i$ the {\bf surinteger normal form} of $a$. \\
\end{customdefn}

This definition looks technical, however the object constructed amounts to a class full of all possible Cantor normal forms with positive {\it and} negative integer coefficients. If we choose to view the ordinals as a 'long line' starting at $0$ and stretching off to the right with an ellipse preceding each new limit ordinal on the line as below: \\
\setlength{\unitlength}{.1cm}
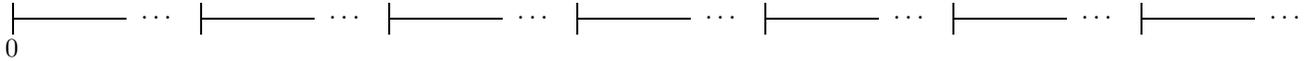
\begin{picture}(100,20)(50,55)
\multiput(50,65)(25,0){7}{\line(20,0){15}}
\multiput(50,63)(25,0){7}{\line(0,0){4}}
\multiput(67,65)(25,0){7}\dots
\put(49,60)0 \\
\end{picture}
 \\Then what we have done by adding the possibility of negative coefficients on Cantor normal forms amounts to 'flipping' the line over $0$ and adding new 'negative' positions on the line prior to each limit ordinal, as follows: \\
\setlength{\unitlength}{.1cm}
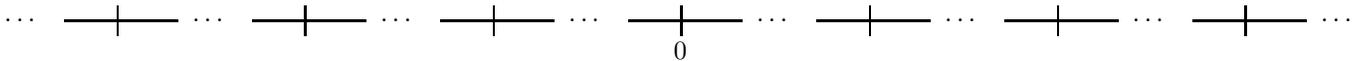
\begin{picture}(100,20)(50,55)
\multiput(53,65)(25,0){7}{\line(20,0){15}}
\multiput(60,63)(25,0){7}{\line(0,0){4}}
\multiput(45,65)(25,0){8}\dots
\put(134,60)0 \\
\end{picture}

Note that at limit surintegers such as $\dot\omega$ there are $\omega$ new positions of the form $\dot\omega\ddot+(n,0)$ for all $n\in\omega$, at limit surintegers such as $\dot\omega^{(0,2)}$ there are $2\times\omega^2$ new positions of the form $\dot\omega^2\ddot+(m,0)\dot\omega\ddot\pm(0,n)$ for all $m,n\in\omega$, and in general at each new $+$-number $\delta>\omega$ we pick up $2\times\delta$ 'new positions' placed before $(0,\delta)\ddot\times(0,n)$ for all $n\in\omega$ -- this behavior is mirrored on the 'negative half' of $\mathbb{Z}_\infty$. \\  Further, observe that each Surinteger has two equivalent 'representations' already present under these definitions -- the 'Surinteger normal form' they are defined in terms of, and a unique ordered pair of Cantor normal forms sharing no mutual powers of $\omega$.  We make this equivalence explicit in the next theorem. \\

\begin{customcor}{9.5}
Suppose $\zeta\in^n(\{0\}\times O_n)$ and $\mu\in^n(\omega\times\{0\}\cup\{0\}\times\omega)$, with $_+\mu=\{(i,\mu_i):\mu_i\in\mathbb{Z}_\infty^+\}$ and $_-\mu=\{(i,\mu_i):\mu_i\in\mathbb{Z}_\infty^-\}$. Then 
$$\ddot\sum_{i<n}\dot\omega^{\zeta_i}\ddot\times\mu_i=\big(\dot\Upsigma_{_{i\in dmn_-\mu}}\omega^{2^{nd}\zeta_i}\dot\times1^{st}_-\mu_i,\dot\Upsigma_{_{i\in dmn_+\mu}}\omega^{2^{nd}\zeta_i}\dot\times2^{nd}_+\mu_i\big).$$
\end{customcor}
\begin{proof}
We first rename $_+\mu=\{(i_0,\mu_{i_0}),\dots,(i_m,\mu_{i_m})\}$ and define $_+\zeta=\{(0,\zeta_{i_0}),\dots,(m,\zeta_{i_m})\}$, and similarly rename $_-\mu=\{(j_0,\mu_{j_0}),\dots,(j_p,\mu_{j_p})\}$ and define $_-\zeta=\{(0,\zeta_{j_0}),\dots,(p,\zeta_{j_p})\}$. We then define $_+\mu'=\{(0,\mu_{i_0}),\dots,(m,\mu_{i_m})\}$ and $_-\mu'=\{(0,\mu_{i_0}),\dots,(p,\mu_{i_p})\}$, and observe that $$\ddot\sum_{i<n}\dot\omega^{\zeta_i}\ddot\times\mu_i=\ddot\sum_{j<p}\dot\omega^{_-\zeta_j}\ddot\times_-\mu'_j\ddot+\ddot\sum_{i<m}\dot\omega^{_+\zeta_i}\ddot\times_+\mu'_i.$$ But we then have that $$\ddot\sum_{j<p}\dot\omega^{_-\zeta_j}\ddot\times_-\mu'_j=(\dot\Upsigma_{_{i\in dmn_-\mu}}\omega^{2^{nd}\zeta_i}\dot\times1^{st}_-\mu_i,0)$$ and $$\ddot\sum_{i<m}\dot\omega^{_+\zeta_i}\ddot\times_+\mu'_i=(0,\dot\Upsigma_{_{i\in dmn_+\mu}}\omega^{2^{nd}\zeta_i}\dot\times2^{nd}_+\mu_i),$$ thus $$\ddot\sum_{j<p}\dot\omega^{_-\zeta_j}\ddot\times_-\mu'_j\ddot+\ddot\sum_{i<m}\dot\omega^{_+\zeta_i}\ddot\times_+\mu'_i=(\dot\Upsigma_{_{i\in dmn_-\mu}}\omega^{2^{nd}\zeta_i}\dot\times1^{st}_-\mu_i,0)\ddot+(0,\dot\Upsigma_{_{i\in dmn_+\mu}}\omega^{2^{nd}\zeta_i}\dot\times2^{nd}_+\mu_i)$$ $$=(\dot\Upsigma_{_{i\in dmn_-\mu}}\omega^{2^{nd}\zeta_i}\dot\times1^{st}_-\mu_i,\dot\Upsigma_{_{i\in dmn_+\mu}}\omega^{2^{nd}\zeta_i}\dot\times2^{nd}_+\mu_i),$$ completing the proof. \\
\end{proof}

Although the exact signifigance of the above corollary may not be immedately apparent, we hopefully elucidate the relationship proven with the following brief example and Corollary. \\\

\begin{customexmp}{9.6}
$$\dot\omega^{(0,5)}\ddot\times(0,4)\ddot+\dot\omega^{(0,4)}\ddot\times(2,0)\ddot+\dot\omega^{(0,2)}\ddot\times(7,0)\ddot+\dot\omega\ddot\times(0,3)\ddot+(1,0)$$ $$=(0,\omega^5)\ddot\times(0,4)\ddot+(0,\omega^4)\ddot\times(2,0)\ddot+(0,\omega^2)\ddot\times(7,0)\ddot+(0,\omega)\ddot\times(0,3)\ddot+(1,0)$$ $$=(0,\omega^5\times4)\ddot+(\omega^4\times2,0)\ddot+(\omega^2\times7,0)\ddot+(0,\omega\times3)\ddot+(1,0)$$ $$=(\omega^4\times2+\omega^2\times7+1,\omega^5\times4+\omega\times3)$$ $$=(\omega^4\dot\times2\dot+\omega^2\dot\times7\dot+1,\omega^5\dot\times4\dot+\omega\dot\times3).$$ \\
\end{customexmp}

\begin{customcor}{9.7}
$\mathbb{Z}_\infty\subset(O_n\times O_n)$
\end{customcor}
\begin{proof}
By Theorem 9.5, we have that $\mathbb{Z}_\infty\subseteq(O_n\times O_n)$. To establish that $\mathbb{Z}_\infty\neq(O_n\times O_n)$, consider that $(3,5)\notin\mathbb{Z}_\infty$. This completes the proof. \\
\end{proof}

We now define the ordering structure of $\mathbb{Z}_\infty$, which will be primarily based around checking the coefficient on the highest present power of $\omega$. \\

\begin{customdefn}{9.8}
\begin{enumerate}
\item $\mathbb{Z}_\infty^+=\{a:1^{st}a\leq2^{nd}a\}$. $\mathbb{Z}_\infty^+$ will be called the {\bf postive} Surintegers.
\item $\mathbb{Z}_\infty^-=\{a:2^{nd}a<1^{st}a\}$. $\mathbb{Z}_\infty^-$ will be called the {\bf negative} Surintegers.
\item $\dot\leq=\{(a,b):(a\in\mathbb{Z}_\infty^-\wedge b\in\mathbb{Z}_\infty^+)\vee(a\in\mathbb{Z}_\infty^+\wedge b\in\mathbb{Z}_\infty^+\wedge[2^{nd}a<2^{nd}b\vee(2^{nd}a=2^{nd}b\wedge1^{st}b\leq1^{st}a)])\vee(a\in\mathbb{Z}_\infty^-\wedge b\in\mathbb{Z}_\infty^-\wedge[1^{st}b<1^{st}a\vee(1^{st}a=1^{st}b\wedge2^{nd}a\leq2^{nd}b)])\}$. We will write $a\dot\leq b$ iff $(a,b)\in\dot\leq$. \\
\end{enumerate}
\end{customdefn}

Thus we see that the ordering structure of $\mathbb{Z}_\infty$ amounts to interpreting the right coordinate of an ordered pair as the 'positive' part of a surinteger, and the left coordinate as the 'negative' part of a surinteger. If the positive part is larger than or equal to the negative part it is a positive surinteger; if the opposite is true, it is a negative surinteger. Note that $(0,0)$ is positive under this definition, and $\mathbb{Z}_\infty^+\cap\mathbb{Z}_\infty^-$ is empty. \\ We now leverage the duality proven in Corollary {\it 9.5} to express an equialent condition for a surinteger being positive or negative based on its normal form expansion, rather than its coordinate form. \\

\begin{customcor}{9.9}
$\mathbb{Z}_\infty=\mathbb{Z}_\infty^+\cup\mathbb{Z}_\infty^-.$
\end{customcor}
\begin{proof}
We have by definition {\it 9.8} that $\mathbb{Z}_\infty^+\subseteq\mathbb{Z}_\infty$ and $\mathbb{Z}_\infty^-\subseteq\mathbb{Z}_\infty$, thus $\mathbb{Z}_\infty^+\cup\mathbb{Z}_\infty^-\subseteq\mathbb{Z}_\infty$. Suppose $a\in\mathbb{Z}_\infty$, so $a=(\alpha,\beta)$ for some $\alpha,\beta\in O_n$ by Corollary {\it 9.7}. Since $\leq$ is a total ordering on $O_n$, either $\alpha\leq\beta$ or $\beta<\alpha$, thus $a\in\mathbb{Z}_\infty^+$ or $a\in\mathbb{Z}_\infty^-$ respectively. Since $a$ was arbitrary in $\mathbb{Z}_\infty$ we have that $\mathbb{Z}_\infty\subseteq(\mathbb{Z}_\infty^+\cup\mathbb{Z}_\infty^-)$, thus $\mathbb{Z}_\infty=\mathbb{Z}_\infty^+\cup\mathbb{Z}_\infty^-$. This completes the proof. \\
\end{proof}

\begin{customdefn}{9.10}
Suppose $\zeta\in^n(\{0\}\times O_n)$, and $a=\ddot\sum_{i<n}\dot\omega^{\zeta_i}\ddot\times_\pm n_i$ is the surinteger normal form of $a$. We then define $$\zeta^\uparrow=\sup rng\zeta,$$ and $_\pm n^\uparrow$ to be the corresponding coefficient in the normal form of a. We then define $$a^\uparrow=\dot\omega^{\zeta^\uparrow}\ddot\times _\pm n^\uparrow.$$ $a^\uparrow$ will be called the {\bf maximum representative of $a$}. We then define $$\zeta^{\uparrow-1}=\sup(rng\zeta\sim\zeta^\uparrow)$$ $$\vdots$$ $$\zeta^{\uparrow-m}=\sup(rng\zeta\sim\{\zeta^{\uparrow-i}\}_{i=0}^{m-1}),\ 1\leq m<n,$$ where $\zeta^{\uparrow-0}=\zeta^\uparrow$. We accordingly define $_\pm n^{\uparrow-m}$, $1\leq m<n$, to be the corresponding coefficient in the normal form of a, and define $$a^{\uparrow-m}=\dot\omega^{\zeta^{\uparrow-m}}\ddot\times_\pm n^{\uparrow-m},\ 1\leq m<n.$$ $a^{\uparrow-m}$ will be called the {\bf $m^{th}$ representative of $a$}, with $a^\uparrow$ sometimes referred to as the {\bf $0^{th}$ represetative of $a$}. \\
\end{customdefn}

We thusly see that the maximum representative (or $0^{th}$ representative) of a surinteger is simply the largest order term in its normal form expansion, and the $m^{th}$ representative is the $m^{th}$ largest order term in the normal form expansion (with the $0^{th}$ being the largest, $1^{st}$ being next largest, etc.). \\

\begin{customcor}{9.11}
For any $a\in\mathbb{Z}_\infty$,
\begin{enumerate}
\item $a\in\mathbb{Z}_\infty^+$ iff $a^\uparrow\in\mathbb{Z}_\infty^+$. 
\item For all $m$ such that $0\leq m<n$, $$a^{\uparrow-m}\in\mathbb{Z}_\infty^+\implies a^{\uparrow-m}=(0,2^{nd}a^{\uparrow-m}).$$
\item $a\in\mathbb{Z}_\infty^-$ iff $a^\uparrow\in\mathbb{Z}_\infty^-$.
\item For all $m$ such that $0\leq m<n$, $$a^{\uparrow-m}\in\mathbb{Z}_\infty^-\implies a^{\uparrow-m}=(1^{st}a^{\uparrow-m},0).$$
\end{enumerate}
\end{customcor}
\begin{proof}
{\it (1)} follows immedately from Corollary {\it 9.5}, observing that the ordered pair form of $a$ will satisfy $1^{st}a\leq2^{nd}a$ iff the cantor normal form of $2^{nd}a$ has a larger power of $\omega$ in it than the largest power of $\omega$ in $1^{st}a$, keeping in mind that $1^{st}a$ and $2^{nd}a$ share no powers of $\omega$. \\ For {\it (2)}, we observe that $$a^{\uparrow-m}\in\mathbb{Z}_\infty^+\implies_\pm n^{\uparrow-m}\in\mathbb{Z}_\infty^+\implies_\pm n^{\uparrow-m}=(0,n^{\uparrow-m}),$$ thus $$a^{\uparrow-m}=\dot\omega^{\zeta^{\uparrow-m}}\ddot\times_\pm n^{\uparrow-m}=(0,\omega^{\zeta^{\uparrow-m}})\ddot\times(0,n^{\uparrow-m})=(0,\omega^{\zeta^{\uparrow-m}}\times n^{\uparrow-m})=(0,2^{nd}a^{\uparrow-m}).$$ {\it (3)} follows from {\it (1)} and Corollary {\it 9.9}, and {\it (4)} then follows since $$a^{\uparrow-m}\in\mathbb{Z}_\infty^-\implies_\pm n^{\uparrow-m}\in\mathbb{Z}_\infty^-\implies_\pm n^{\uparrow-m}=(n^{\uparrow-m},0),$$ thus $$a^{\uparrow-m}=\dot\omega^{\zeta^{\uparrow-m}}\ddot\times_\pm n^{\uparrow-m}=(0,\omega^{\zeta^{\uparrow-m}})\ddot\times(n^{\uparrow-m},0)=(\omega^{\zeta^{\uparrow-m}}\times n^{\uparrow-m},0)=(1^{st}a^{\uparrow-m},0).$$ This completes the proof. \\
\end{proof}

\begin{customthm}{9.12}
$\dot\leq$ is a total ordering on $\mathbb{Z}_\infty$.
\end{customthm}
\begin{proof}
Suppose $a\dot\leq b$ and $b\dot\leq c$. By Corollary {\it 9.9}, we have four possible options to work with: 
\begin{enumerate}
\item $a,b,c\in\mathbb{Z}_\infty^-$, 
\item $a,b\in\mathbb{Z}_\infty^-$ and $c\in\mathbb{Z}_\infty^+$, 
\item $a\in\mathbb{Z}_\infty^-$ and $b,c\in\mathbb{Z}_\infty^+$, or 
\item $a,b,c\in\mathbb{Z}_\infty^+$. 
\end{enumerate}
Any other combination would give a contradiction, since (for example) if $a,b\in\mathbb{Z}_\infty^+$ and $c\in\mathbb{Z}_\infty^-$ then since $b\dot\leq c$ we have that $b,c\in\mathbb{Z}_\infty^+$ or $b,c\in\mathbb{Z}_\infty^-$ or $b\in\mathbb{Z}_\infty^-$ and $c\in\mathbb{Z}_\infty^+$, any of which is a contradiction since $\mathbb{Z}_\infty^+\cap\mathbb{Z}_\infty^-=0$. Suppose $(1)$ holds; we then have four possible sub-options:
\begin{itemize}
\item $1^{st}c<1^{st}b<1^{st}a\implies1^{st}c<1^{st}a\implies a\dot\leq c$.
\item $1^{st}c<1^{st}b=1^{st}a$, and $2^{nd}a\leq2^{nd}b\implies1^{st}c<1^{st}a\implies a\dot\leq c$.
\item $1^{st}c=1^{st}b<1^{st}a$, and $2^{nd}b\leq2^{nd}c\implies1^{st}c<1^{st}a\implies a\dot\leq c$.
\item $1^{st}c=1^{st}b=1^{st}a$, and $2^{nd}a\leq2^{nd}b\leq2^{nd}c\implies 2^{nd}a\leq2^{nd}c\implies a\dot\leq c$.
\end{itemize}
We are thusly satisfied that $\dot\leq$ is transitive in this case. If $(2)$ or $(3)$ hold then $a\in\mathbb{Z}_\infty^-$ and $c\in\mathbb{Z}_\infty^+$, thus $a\dot\leq c$ and consequently $\dot\leq$ is transitive. The case if $(4)$ holds is entirely symmetric to the first case. This completes the proof that $\dot\leq$ is transitive. \\ Now suppose that $a\dot\leq b$ and $b\dot\leq a$ -- we then have a few cases to work with:
\begin{enumerate}
\item $a\in\mathbb{Z}_\infty^+\wedge b\in\mathbb{Z}_\infty^+\wedge 2^{nd}a=2^{nd}b$, and $1^{st}a\leq1^{st}b\wedge1^{st}b\leq1^{st}a\implies1^{st}a=1^{st}b\implies a=b$.
\item $a\in\mathbb{Z}_\infty^-\wedge b\in\mathbb{Z}_\infty^-\wedge 1^{st}a=1^{st}b$, and $2^{nd}b\leq2^{nd}a\wedge2^{nd}a\leq2^{nd}b\implies2^{nd}a=2^{nd}b\implies a=b$.
\end{enumerate}
Any other placements give contradictions in the same fashion as above, since $\mathbb{Z}_\infty^+\cap\mathbb{Z}_\infty^-=0$. This completes the proof that $\dot\leq$ is antisymmetric, and this together with the fact that it is transitive proves that $\dot\leq$ is an ordering with field $\mathbb{Z}_\infty\times\mathbb{Z}_\infty$. \\ We now prove that $\dot\leq$ is a total ordering; suppose $a,b\in\mathbb{Z}_\infty$. We then have four possible cases:
\begin{enumerate}
\item $1^{st}a\leq2^{nd}a\wedge1^{st}b\leq2^{nd}b\implies a\in\mathbb{Z}_\infty^+\wedge b\in\mathbb{Z}_\infty^+$.
\item $1^{st}a\leq2^{nd}a\wedge2^{nd}b\leq1^{st}b\implies a\in\mathbb{Z}_\infty^+\wedge \big(b=0\vee b\in\mathbb{Z}_\infty^-\big)\implies b\dot\leq a$.
\item $1^{st}b\leq2^{nd}b\wedge2^{nd}a\leq1^{st}a\implies \big(a=0\vee a\in\mathbb{Z}_\infty^-\big)\wedge b\in\mathbb{Z}_\infty^+\implies a\dot\leq b$.
\item $2^{nd}a\leq1^{st}a\wedge2^{nd}b\leq1^{st}b\implies \big(a=0\vee a\in\mathbb{Z}_\infty^-\big)\wedge\big(b=0\vee b\in\mathbb{Z}_\infty^-\big)$.
\end{enumerate}
If $(1)$ holds, we observe that $2^{nd}a\leq2^{nd}b$ or $2^{nd}b\leq2^{nd}a$. If $2^{nd}a<2^{nd}b$ or $2^{nd}b<2^{nd}a$ then $a\dot\leq b$ or $b\dot\leq a$, respectively. If $2^{nd}a=2^{nd}b$ we observe that $1^{st}a\leq1^{st}b$ or $1^{st}b\leq1^{st}a$, thus $b\dot\leq a$ or $a\dot\leq b$, respectively. If $(4)$ holds and $a=(0,0)$ then $b\dot\leq a$, and if $b=(0,0)$ then $a\dot\leq b$. If $a\in\mathbb{Z}_\infty^-$ and $b\in\mathbb{Z}_\infty^-$ then the argument is symmetric to the first case. Since $a$ and $b$ were arbitrary in $\mathbb{Z}_\infty$, this completes the proof that $\dot\leq$ is a total ordering on $\mathbb{Z}_\infty$. \\
\end{proof}

We thusly begin to see the ordering properties of the integers emerging in $\mathbb{Z}_\infty$. We now define the expected additive and multiplicative identity, together with a convenient notational expansion, then proceed to define the binary operations we will be using in $\mathbb{Z}_\infty$. \\

\begin{customdefn}{9.13}
\begin{enumerate}
\item $\dot0=(0,0)$.
\item $\dot1=(0,1)$.
\item For all $\alpha$, we define $_+\alpha=(0,\alpha)$ and $_-\alpha=(\alpha,0)$.
\item For a coefficient sequence $\mu\in^n(\omega\times\{0\}\cup\{0\}\times\omega)$, we define $\{_\pm n_i\}=\{\mu_i\}=rng\mu$.
\item $\{\mp n_i\}=\{\pm n_i\}^{-1}$.
\item For all $(m,n)\in\omega\times\omega$, we define $$\overline{(m,n)}=\begin{cases}(0,n-m),&\ \text{if}\ m\leq n, \\(m-n,0),&\ \text{if}\ n<m.\end{cases}$$ \\
\end{enumerate}
\end{customdefn}

Note that part {\it (6)} of the preceding definition is equivalent to defining equivalence classes on $\omega\times\omega$ that partition it by the difference of coordinates, minding 'positive' and 'negative' differences, then choosing the element of each equivalence class with a zero coordinate.  This expression is simply more concise, and equivalent for our purposes. \\

\begin{customdefn}{9.14}
Suppose $\zeta\in^n(\{0\}\times O_n)$ and $\gamma\in^m(\{0\}\times O_n)$, with $a=\ddot\sum_{_{i<n}}\dot\omega^{\zeta_i}\ddot\times_\pm n'_i$ and \\ $b=\ddot\sum_{_{i<m}}\dot\omega^{\gamma_i}\ddot\times_\pm m'_i$ the surinteger normal forms of $a$ and $b$. We then define $\psi\in^\ell(\{0\}\times O_n)$ with \\ $n\cap m\leq\ell\leq m+n$ such that $rng\psi=rng\zeta\cup rng\gamma$, and $$a\hat+b=\ddot\sum_{_{i<\ell}}\dot\omega^{\psi_i}\ddot\times(\overline{n_i\ddot+m_i}),$$ where $$n_i=\begin{cases}n'_i,&\ \text{if}\ \psi_i\in rng\zeta \\ 0,&\ \text{if}\ \psi_i\notin rng\zeta,\end{cases}$$ and $$m_i=\begin{cases}m'_i,&\ \text{if}\ \psi_i\in rng\gamma \\ 0,&\ \text{if}\ \psi_i\notin rng\gamma.\end{cases}$$ We then define $$\hat+=\{\big((a,b),a\hat+b\big):a\in\mathbb{Z}_\infty\wedge b\in\mathbb{Z}_\infty\}.$$ $\hat+$ is called {\bf Surinteger addition}, or simply {\bf addition}.  $a\hat+b$ will be called the {\bf sum} of $a$ and $b$, and we will write $a\hat+b=c$ iff $\big((a,b),c\big)\in\hat+$. \\
\end{customdefn}

We thusly see that addition of surintegers is very closely analogous to Hessenberg addition of ordinals, except we may order the terms in the sum however we like since our base additive structure is commutative and we must use equivalence classes to 'select' the correct members of $O_n\times O_n$ that lie in $\mathbb{Z}_\infty$. \\

\begin{customcor}{9.15}
$\hat+$ is a function.
\end{customcor}
\begin{proof}
Let $a=\ddot\sum_{_{i<n}}\dot\omega^{\zeta_i}\ddot\times_\pm n'_i$ and $b=\ddot\sum_{_{i<m}}\dot\omega^{\gamma_i}\ddot\times_\pm m'_i$, and suppose $a\hat+b=c$ and $a\hat+b=d$. Then by definition {\it 9.14} we have that $c=\ddot\sum_{_{i<\ell}}\dot\omega^{\psi_i}\ddot\times(\overline{n_i\ddot+m_i})=d$. This completes the proof. \\
\end{proof}

We now prove two Lemmas which will be useful when proving that $\mathbb{Z}_\infty$ is a Ring. \\

\begin{customlem}{9.16}
For any $a,b,c\in(\omega\times\{0\}\cup\{0\}\times\omega)$, $$\overline{\big(a\ddot+\overline{(b\ddot+c)}\big)}=\overline{\big(\overline{(a\ddot+b)}\ddot+c\big)}.$$
\end{customlem}
\begin{proof}
Suppose $a,b,c\in(\omega\times\{0\}\cup\{0\}\times\omega)$. We first note that if $a,b,c\in\mathbb{Z}_\infty^+$ or $a,b,c\in\mathbb{Z}_\infty^-$, then this lemma holds trivially by the associativity of $+$ on ordinals. This gives us three remaining options to work with: 
\begin{enumerate} 
\item $a,b\in\mathbb{Z}_\infty^-\wedge c\in\mathbb{Z}_\infty^+$,
\item $a\in\mathbb{Z}_\infty^-\wedge b,c\in\mathbb{Z}_\infty^+$,
\item $b\in\mathbb{Z}_\infty^-\wedge a,c\in\mathbb{Z}_\infty^+$,
\end{enumerate}
observing that a proof the negation of the above three situations is symmetric to the a proof of the above. Suppose {\it (1)} holds; we then have that {\bf (1)} $1^{st}a+1^{st}b\leq2^{nd}c$ or {\bf (2)} $2^{nd}c\leq1^{st}a+1^{st}b.$ If {\bf (1)} holds, then $$\overline{\big(a\ddot+\overline{(b\ddot+c)}\big)}=\overline{a\ddot+(0,2^{nd}c-1^{st}b)}=\overline{(1^{st}a,2^{nd}c-1^{st}b)}=(0,2^{nd}c-1^{st}b-1^{st}a)=(0,2^{nd}c-1^{st}a-1^{st}b)$$ and $$\overline{\big(\overline{(a\ddot+b)}\ddot+c\big)}=\overline{(1^{st}a+1^{st}b,0)\ddot+c}=\overline{(1^{st}a+1^{st}b,2^{nd}c)}=\big(0,2^{nd}c-(1^{st}a+1^{st}b)\big)=(0,2^{nd}c-1^{st}a-1^{st}b),$$ thus $\overline{\big(a\ddot+\overline{(b\ddot+c)}\big)}=\overline{\big(\overline{(a\ddot+b)}\ddot+c\big)}$. If {\bf (2)} holds, then we first observe that $$\overline{\big(\overline{(a\ddot+b)}\ddot+c\big)}=\overline{(1^{st}a+1^{st}b,0)\ddot+c}=\overline{(1^{st}a+1^{st}b,2^{nd}c)}=(1^{st}a+1^{st}b-2^{nd}c,0).$$ Further, we have that {\bf (3)} $1^{st}b\leq2^{nd}c$ or {\bf (4)} $2^{nd}c\leq1^{st}b$. If {\bf (3)} holds, then $$\overline{\big(a\ddot+\overline{(b\ddot+c)}\big)}=\overline{a\ddot+(0,2^{nd}c-1^{st}b)}=\overline{(1^{st}a,2^{nd}c-1^{st}b)}=(1^{st}a-(2^{nd}c-1^{st}b),0)=(1^{st}a+1^{st}b-2^{nd}c,0),$$ thus $\overline{\big(a\ddot+\overline{(b\ddot+c)}\big)}=\overline{\big(\overline{(a\ddot+b)}\ddot+c\big)}$. If {\bf (4)} holds, then $$\overline{\big(a\ddot+\overline{(b\ddot+c)}\big)}=\overline{a\ddot+(1^{st}b-2^{nd}c,0)}=(1^{st}a+1^{st}b-2^{nd}c,0),$$ thus $\overline{\big(a\ddot+\overline{(b\ddot+c)}\big)}=\overline{\big(\overline{(a\ddot+b)}\ddot+c\big)}$ again. This completes the proof if ${\it (1)}$ above holds. \\ A proof of the case if {\it (2)} holds is symmetric to a proof of the case if {\it (1)} holds. \\ Now suppose {\it (3)} holds; then we have four possible options to work with, namely {\bf (5)} $2^{nd}a\leq1^{st}b$ and $2^{nd}c\leq1^{st}b$ or {\bf (6)} $1^{st}b\leq2^{nd}a$ and $2^{nd}c\leq1^{st}b$ and their negations. \\  If {\bf (5)} holds, then $$\overline{\big(a\ddot+\overline{(b\ddot+c)}\big)}=\overline{a\ddot+(1^{st}b-2^{nd}c,0)},$$ $$\overline{\big(\overline{(a\ddot+b)}\ddot+c\big)}=\overline{(1^{st}b-2^{nd}a,0)\ddot+c},$$ leaving us with two more options; {\bf (7)} $2^{nd}a+2^{nd}c\leq1^{st}b$ or {\bf (8)} $1^{st}b\leq2^{nd}a+2^{nd}c$. If {\bf (7)} holds, then $$\overline{\big(a\ddot+\overline{(b\ddot+c)}\big)}=\overline{a\ddot+(1^{st}b-2^{nd}c,0)}=(1^{st}b-2^{nd}a-2^{nd}c,0),$$ $$\overline{\big(\overline{(a\ddot+b)}\ddot+c\big)}=\overline{(1^{st}b-2^{nd}a,0)\ddot+c}=(1^{st}b-2^{nd}a-2^{nd}c,0),$$ thus $\overline{\big(a\ddot+\overline{(b\ddot+c)}\big)}=\overline{\big(\overline{(a\ddot+b)}\ddot+c\big)}$. If {\bf (8)} holds, then $$\overline{\big(\overline{(a\ddot+b)}\ddot+c\big)}=\overline{(1^{st}b-2^{nd}a,0)\ddot+c}=(0,2^{nd}a+2^{nd}c-1^{st}b),$$ $$\overline{\big(\overline{(a\ddot+b)}\ddot+c\big)}=\overline{(1^{st}b-2^{nd}a,0)\ddot+c}=(0,2^{nd}a+2^{nd}c-1^{st}b),$$ thus $\overline{\big(a\ddot+\overline{(b\ddot+c)}\big)}=\overline{\big(\overline{(a\ddot+b)}\ddot+c\big)}$ once again; this completes the proof if {\bf (5)} holds. A proof of the negation of {\bf (5)} is enirely symmetric. \\ Now suppose {\bf (6)} holds. Then $$\overline{\big(a\ddot+\overline{(b\ddot+c)}\big)}=\overline{a\ddot+(1^{st}b-2^{nd}c,0)}=(0,2^{nd}a+2^{nd}c-1^{st}b),$$ $$\overline{\big(\overline{(a\ddot+b)}\ddot+c\big)}=\overline{(0,2^{nd}a-1^{st}b)\ddot+c}=(0,2^{nd}a+2^{nd}c-1^{st}b),$$ thus $\overline{\big(a\ddot+\overline{(b\ddot+c)}\big)}=\overline{\big(\overline{(a\ddot+b)}\ddot+c\big)}$; this completes the proof if {\bf (6)} holds. A proof of the negation of {\bf (6)} is enirely symmetric. This completes the proof if {\it (3)} above holds, completing the proof overall. \\
\end{proof}

\begin{customlem}{9.17}
$\mathbb{Z}_\infty^+$ and $\mathbb{Z}_\infty^-$ are closed under $\hat+$.
\end{customlem}
\begin{proof}
Suppose $a,b\in\mathbb{Z}_\infty^+$, with $a=\ddot\sum_{i<n}\dot\omega^{\zeta_i}\ddot\times_\pm n_i$ and $b=\ddot\sum_{i<m}\dot\omega^{\gamma_i}\ddot\times_\pm m_i$. Then $a^\uparrow\in\mathbb{Z}_\infty^+$ and $b^\uparrow\in\mathbb{Z}_\infty^+$ by Corollary {\it 9.11}, and we then have that $$(a\hat+ b)^\uparrow=\begin{cases}a^\uparrow\ddot+b^\uparrow,&\ \text{if}\ \zeta^\uparrow=\gamma^\uparrow, \\ a^\uparrow,&\ \text{if}\ \gamma^\uparrow<\zeta^\uparrow, \\ b^\uparrow,&\ \text{if}\ \zeta^\uparrow<\gamma^\uparrow.\end{cases}$$ In any of the above three cases we have that $(a\hat+b)^\uparrow\in\mathbb{Z}_\infty^+$, keeping in mind that $$a^\uparrow\ddot+ b^\uparrow=(0,2^{nd}a^\uparrow)\ddot+(0,2^{nd}b^\uparrow)=(0,2^{nd}a^\uparrow+2^{nd}b^\uparrow)$$ by Lemma {\it 9.11}, thus $(a\hat+b)\in\mathbb{Z}_\infty^+$. The proof if $a,b\in\mathbb{Z}_\infty^-$ is entirely symmetric.  This completes the proof. \\
\end{proof}

We now define a notion of negation in $\mathbb{Z}_\infty$, something which will have no analogue in the ordinals. The process of negating a surinteger has two completely equivalent interpretations; we will first present the more concise of the two options. \\

\begin{customdefn}{9.18}
$$\hat-=\{(a,b):(2^{nd}a,1^{st}a)=b\}.$$ $\hat-$ will be called {\bf Surinteger negation}, or simply {\bf negation}. We will write $b=-a$ iff $(a,b)\in\hat-$. \\
\end{customdefn}

Thus we see that negating a surinteger amounts to switching its second and first coordinate. It is worth noting that this is entirely equivalent to switching the sign on all coefficients in the normal form expansion of a surinteger; we make this explicit with the next corollary. \\

\begin{customcor}{9.19}
$\hat-=\{(a,b):a=\ddot\sum_{_{i<n}}\dot\omega^{\zeta_i}\ddot\times_\pm n_i\wedge b=\ddot\sum_{_{i<n}}\dot\omega^{\zeta_i}\ddot\times_\mp n_i\}$. \\
\end{customcor}
\begin{proof}
This follows immedately as a consequence of Corollary {\it 9.5}. \\
\end{proof}

\begin{customcor}{9.20}
$\hat-$ is a function.
\end{customcor}
\begin{proof}
Suppose $(a,b)\in\hat-$ and $(a,c)\in\hat-$. Then by definition {\it 9.18} we have $$b=(2^{nd}a,1^{st}a)=c,$$ completing the proof. \\
\end{proof}

We now define the multiplicative structure in $\mathbb{Z}_\infty$. \\

\begin{customdefn}{9.21}
Let $a=\ddot\sum_{_{i<n}}\dot\omega^{\zeta_i}\ddot\times_\pm n_i$ and $b=\ddot\sum_{_{i<m}}\dot\omega^{\gamma_i}\ddot\times_\pm m_i$. We then define $\mathscr{A}=\{\zeta_i\ddot+\gamma_j\}$, and for each $a\in\mathscr{A}$ with $\zeta_i\ddot+\gamma_j=a,\dots,\zeta_p\ddot+\gamma_q=a$, let $$c_a=\overline{n_i\ddot\times m_j\ddot+\dots\ddot+n_p\ddot\times m_q}.$$ We then define a sequence $\alpha\in^\ell\mathscr{A}$ for some $\ell\in\omega$, and $$a\hat\times b=\ddot\sum_{i<\ell}\dot\omega^{\alpha_i}\ddot\times c_{\alpha_i}.$$ We then define $$\hat\times=\{\big((a,b),a\hat\times b\big):a\in\mathbb{Z}_\infty\wedge b\in\mathbb{Z}_\infty\}.$$ $\hat\times$ will be called {\bf Surinteger multiplication}, or simply {\bf multiplication}. We will write $a\hat\times b=c$ iff $\big((a,b),c\big)\in\hat\times$. \\
\end{customdefn}

Thus we see that the multiplicative structure in $\mathbb{Z}_\infty$ behaves essentially like polynomial multiplication with $\omega$ as the lone independent variable. \\

\begin{customcor}{9.22}
$\hat\times$ is a function.
\end{customcor}
\begin{proof}
Let $a=\ddot\sum_{_{i<n}}\dot\omega^{\zeta_i}\ddot\times_\pm n_i$ and $b=\ddot\sum_{_{j<m}}\dot\omega^{\gamma_j}\ddot\times_\pm m_j$, and suppose that $a\hat\times b=c$ and $a\hat\times b=d$. Also, let $\mathscr{A}$, $c_a$ and $\alpha$ be defined as in definition {\it 9.21}. By definition {\it 9.21} again, we then have that $$c=\ddot\sum_{i<\ell}\dot\omega^{\alpha_i}\ddot\times c_{\alpha_i}=d,$$ completing the proof. \\
\end{proof}

We now prove the last three theorems necessary prior to establishing that $\mathbb{Z}_\infty$ is a Ring. \\

\begin{customlem}{9.23}
For any $a,b,c\in(\omega\times\{0\}\cup\{0\}\times\omega)$, 
\begin{enumerate}
\item $\Big(a\in\mathbb{Z}_\infty^+\wedge b\in\mathbb{Z}_\infty^+\wedge c\in\mathbb{Z}_\infty^-\Big)\implies\Big(a\ddot\times b=(0,2^{nd}a\times2^{nd}b)\in\mathbb{Z}_\infty^+\wedge a\ddot\times c=(2^{nd}a\times1^{st}c,0)\in\mathbb{Z}_\infty^-\Big)$.
\item $a\ddot\times\overline{(b\ddot+c)}=\overline{(a\ddot\times b)\ddot+(a\ddot\times c)}$.
\end{enumerate}
\end{customlem}
\begin{proof}
For {\it (1)}, we simply observe that if $a=(0,\alpha)$, $b=(0,\beta)$, and $c=(\gamma,0)$, then $$a\ddot\times b=(0\times\beta+\alpha\times 0,0\times0+\alpha\times\beta)=(0,\alpha\times\beta)\in\mathbb{Z}_\infty^+,$$ and $$a\ddot\times c=(0\ddot\times0+\alpha\times\gamma,0\times\gamma+\alpha\times0)=(\alpha\times\gamma,0)\in\mathbb{Z}_\infty^-.$$ For {\it (2)}, we first suppose that $b,c\in\mathbb{Z}_\infty^+$ or $b,c\in\mathbb{Z}_\infty^-$. Then $b\ddot+c=(0,2^{nd}b+2^{nd}c)$ or $b\ddot+c=(1^{st}b+1^{st}c,0)$ respectively, thus $\overline{b\ddot+c}=b\ddot+c$ in either case. Consequently, if $a=(0,\alpha)$, then by part {\it (1)} we have $$a\ddot\times\overline{(b\ddot+c)}=a\ddot\times(b\ddot+c)=\big(0,\alpha\times(2^{nd}b+2^{nd}c)\big)=(0,\alpha\times2^{nd}b+\alpha\times2^{nd}c)$$ or $$a\ddot\times\overline{(b\ddot+c)}=a\ddot\times(b\ddot+c)=\big(\alpha\times(1^{st}b+1^{st}c),0\big)=(\alpha\times1^{st}b+\alpha\times1^{st}c,0)$$ respectively. Likewise, by part {\it (1)} we have that $a\ddot\times b=(0,\alpha\times2^{nd}b)$ and $a\ddot\times c=(0,\alpha\times2^{nd}c)$ or $a\ddot\times b=(\alpha\times1^{st}b,0)$ and $a\ddot\times c=(\alpha\times1^{st}c,0)$ respectively, thus $\overline{(a\ddot\times b)\ddot+(a\ddot\times c)}=(a\ddot\times b)\ddot+(a\ddot\times c)$ in either case. Consequently, by part {\it (1)} we have $$\overline{(a\ddot\times b)\ddot+(a\ddot\times c)}=(a\ddot\times b)\ddot+(a\ddot\times c)=(0,\alpha\times2^{nd}b)\ddot+(0,\alpha\times2^{nd}c)=(0,\alpha\times2^{nd}b+\alpha\times2^{nd}c)$$ or $$\overline{(a\ddot\times b)\ddot+(a\ddot\times c)}=(a\ddot\times b)\ddot+(a\ddot\times c)=(\alpha\times1^{st}b,0)\ddot+(\alpha\times1^{st}c,0)=(\alpha\times1^{st}b+\alpha\times1^{st}c,0)$$ respectively, thus $a\ddot\times\overline{(b\ddot+c)}=\overline{(a\ddot\times b)\ddot+(a\ddot\times c)}$. The proof if $a=(\alpha,0)$ is entirely symmetric. \\ Now suppose $b\in\mathbb{Z}_\infty^+$ and $c\in\mathbb{Z}_\infty^-$, and let $a=(0,\alpha)$. We then have that $1^{st}c\leq2^{nd}b$ or $2^{nd}b\leq1^{st}c$, thus $$a\ddot\times\overline{(b\ddot+c)}=a\ddot\times(0,2^{nd}b-1^{st}c)=\big(0,\alpha\times(2^{nd}b-1^{st}c)\big)=(0,2^{nd}b\times\alpha-1^{st}c\times\alpha)$$ or $$a\ddot\times\overline{(b\ddot+c)}=a\ddot\times(1^{st}c-2^{nd}b,0)=\big((1^{st}c-2^{nd}b)\times\alpha,0\big)=(1^{st}c\times\alpha-2^{nd}b\times\alpha,0)$$ respectively. Further, we observe that since $\times$ respects $\leq$ we have that $\alpha\times1^{st}c\leq\alpha\times2^{nd}b$ or $\alpha\times2^{nd}b\leq\alpha\times1^{st}c$ respectively, and consequently that $$\overline{(a\ddot\times b)\ddot+(a\ddot\times c)}=\overline{(\alpha\times1^{st}c,0)\ddot+(0,\alpha\times2^{nd}b)}=\overline{(\alpha\times1^{st}c,\alpha\times2^{nd}b)}=(0,2^{nd}b\times\alpha-1^{st}c\times\alpha)$$ or $$\overline{(a\ddot\times b)\ddot+(a\ddot\times c)}=\overline{(\alpha\times1^{st}c,0)\ddot+(0,\alpha\times2^{nd}b)}=\overline{(\alpha\times1^{st}c,\alpha\times2^{nd}b)}=(1^{st}c\times\alpha-2^{nd}b\times\alpha,0)$$ respectively, thus $a\ddot\times\overline{(b\ddot+c)}=\overline{(a\ddot\times b)\ddot+(a\ddot\times c)}$ in either case. The proof if $a=(\alpha,0)$ is entirely symmetric, as is the proof if $b\in\mathbb{Z}_\infty^-$ and $c\in\mathbb{Z}_\infty^+$.  This completes the proof. \\
\end{proof}

\begin{customdefn}{9.24}
For any $a,b\in\mathbb{Z}_\infty$ such that $a\neq b$, we define $$m_{a,b}^{\uparrow}=\inf\{m:a^{\uparrow-m}\neq b^{\uparrow-m}\}.$$ Since $a\neq b$, $\inf\{m:a^{\uparrow-m}\neq b^{\uparrow-m}\}$ is not empty and consequently $m_{a,b}^\uparrow$ is well defined. \\
\end{customdefn}

\begin{customcor}{9.25}
For all $a,b\in\mathbb{Z}_\infty$ such that $a\neq b$,
\begin{enumerate}
\item $m_{a,b}^\uparrow=0\iff a^\uparrow\neq b^\uparrow$.
\item $m_{a,b}^\uparrow>0\iff a^\uparrow=b^\uparrow$.
\item $m_{a,b}^\uparrow=n\iff[a^{\uparrow-n}\neq b^{\uparrow-n}\wedge\forall i(i<n\implies a^{\uparrow-i}=b^{\uparrow-i})]$.
\end{enumerate}
\end{customcor}
\begin{proof}
These follow immedately from Definition {\it 9.24}, keeping in mind that the last part of {\it (3)} holds vacuously if $n=0$. \\
\end{proof}

We thusly see that for any two $a,b\in\mathbb{Z}_\infty$ such that $a\neq b$, $m_{a,b}^\uparrow$ is a measure of the highest order term that is not the same in the normal forms of $a$ and $b$. \\

\begin{customlem}{9.26}
For any $a,b\in\mathbb{Z}_\infty$, $$a\dot<b\iff a^{\uparrow-m_{a,b}^\uparrow}\dot<b^{\uparrow-m_{a,b}^\uparrow}.$$
\end{customlem}
\begin{proof}
Suppose $a\dot<b$. If $a\in\mathbb{Z}_\infty^-$ and $b\in\mathbb{Z}_\infty^+$ then $a^\uparrow\neq b^\uparrow$ so $m_{a,b}^\uparrow=0$, and $a^\uparrow\in\mathbb{Z}_\infty^-$ and $b^\uparrow\in\mathbb{Z}_\infty^+$ by Corollary {\it 9.11 (1)}, thus $a^\uparrow\dot<b^\uparrow$ as desired. If $a,b\in\mathbb{Z}_\infty^+$, then $2^{nd}a<2^{nd}b$ or $2^{nd}a=2^{nd}b$ and $1^{st}b<1^{st}a$, and $a^\uparrow,b^\uparrow\in\mathbb{Z}_\infty^+$ by Corollary {\it 9.11 (1)}. If $2^{nd}a<2^{nd}b$, then by Corollary {\it 9.5} we have that $m_{a,b}^\uparrow=0$ and $a^\uparrow\dot<b^\uparrow$ once again. If $2^{nd}a=2^{nd}b$ and $1^{st}b<1^{st}a$, then again by Corollary {\it 9.11} we have that $m_{a,b}^\uparrow>0$ and $a^\uparrow=b^\uparrow$, and by Definition {\it 9.8} we have that $a^{\uparrow-m_{a,b}^\uparrow}\dot<b^{\uparrow-m_{a,b}^\uparrow}$, completing the proof if $a,b\in\mathbb{Z}_\infty^+$. The proof if $a,b\in\mathbb{Z}_\infty^-$ is entirely symmetric, completing the proof one way. \\ Now, suppose $a^{\uparrow-m_{a,b}^\uparrow}\dot<b^{\uparrow-m_{a,b}^\uparrow}$ and $m_{a,b}^\uparrow=0$. We then have that if $a^{\uparrow-m_{a,b}^\uparrow}\in\mathbb{Z}_\infty^-$ and $b^{\uparrow-m_{a,b}^\uparrow}\in\mathbb{Z}_\infty^+$, then $a\in\mathbb{Z}_\infty^-$ and $b\in\mathbb{Z}_\infty^+$ by Corollary {\it 9.11 (1)}, thus $a\dot<b$ as desired. If $a^{\uparrow-m_{a,b}^\uparrow},b^{\uparrow-m_{a,b}^\uparrow}\in\mathbb{Z}_\infty^+$, then by Corollary {\it 9.11 (1)} again we have that $a,b\in\mathbb{Z}_\infty^+$, and by Corollary {\it 9.5} we have that $2^{nd}a<2^{nd}b$, thus $a\dot<b$ again. The proof if $a^{\uparrow-m_{a,b}^\uparrow},b^{\uparrow-m_{a,b}^\uparrow}\in\mathbb{Z}_\infty^-$ is entirely symmetric, completing the proof if $m_{a,b}^\uparrow=0$. Now suppose $m_{a,b}^\uparrow>0$. By Corollary {\it 9.11} we have that $a,b\in\mathbb{Z}_\infty^+$ or $a,b\in\mathbb{Z}_\infty^-$, observing that the other options give us a contradiction by Corollary {\it 9.11 (1)} and {\it (3)}, since $m_{a,b}^\uparrow>0\Rightarrow a^\uparrow=b^\uparrow$ by Corollary {\it 9.25 (3)}. If $a,b\in\mathbb{Z}_\infty^+$, then by Corollary {\it 9.5} we have that $2^{nd}a<2^{nd}b$, thus $a\dot<b$ as desired.  The proof if $a,b\in\mathbb{Z}_\infty^-$ is entirely symmetric, completing the proof if $m_{a,b}^\uparrow>0$. This completes the proof overall.  \\
\end{proof}

We have now established all of the essential components of $\mathbb{Z}_\infty$ in a convienent form.  The next definitions will provide us with an exact definition for a ring in the context of MK class theory. \\

\begin{customdefn}{9.27}
A set $\mathfrak{R}$ is a {\bf commutative ordered ring with identity} iff $\mathfrak{R}=<\mathbb{R},\preceq_\mathbb{R},+_\mathbb{R},-_\mathbb{R},\times_\mathbb{R},0_\mathbb{R},1_\mathbb{R}>$ is a 7-termed sequence such that $\preceq_\mathbb{R}$ is a total ordering on $\mathbb{R}$, $+_\mathbb{R}:\mathbb{R}\times\mathbb{R}\rightarrow\mathbb{R}$, $-_\mathbb{R}:\mathbb{R}\rightarrow\mathbb{R}$, and $\times_\mathbb{R}:\mathbb{R}\times\mathbb{R}\rightarrow\mathbb{R}$ are all functions, and for all $a,b,c\in\mathbb{R}$:
\begin{enumerate}
\item $(a+b)+c=a+(b+c)$.
\item $a+b=b+a$.
\item $a+0=a$.
\item $a+(-a)=0$.
\item $(a\times b)\times c=a\times(b\times c)$.
\item $a\times b=b\times a$.
\item $a\times 1 =a$.
\item $a\times(b+c)=a\times b+a\times c$.
\item If $a\prec b$ and $c\prec d$, then $a+c\prec b+d$.
\item If $0\prec a$ and $0\prec b$, then $0\prec ab$.
\end{enumerate}
Note that $\preceq,+,-,\times,1$ and $0$ in {\it1-10} above all represent $\preceq_\mathbb{R},+_\mathbb{R},-_\mathbb{R},\times_\mathbb{R},1_\mathbb{R}$ and $0_\mathbb{R}$, respectively. We may refer to $\mathfrak{R}$ as simply a {\bf ring} when the other distinctions are not useful. \\
\end{customdefn}

\begin{customdefn}{9.28}
A proper class $\mathfrak{R}$ is a {\bf commutative ordered Ring with identity} iff \\ $\mathfrak{R}=\mathbb{R}\ \cup\preceq_\mathbb{R}\cup+_\mathbb{R}\cup-_\mathbb{R}\cup\times_\mathbb{R}$, $\preceq_\mathbb{R}$ is a total ordering on $\mathbb{R}$, $+_\mathbb{R}:\mathbb{R}\times\mathbb{R}\rightarrow\mathbb{R}$, $-_\mathbb{R}:\mathbb{R}\rightarrow\mathbb{R}$, and $\times_\mathbb{R}:\mathbb{R}\times\mathbb{R}\rightarrow\mathbb{R}$ are all functions, and there exist $0_\mathbb{R}\in\mathbb{R}$ and $1_\mathbb{R}\in\mathbb{R}	$ such that for all $a,b,c\in\mathbb{R}$:
\begin{enumerate}
\item $(a+b)+c=a+(b+c)$.
\item $a+b=b+a$.
\item $a+0=a$.
\item $a+(-a)=0$.
\item $(a\times b)\times c=a\times(b\times c)$.
\item $a\times b=b\times a$.
\item $a\times 1 =a$.
\item $a\times(b+c)=a\times b+a\times c$.
\item If $a\prec b$ and $c\prec d$, then $a+c\prec b+d$.
\item If $0\prec a$ and $0\prec b$, then $0\prec ab$.
\end{enumerate}
Note that $\preceq,+,-,\times,1$ and $0$ in {\it1-10} above all represent $\preceq_\mathbb{R},+_\mathbb{R},-_\mathbb{R},\times_\mathbb{R},1_\mathbb{R}$ and $0_\mathbb{R}$, respectively.  We may refer to $\mathfrak{R}$ as simply a {\bf Ring} when the other distinctions are not useful. \\
\end{customdefn}

Accordingly, we will use 'ring' to denote a commutative ordered ring with identity that is a set, and 'Ring' to denote the same structure in proper class form. Further, we see that a ring may be 'separated out' into sets containing each piece of the structure using a sequence over the ordinal $7$ so that we may consider each piece of its structure separately in a straightforward fashion, while in order to legitimately meld together all of the necessary pieces of a Ring we must use the notion of a union to pile all of the necessary components in one class without putting the separate classes containing those components in any other class. We could not use a sequence for a Ring, since being a term in a sequence forces a class to be a set and would thusly give us contradictions. Note that for a ring we have $\mathbb{R}\in\mathfrak{R}$, while for a Ring we have $\mathbb{R}\subset\mathfrak{R}$. \\

\begin{customdefn}{9.29}
$\mathcal{Z}_\infty=\mathbb{Z}_\infty\cup\dot\leq\cup\hat+\cup\hat-\cup\hat\times$. $\mathcal{Z}_\infty$ will be called the {\bf Surintegers$_+$}. \\
\end{customdefn}

Thus wee see that the Surintegers$_+$ simply amount to the Surintegers equipped with their ordering and binary operations.  We now prove that this does indeed form a ring. \\

\begin{customthm}{9.30 -- Main Theorem 1}
$\mathcal{Z}_\infty$ is a Ring, with $0_\mathcal{Z}=\dot0$ and $1_\mathcal{Z}=\dot1$.
\end{customthm}
\begin{proof}
Suppose $\zeta\in^n(\{0\}\times O_n)$, $\gamma\in^m(\{0\}\times O_n)$, and $\psi\in^\ell(\{0\}\times O_n)$, and let $a=\ddot\sum_{_{i<n}}\dot\omega^{\zeta_i}\ddot\times_\pm n'_i$, \\ $b=\ddot\sum_{_{j<m}}\dot\omega^{\gamma_j}\ddot\times_\pm m'_j$, and $c=\ddot\sum_{_{k<\ell}}\dot\omega^{\psi_k}\ddot\times_\pm \ell'_k$. We then define $\rho\in^q(\{0\}\times O_n)$ with $n\cap m\leq q\leq n+m$, and $\phi\in^r(\{0\}\times O_n)$ with $m\cap\ell\leq r\leq m+\ell$, such that $rng\rho=rng\zeta\cup rng\gamma$ and $rng\phi=rng\gamma\cup rng\psi$. Finally, we define $\delta\in^s(\{0\}\times O_n)$ with $q\cap r\leq s\leq m+n+\ell$ such that $rng\delta=rng\rho\cup rng\phi$. \\ For {\it (1)}, we observe that $$a\hat+(b\hat+c)=a\hat+\Big(\ddot\sum_{_{i<r}}\dot\omega^{\phi_i}\ddot\times(\overline{m_i\ddot+\ell_i})\Big)=\ddot\sum_{_{i<s}}\dot\omega^{\delta_i}\ddot\times\big(\overline{n_i\ddot+\overline{(m_i\ddot+\ell_i)}}\big)$$ $$=\ddot\sum_{_{i<s}}\dot\omega^{\delta_i}\ddot\times\big(\overline{\overline{(n_i\ddot+m_i)}\ddot+\ell_i}\big)=\Big(\ddot\sum_{_{i<q}}\dot\omega^{\rho_i}\ddot\times(\overline{n_i\ddot+m_i})\Big)\hat+c=(a\hat+b)\hat+c,$$ where the third equality follows from Lemma {\it 9.14}. For {\it (2)}, we observe that $$a\hat+b=\ddot\sum_{_{i<q}}\dot\omega^{\rho_i}\ddot\times(\overline{n_i\ddot+m_i})=\ddot\sum_{_{i<q}}\dot\omega^{\rho_i}\ddot\times(\overline{m_i\ddot+n_i})=b\hat+a.$$ For {\it (3)}, we observe that the surinteger normal form of $\dot0$ is $\dot0=(0,0)$, thus $$a\hat+\dot0=\ddot\sum_{_{i<n}}\dot\omega^{\zeta_i}\ddot\times(_\pm n'_i\ddot+\dot0)=\ddot\sum_{_{i<n}}\dot\omega^{\zeta_i}\ddot\times_\pm n'_i=a.$$ For {\it (4)}, we observe that $$a\hat+(\hat-a)=\ddot\sum_{_{i<n}}\dot\omega^{\zeta_i}\ddot\times_\pm n'_i\hat+\ddot\sum_{_{i<n}}\dot\omega^{\zeta_i}\ddot\times_\mp n'_i=\ddot\sum_{_{i<n}}\dot\omega^{\zeta_i}\ddot\times(\overline{_\pm n'_i\ddot+_\mp n'_i})=\ddot\sum_{_{i<n}}\dot\omega^{\zeta_i}\ddot\times\dot0=\dot0.$$ This completes the proof that $\mathcal{Z}_\infty$ is an abelian group with identity under $\hat+$ and $\hat-$. \\ For {\it (5-8)}, as in Definition {\it 9.21} we let $\mathscr{A}=\{\zeta_i\ddot+\gamma_j\}$ and $\alpha\in^t\mathscr{A}$, $\mathscr{B}=\{\gamma_i\ddot+\psi_j\}$ and $\beta\in^u\mathscr{B}$, $\mathscr{C}=\{\zeta_i\ddot+\psi_j\}$ and $\pi\in^v\mathscr{C}$, and lastly $\mathscr{D}=\{\zeta_i\ddot+\gamma_j\ddot+\psi_k\}$ with $\varepsilon\in^w\mathscr{D}$, and $c_a$ be defined as in Definition {\it 9.21} for $\mathscr{A},\mathscr{B},\mathscr{C},$ and $\mathscr{D}$. We then observe that $$a\hat\times(b\hat\times c)=a\hat\times\ddot\sum_{i<u}\dot\omega^{\beta_i}\ddot\times c_{\beta_i}=\ddot\sum_{i<w}\dot\omega^{\varepsilon_i}\ddot\times c_{\varepsilon_i}=\ddot\sum_{i<t}\dot\omega^{\alpha_i}\ddot\times c_{\alpha_i}\hat\times c=(a\hat\times b)\hat\times c.$$ This completes the proof of {\it (5)}. \\ For {\it (6)}, we simply observe that $\{\zeta_i\ddot+\gamma_j\}=\{\gamma_j\ddot+\zeta_i\}$ and $\overline{(n_i\ddot+m_j)}=\overline{(m_j\ddot+n_i)}$ for all $i$ and $j$, thus $$a\hat\times b=\ddot\sum_{i<t}\dot\omega^{\alpha_i}\ddot\times c_{\alpha_i}=b\hat\times a.$$ This completes the proof of {\it (6)}. \\ For {\it (7)}, we observe that $\{\zeta_i\ddot+0\}=\{\zeta_i\}$ and $_\pm n_i\ddot\times\dot1=_\pm n_i$ for all $i$, thus $$a\hat\times\dot 1=\ddot\sum_{i<n}\dot\omega^{\zeta_i\ddot+\dot0}\ddot\times(_\pm n_i\ddot\times 1)=\ddot\sum_{i<n}\dot\omega^{\zeta_i}\ddot\times _\pm n_i=a.$$ This completes the proof of {\it (7)}, completing the proof that $\mathcal{Z}_\infty$ is a monoid under $\hat\times$. \\ For {\it (8)}, we first observe that $$a\hat\times(b\hat+c)=\Big(\ddot\sum_{i<n}\dot\omega^{\zeta_i}\ddot\times n'_i\Big)\hat\times\Big(\ddot\sum_{i<r}\dot\omega^{\phi_i}\ddot\times\overline{(m_i\ddot+\ell_i)}\Big),$$ and that $\{\zeta_i\ddot+\phi_j\}=\{\zeta_i\ddot+\gamma_j\}\cup\{\zeta_i\ddot+\psi_j\}=\mathscr{A}\cup\mathscr{C}$. Next, we define $\eta\in^f(\{0\}\times O_n)$ with $t\cap v\leq f\leq t+v$ such that $rng\eta=rng\alpha\cup rng\pi=\mathscr{A}\cup\mathscr{C}$, thus we have $$a\hat\times(b\hat+c)=\Big(\ddot\sum_{i<n}\dot\omega^{\zeta_i}\ddot\times n_i\Big)\hat\times\Big(\ddot\sum_{i<r}\dot\omega^{\phi_i}\ddot\times\overline{(m_i\ddot+\ell_i)}\Big)=\ddot\sum_{i<f}\dot\omega^{\eta_i}\ddot\times c_{\eta_i}.$$ Next, we observe that $$(a\hat\times b)\hat+(a\hat\times c)=\ddot\sum_{i<t}\dot\omega^{\alpha_i}\ddot\times c'_{\alpha_i}\hat+\ddot\sum_{i<v}\dot\omega^{\pi_i}\ddot\times c'_{\pi_i}=\ddot\sum_{i<f}\omega^{\eta_i}\ddot\times\overline{(c_{\alpha_i}\ddot+c_{\pi_i})},$$ thus a proof of {\it (8)} amounts to showing that $\overline{(c_{\alpha_i}\ddot+c_{\pi_i})}=c_{\eta_i}$. We first note that $$c_{\eta_i}=\overline{\ddot\sum_{i<n^*}n_i\ddot\times\overline{(m_i\ddot+\ell_i)}},$$ where we have re-indexed the elements of the ranges of $n,m,$ and $\ell$ as necessary, and $n^*\in\omega$. Further, we note that $$\overline{(c_{\alpha_i}\ddot+c_{\pi_i})}=\overline{\ddot\sum_{i<n^*}\overline{(n_i\ddot\times m_i)\ddot+(n_i\ddot\times\ell_i)}}=\overline{\ddot\sum_{i<n^*}n_i\ddot\times\overline{(m_i\ddot+\ell_i)}}$$ by lemma {\it 9.20 (2)}, thus $c_{\eta_i}=\overline{(c_{\alpha_i}\ddot+c_{\pi_i})}$ as desied. This completes the proof of {\it (8)}, completing the proof that $\mathbb{Z}_\infty$ is a Ring under $\hat+$, $\hat-$, and $\hat\times$. \\ For {\it (9)}, suppose $a\dot<b$ and $c\dot<d$. By Lemma {\it 9.26} we then have that $$a^{\uparrow-m_{a,b}^\uparrow}\dot<b^{\uparrow-m_{a,b}^\uparrow}$$ and $$c^{\uparrow-m_{c,d}^\uparrow}\dot<d^{\uparrow-m_{c,d}^\uparrow}.$$ Suppose $m_{a,b}^\uparrow=0=m_{c,d}^\uparrow$. We observe that since $\mathbb{Z}_\infty^+$ and $\mathbb{Z}^-$ are closed under $\hat+$ by Lemma {\it 9.17}, we have that by Lemma {\it 9.11 (1)} and {\it (3)} $$a\in\mathbb{Z}_\infty^-\wedge b\in\mathbb{Z}_\infty^+\wedge c\in\mathbb{Z}_\infty^-\wedge d\in\mathbb{Z}_\infty^+\implies a^\uparrow\in\mathbb{Z}_\infty^-\wedge b^\uparrow\in\mathbb{Z}_\infty^+\wedge c^\uparrow\in\mathbb{Z}_\infty^-\wedge d^\uparrow\in\mathbb{Z}_\infty^+$$ $$\implies (a^\uparrow\hat+c^\uparrow)\in\mathbb{Z}_\infty^-\wedge (b^\uparrow\hat+d^\uparrow)\in\mathbb{Z}_\infty^+\implies a^\uparrow\hat+c^\uparrow\dot<b^\uparrow\hat+d^\uparrow\implies a\hat+c\dot<b\hat+d.$$ Now suppose that $a,b\in\mathbb{Z}_\infty^+$ and $c,d\in\mathbb{Z}_\infty^+$. We then have that $2^{nd}a<2^{nd}b$ and $2^{nd}c<2^{nd}d$ since $m_{a,b}^\uparrow=0=m_{c,d}^\uparrow$ and $a\dot<b$ and $c\dot<d$, and consequently we have that $2^{nd}a+2^{nd}c<2^{nd}b+2^{nd}d$, thus $a\hat+c\dot<b\hat+d$. The proof if $a,b,c,d\in\mathbb{Z}_\infty^-$ is entirely symmetric. This completes the proof if $m_{a,b}^\uparrow=0=m_{c,d}^\uparrow$. Now, suppose $m_{a,b}^\uparrow=m>0$ and $m_{c,d}^\uparrow=n>0$, and suppose that $a,b,c,d\in\mathbb{Z}_\infty^+$. We then have that $$a^{\uparrow-m}=(0,2^{nd}a^{\uparrow-m})\dot<b^{\uparrow-m}=(0,2^{nd}b^{\uparrow-m})$$ and $$c^{\uparrow-n}=(0,2^{nd}c^{\uparrow-n})\dot<d^{\uparrow-n}=(0,2^{nd}d^{\uparrow-n})$$ by Corollary {\it 9.11 (2)} and {\it (4)} and Lemma {\it 9.26}, thus $$a^{\uparrow-m}\ddot+c^{\uparrow-n}=(0,2^{nd}a^{\uparrow-m}+2^{nd}c^{\uparrow-n})\dot<(0,2^{nd}b^{\uparrow-m}+2^{nd}d^{\uparrow-n})=b^{\uparrow-m}\ddot+d^{\uparrow-n}.$$ Further, we observe that $$m_{a\hat+c,b\hat+d}^\uparrow=m_{a,b}^\uparrow+m_{c,d}^\uparrow=m+n,$$ thus $$(a\hat+c)^{\uparrow-m_{a\hat+c,b\hat+d}^\uparrow}=a^{\uparrow-m}\ddot+c^{\uparrow-n}\dot<b^{\uparrow-m}\ddot+d^{\uparrow-n}=(b\hat+d)^{\uparrow-m_{a\ddot+c,b\ddot+d}^\uparrow},$$ and consequently by Lemma {\it 9.26} we have that $a\hat+c\dot<b\hat+d$. This completes the proof if $a,b,c,d\in\mathbb{Z}_\infty^+$; the proof if $a,b,c,d\in\mathbb{Z}_\infty^-$ is entirely symmetric. Since $m>0$ and $n>0$, we can't have that $a\in\mathbb{Z}_\infty^-$ and $b\in\mathbb{Z}_\infty^+$ since this requires that $m=0$ by Corollary {\it 9.25}; a similar comment applies to $c$ and $d$. This completes the proof of (9), completing the proof that $\mathcal{Z}_\infty$ is an ordered abelian group under $\hat+$ and $\dot\leq$. \\ For {\it (10)}, suppose $a,b\in\mathbb{Z}_\infty^+\sim\{\dot0\}$, so that by Corollary {\it 9.11 (1)} we have that $a^\uparrow,b^\uparrow\in\mathbb{Z}_\infty^+\sim\{0\}$. By Lemma {\it 9.11 (2)}, we then have that $$a^\uparrow\ddot\times b^\uparrow=(0,2^{nd}a^\uparrow)\ddot\times(0,2^{nd}b^\uparrow)=(0,2^{nd}a^\uparrow\times2^{nd}b^\uparrow),$$ thus $a^\uparrow\ddot\times b^\uparrow\in\mathbb{Z}_\infty^+\sim\{0\}$. Further, we observe that $$(a\hat\times b)^\uparrow=a^\uparrow\ddot\times b^\uparrow,$$ thus $(a\hat\times b)^\uparrow\in\mathbb{Z}_\infty^+\sim\{0\}$ and consequently by Corrolary {\it 9.11 (1)} again we have that $a\hat\times b\in\mathbb{Z}_\infty^+\sim\{0\}$, completing the proof of {\it (10)}. This completes the proof that $\mathcal{Z}_\infty$ is an ordered monoid under $\hat\times$ and $\dot\leq$, and all of these facts taken together prove that $\mathbb{Z}_\infty$ is an ordered r/Ring under $\hat+$, $\hat\times$, and $\dot\leq$. \\ To prove that $\mathbb{Z}_\infty$ is a proper class, we simply observe that $(\{0\}\times O_n)\subset\mathbb{Z}_\infty^+\subset\mathbb{Z}_\infty$, and $<\alpha:(0,\alpha)\in O_n>$ is a function with domain $\{0\}\times O_n$ and range $O_n$, thus if $\{0\}\times O_n$ was a set then $O_n$ would also be a set by the axiom of substitution, a contradiction. We thusly conclude that $\{0\}\times O_n$ is a proper class, thus $\mathbb{Z}_\infty^+$ and $\mathbb{Z}_\infty$ are also a proper classes. This completes the proof that $\mathbb{Z}_\infty$ is a Ring. \\
\end{proof}

Having shown rigorously that $\mathcal{Z}_\infty$ has all the structure of a Ring, we now take some notational convieniences common to ring-theoretic discussions for the sake of conciseness. We make these changes explicit with the next definition. \\

\begin{customdefn}{9.31}
From now on, lowercase italic letters from the beginning of the alphabet $a,b,c,\dots$ will denote surintegers, while $-a,-b,-c,\dots$ will denote their negations -- we will also write $\dot\omega$ as simply $\omega$, $\dot0$ as $0$, $\dot1$ as $1$, and $\dot\leq$ as $\leq$. We will use lowercase italic letters from the middle of the alphabet $i,j,k,l,m,n,...$ to denote finite surintegers. Further, we will write the product $a\hat\times b$ simply as $ab$, the sum $a\hat+b$ as $a+b$, and we will identify ordinals $\zeta$ in general with their corresponding surinteger $(0,\zeta)$. Correspondingly, sums of the form $\ddot\sum_{i<n}\dot\omega^{\zeta_i}\ddot\times_\pm n_i$ will be written as $\sum_{i<n}\omega^{\zeta_i}\times_\pm n_i$, where we may omit the $_\pm$ if it is not relevant, leaving $\sum_{i<n}\omega^{\zeta_i}\times n_i$. \\
\end{customdefn}
 
Accordingly, $i,j,k,l,m,n,\dots$ will denote integers, while $a,b,c,d,e,f,\dots$ will denote surintegers, in general.  If we run into a situation where the two alphabets would collide due to a profusion of variables, we will simply begin priming at the beginning of the list again. \\ We now define a unique cyclic subring of $\mathbb{Z}_\infty$ and identify it with the Integers. \\

\begin{customdefn}{9.32}
$$\mathbb{Z}=\mathbb{Z}_\infty\upharpoonleft\upharpoonright\omega=\{a:1^{st}a<\omega\wedge2^{nd}a<\omega\}.$$ $\mathbb{Z}$ will be called the {\bf Integers}. Further, we define
\begin{enumerate}
\item $\leq_\omega\ =\ \leq\upharpoonleft\upharpoonright\mathbb{Z}$.
\item $+_\omega=+\upharpoonleft(\mathbb{Z}\times\mathbb{Z})$.
\item $-_\omega=-\upharpoonleft\mathbb{Z}$.
\item $\times_\omega=\times\upharpoonleft(\mathbb{Z}\times\mathbb{Z})$.
\item $\mathcal{Z}=\langle\mathbb{Z},\leq_\omega,+_\omega,-_\omega,\times_\omega,0,1\rangle$. \\
\end{enumerate}
\end{customdefn}

We will now proceed to generalize the process used above to form the Integers in order to to define a proper class of non-isomorpic ordered rings of somewhat arbitrary transfinite order-type and cardinality. $\mathbb{Z}$ will be the smallest in a cardinality sense and the simplest in an order-type sense, while $\mathbb{Z}_\infty$ will be (by far) the largest and most complicated in cardinality and order-type senses. \\

\begin{customdefn}{9.33}
For any transfinite $\times-number$ $\lambda$ as defined in {\it 8.7}, we define 
\begin{enumerate}
\item $\mathbb{Z}_\lambda=\mathbb{Z}_\infty\upharpoonleft\upharpoonright\lambda$. $\mathbb{Z}_\lambda$ is called the {\bf $\lambda$-Integers}.
\item $\leq_\lambda=\ \leq\upharpoonleft\upharpoonright\mathbb{Z}_\lambda$.
\item $+_\lambda=+_\infty\upharpoonleft\mathbb{Z}_\lambda\times\mathbb{Z}_\lambda$.
\item $-_\lambda=-_\infty\upharpoonleft\mathbb{Z}_\lambda$.
\item $\times_\lambda=\times_\infty\upharpoonleft\mathbb{Z}_\lambda\times\mathbb{Z}_\lambda$.
\item $\mathcal{Z}_\lambda=\langle\mathbb{Z}_\lambda,\leq_\lambda,+_\lambda,-_\lambda,\times_\lambda,0,1\rangle$.
\item $\mathcal{O}_\mathcal{Z}=\{\mathcal{Z}_\lambda:\lambda \ \text{is a}\ \times\text{-number}\}$. \\
\end{enumerate}
\end{customdefn}

\begin{customthm}{9.34 -- Main Theorem 2}
For any transfinite $\times$-number $\lambda$:
\begin{enumerate}
\item $\mathbb{Z}_\lambda$ is a set.
\item $\preceq_\lambda$ is a total ordering on $\mathbb{Z}_\lambda$.
\item $+_\lambda$ is a function.
\item $-_\lambda$ is a function.
\item $\times_\lambda$ is a function.
\item $\mathcal{Z}_\lambda$ is a ring.
\item $\mathcal{O}_\mathcal{Z}$ is a proper class.
\end{enumerate}
\end{customthm}
\begin{proof}
To see that $\mathbb{Z}_\lambda$ is a set, we consider that $\mathbb{Z}_\lambda\subset(\lambda\times\lambda)$ as a consequence of Corollary {\it 9.7}, and $\lambda\times\lambda\subset\mathcal{P}\lambda$, and since $\lambda$ is a set by Theorem {\it 2.12} and $\mathcal{P}\lambda$ is then a set by the power set axiom, $\mathbb{Z}_\lambda$ is also a set. This completes the proof of {\it (1)}. {\it (2-5)} follow immedately from Theorems {\it 9.12}, {\it 9.15}, {\it 9.20}, and {\it 9.22}, with closure following from the fact that $\lambda$ is a $\times$-number together with Definition 8.6 and Theorem 8.7. {\it (6)} then immedately follows from {\it (1-5)} together with Main Theorem 1 and Definition {\it 9.27}. To prove {\it (7)}, we simply observe that every transfinite cardinal number is a $\times$-number, thus $\mathbb{H}=<\alpha:\mathbb{Z}_\alpha\in \mathcal{O}_\mathcal{Z}>$ is a function with $dmn\mathbb{H}=\mathcal{O}_\mathcal{Z}$ and $Card\subseteq rng\mathbb{H}$. Since $Card$ is a proper class, if $\mathcal{O}_\mathcal{Z}$ was a set we would have a contradiction by the axiom of substitution, since $rng\mathbb{H}$ would be a set, thus we conclude that $\mathcal{O}_\mathcal{Z}$ is a proper class.  This completes the proof.  \\
\end{proof}

Accordingly, we see that every $\mathbb{Z}_\lambda$ is a unique ring that is \underline{not} ring-isomorphic to any other $\mathbb{Z}_{\lambda'}$, with $\mathbb{Z}_\infty$ being the only Ring and consequently not a member of $\mathcal{O}_\mathcal{Z}$. Note that $\omega$ is the smallest transfinite $\times$-number, while by Theorem {\it 7.7 (3)} we have that the next option for 'closing out' a ring is at $\omega^\omega$, observing that $\mercury_2=\dot\times$, and $\omega\mercury_3\omega=\uparrow(\omega,\omega)=\omega^\omega$. The next will be $$\omega^\omega\mercury_3\omega=\bigcup_{n<\omega}\omega^\omega\mercury_3n=\omega^{\bigcup_{n<\omega}\omega\mercury_2 n}=\omega^{\bigcup_{n<\omega}\omega\dot\times n}=\omega^{\omega^2},$$ so on and so forth -- it is a well known result that, in general, all $\times$-numbers ($\mercury_2$-numbers) will be of the form $$\omega^{\omega^\zeta}=\omega\mercury_3(\omega\mercury_3\zeta),\ \zeta\in O_n,$$ and in fact all numbers of this form are $\times$-numbers. The generalization of this fact given in Theorem {\it 7.8} states that all $\mercury_n$-numbers will be of the form $$\omega\mercury_{n\dot+1}(\omega\mercury_{n\dot+1}\zeta),\ \zeta\in O_n,$$ and all numbers of this form are $\mercury_n$-numbers.  \\Prior to demonstrating the unique characteristics of $\mathcal{Z}$, we must first define the notion of a finite sum in a ring. \\

\begin{customdefn}{9.35}
Let $\mathfrak{R}=<\mathbb{R},\preceq_\mathbb{R},+_\mathbb{R},-_\mathbb{R},\times_\mathbb{R},0_\mathbb{R},1_\mathbb{R}>$ be a ring. If $\mu\in^\omega\mathbb{R}$, we define $$\Sigma(\mu,n)=\mu_0+\mu_1+\dots+\mu_{n-1}.$$ We will write $\Sigma_{i<n}\mu_i$ instead of $\Sigma(\mu,n)$. \\
\end{customdefn}

We now define some language common to ring-theoretic discussions of the Integers, and show that $\mathbb{Z}$ as defined is the only ordered ring with all of these properties. \\

\begin{customdefn}{9.36}
Let $\mathfrak{R}=<\mathbb{R},\preceq_\mathbb{R},+_\mathbb{R},-_\mathbb{R},\times_\mathbb{R},0_\mathbb{R},1_\mathbb{R}>$ be a ring, and let $x$ and $y$ denote members of $\mathbb{R}$.
\begin{enumerate}
\item $\mathfrak{R}$ is {\bf discrete} $\iff \forall x\nexists y\big[x<y<x+1\big]$.
\item $\mathfrak{R}$ is {\bf cyclic} $\iff \forall x\exists\mu\exists n\big[\mu\in^\omega\{1\}\wedge \big(x=\Sigma_{i<n}\mu_i\vee x=-\Sigma_{i<n}\mu_i\big)\big]$. \\
\end{enumerate}
\end{customdefn}

Thus we see that a discrete ring is one in which there are no elements ordered between a given element and its additive composition with the multiplicative identity in the ring, and a cyclic ring is one in which any element can be obtained as a finite sum of multiplicative identities that we may or may not negate. We now show that only $\mathcal{Z}$ satisfies both of these properties, despite the abundance of rings we can create. \\

\begin{customthm}{9.37}
\begin{enumerate}
\item $\mathcal{Z}_\lambda$ is a discrete ring for all $\lambda$.
\item $\mathcal{Z}$ is a cyclic ring.
\item $\mathcal{Z}_\lambda$ is a cyclic ring $\iff \lambda=\omega$.
\end{enumerate}
\end{customthm}
\begin{proof}
That $\mathcal{Z}_\lambda$ is discrete for all $\lambda$ follows immedately from Theorem 2.24 {\it (6)}. That $\mathcal{Z}$ is cyclic follows immedately from the Definition 6.2 and Definition 8.1, since for any natural number $n$ we have that $n=\sum_{i<n}1$ and $\mathbb{Z}$ is composed of all members of the form $(0,n)$ or $(n,0)$. To establish {\it (3)} suppose $\lambda\neq\omega$, so $\omega<\lambda$ since $\omega$ is the smallest transfinite $\times$-number. We thusly have that $(0,\omega)\in\mathbb{Z}_\lambda$, and $\omega$ is not equal to any finite sum of natural numbers, thus $(0,\omega)$ is not the result of any finite sum of $\dot1$s. This completes the proof. \\
\end{proof}

\begin{customthm}{9.38 -- Main Theorem 3}
$\mathcal{Z}$ is the only discrete cyclic ring up to isomorphism.
\end{customthm}
\begin{proof}
By Theorem 9.23, $\mathcal{Z}$ is a discrete cyclic ring. Suppose that $\mathfrak{R}=<\mathbb{R},\preceq_\mathbb{R},+_\mathbb{R},-_\mathbb{R},\times_\mathbb{R},0_\mathbb{R},1_\mathbb{R}>$ is also a discrete cyclic ring. \\ Since $\mathfrak{R}$ is cyclic by assumption, we have that $$r\in\mathbb{R}\implies\exists\mu\exists n\big[\mu\in^\omega\{1_\mathbb{R}\}\wedge\big(\Sigma_{i<n}\mu_i=r\vee -\Sigma_{i<n}\mu_i=r\big)\big].$$ Further, we have that $$z\in\mathbb{Z}\implies\exists\rho\exists n\big[\rho\in^\omega\{\dot1\}\wedge\big(\Sigma_{i<n}\rho_i=z\vee-\Sigma_{i<n}\rho_i=z\big)\big].$$ We now construct an isomorphism $\phi:\mathbb{Z}\rightarrow\mathbb{R}$ as follows. For each $n\in\omega$ with $\mu\in^n\{1_\mathbb{R}\}$ and $\rho\in^n\{\dot1\}$ the unique functions with domain $n$ and range $\{1_\mathbb{R}\}$ and $\{\dot1\}$ respectively, we define: $$\phi=<\Sigma_{i<n}\mu_i:\Sigma_{i<n}\rho_i\in\mathbb{Z}>\cup<-\Sigma_{i<n}\mu_i:-\Sigma_{i<n}\rho_i\in\mathbb{Z}>.$$ $\phi$ is clearly an order preserving ring isomorphism between $\mathcal{Z}$ and $\mathfrak{R}$; since $\mathfrak{R}$ was an arbitrary discrete cyclic ring, this completes the proof. \\
\end{proof}

Note that we could have equivalently asserted that $\mathcal{Z}$ is the only discrete ring whose set of positive elements is well-ordered, however the isomorphism was easier to construct taking this view. We have now constructed all of the ordered rings that we wish to use in our construction. \\ As a last note before we move on to constructing a field of fractions for each one of them, I would like to mention that we can construct sets with binary operations that are 'larger' than $+$, $\times$ and $\uparrow$ in $\mercury$, and then find closure points for those sets as a direct consequence of Theorem {\it 7.7}. I will not explore such structures in this paper as they have little to do with fields in the classical sense, however they strike me as potentially interesting objects to explore and I may do so in a subsequent paper. \\  It is also worth noting that we can close out additive groups at each $+$-number like $\omega^2$, $\omega^3$, $\dots$, however we can only close out rings at $\times$-numbers like $\omega^{\omega}$, $\omega^{\omega^2}$, $\dots$.

\vspace{40mm}

\section{The Surrational Numbers}

We now define a field of fractions for $\mathbb{Z}_\infty$, and consequently obtain a proper class of unique fields from each $\mathbb{Z}_\lambda\in\mathcal{O}_\mathcal{Z}$ -- note that any field built from some $\lambda<\omega^+$ will consequently be countable. \\ We begin by defining the notion of being prime and coprime in $\mathbb{Z}_\infty$. \\

\begin{customdefn}{10.1} 
For all $a\in\mathbb{Z}_\infty$,
\begin{itemize}
\item $\mathcal{F}_a=\{b:\exists c\big[bc=a\big]\}.$ $\mathcal{F}_a$ will be called the {\bf class of factors of $a$}, and the members $b\in\mathcal{F}_a$ will be called {\bf factors of $a$}. 
\item We will say that $a$ is {\bf prime} iff $\mathcal{F}(a)=\{1,a\}$.
\end{itemize}
For all $\psi\in\mathbb{Z}_\infty\times\mathbb{Z}_\infty$, 
\begin{itemize}
\item$\mathcal{F}_\psi=\mathcal{F}_{1^{st}\psi}\cap\mathcal{F}_{2^{nd}\psi}$. We will call $\mathcal{F}_\psi$ the {\bf shared factors of $\psi$}.
\item We will say that $\psi$ is {\bf coprime} iff $\mathcal{F}_\psi=\{1\}$. \\
\end{itemize}
\end{customdefn}

\begin{customdefn}{10.2}
From now on, $\psi$ will denote a member of $\mathbb{Z}_\infty\times\mathbb{Z}^+_\infty\sim\{0\}$.
$$\mathbb{Q}_\infty=\{\psi:\mathcal{F}_\psi=\{1\}\}.$$ $\mathbb{Q}_\infty$ will be called the {\bf Surrational numbers}, and lowercase italic letters from the middle of the alphabet $p,q,r,s,t,\dots$ will denote members of $\mathbb{Q}_{\infty}$, called {\bf surrational numbers}. If $p=(a,b)$, we may also write $p=\frac{a}{b}$ and refer to $a$ as the {\bf numerator} of $p$ and $b$ as the {\bf denominator} of $p$. \\
\end{customdefn}

There is once again a nice visual interpretation for this newly defined structure. We first note that the above definition essentially amounts to a class containing all ordered pairs of coprime Surintegers. We could not use the full equivalence classes typically used when constructing the Rationals out of the Integers, since each equivalence class is a proper class in $\mathbb{Z}_\infty$. \\ Recall from the previous section that moving from the Ordinals to the Surintegers amounted to flipping the 'Ordinal line' over $0$, then adding new positions behind each limit ordinal. These new positions approached old positions, however they never 'reached them' --  for example, $n<\omega-n$ for all $n\in\mathbb{Z}^+$. We denoted this approach with ellipse between the new and old positions, and we will reproduce this diagram below. \\
\setlength{\unitlength}{.1cm}
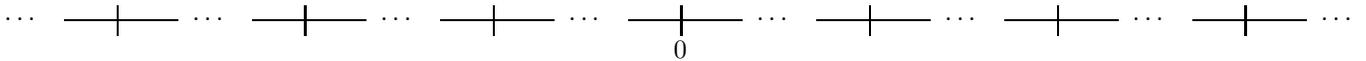
\begin{picture}(100,20)(50,55)
\multiput(53,65)(25,0){7}{\line(20,0){15}}
\multiput(60,63)(25,0){7}{\line(0,0){4}}
\multiput(45,65)(25,0){8}\dots
\put(134,60)0 \\
\end{picture}

What we have done by moving to $\mathbb{Q}_\infty$ amounts to creating fractional Surintegers, with the first coordinate playing the role of a numerator and the second coordinate playing the role of a non-zero denominator. Visually, this adds density to the line in much the same sense as the rationals typically 'add density' to the Integers, however it also creates a host of new positions that reside inside of each gap on the Surinteger line, which we represented with the ellipses above. For example, we will see shortly that the previous observation $n<\omega-n$ for all $n\in\mathbb{Z}^+$ now generalizes to $n<\frac{\omega}{m}<\omega-n$ for all $m>1$ and $n\in\mathbb{Z}^+$. This fractional behaviour is duplicated at $\omega^\omega$ in the sense that $n\omega<\frac{\omega^\omega}{m\omega}<\omega^\omega-n\omega$ for all $m>1$ and $n\in\mathbb{Z}^+$, and again at $\omega^{\omega^2}$, etc. We now define a positive and negative subclass of $\mathbb{Q}_\infty$. \\

\begin{customdefn}{10.3}
\begin{enumerate}
\item $\mathbb{Q}_\infty^+=\{q:0\leq1^{st}q\}$. $\mathbb{Q}_\infty^+$ will be called the {\bf positive} Surrationals.
\item $\mathbb{Q}_\infty^-=\{q:1^{st}q<0\}$. $\mathbb{Q}_\infty^-$ will be called the {\bf negative} Surrationals.
\item $\hat0=(\dot0,\dot1)=\frac{\dot0}{\dot1}$.
\item $\hat1=(\dot1,\dot1)=\frac{\dot1}{\dot1}$. \\
\end{enumerate}
\end{customdefn}

Note that although it may seem odd to define a notion of 'positive' and 'negative' prior to defining an ordering, we are doing this primarily for two reasons. Firstly it is more convienient when expressing the ordering on $\mathbb{Q}_\infty$ to do it this way, and it is equivalent to first defining the ordering in terms of a 'lifted' ordering using $\mathbb{Z}_\infty$ and $\times_\infty$ then defining the positive and negative Surrationals as being ordered above or below $\hat0$. Secondarily however, and perhaps more importantly for the purposes of this development, defining the ordering in this way should match all classical intuitions on how we order the Rationals, and we know that this ordering is (up to isomorphism) the only one that makes them an ordered ring.  We now present a few definitions and theorems that will elucidate the relationship between $\mathbb{Q}_\infty$ and $\mathbb{Z}_\infty$, and $O_n$. \\

\begin{customthm}{10.4} \
\begin{enumerate}
\item $\mathbb{Q}_\infty=\mathbb{Q}_\infty^+\cup\mathbb{Q}_\infty^-$.
\item $\mathbb{Q}_\infty\subset\Big(\mathbb{Z}_\infty\times\mathbb{Z}_\infty^+\sim\{0\}\Big)\subset\Big(O_n\times O_n\times O_n\times O_n\Big)$.
\end{enumerate}
\end{customthm}
\begin{proof}
{\it (1)} follows immedately from the fact that $\leq$ is a total ordering on $\mathbb{Z}_\infty$, thus for all $q$ we have either $0\leq1^{st}q$ or $1^{st}q<0$. \\ To see the first containment in {\it (2)}, we observe that $\mathbb{Q}_\infty\subseteq\Big(\mathbb{Z}_\infty\times\mathbb{Z}_\infty^+\sim\{0\}\Big)$ by Definition {\it 10.2}. We then observe that $(2,4)\in\Big(\mathbb{Z}_\infty\times\mathbb{Z}_\infty^+\sim\{0\}\Big)$, but $\mathcal{F}(2)\cap\mathcal{F}(4)=\{1,2\}$ since $2\times2=4$, thus $(2,4)\notin\mathbb{Q}_\infty$ (note however that $(1,2)\in\mathbb{Q}_\infty$). Consequently, $\mathbb{Q}_\infty\subset\Big(\mathbb{Z}_\infty\times\mathbb{Z}_\infty^+\sim\{0\}\Big)$ as desired. The second containment then follows immedately from Corollary {\it 9.7}. This completes the proof. \\
\end{proof}

\begin{customdefn}{10.5}
For all $\phi\in\mathbb{Z}_\infty\times\mathbb{Z}_\infty^+\sim\{0\}$, we define $$\overline{\underline{\phi}}=\{\psi:\exists a\exists b\big[a\times1^{st}\psi=b\times1^{st}\phi\wedge a\times2^{nd}\psi=b\times2^{nd}\phi\big]\}.$$ \\
\end{customdefn}

Thus we see that for each $\psi$, $\overline{\underline\psi}$ is the class containing all ordered pairs of surintegers whose first and second coordinates are some shared multiple away from the first and second coordinates of $\psi$. \\

\begin{customthm}{10.6}
For all $\phi\in\mathbb{Z}_\infty\times\mathbb{Z}_\infty^+\sim\{0\}$, 
\begin{enumerate}
\item $\phi\in\overline{\underline\phi}$.
\item $\psi\in\overline{\underline{\phi}}\implies\overline{\underline\psi}=\overline{\underline\phi}$.
\end{enumerate}
\end{customthm}
\begin{proof}
The first part follows immedately, observing that for all $\phi\in\mathbb{Z}_\infty\times\mathbb{Z}_\infty^+\sim\{0\}$ we have $1^{st}\phi\times1=1^{st}\phi\times1$ and $2^{nd}\phi\times1=2^{nd}\phi\times1$. \\  For {\it (2)}, suppose that $\psi\in\overline{\underline{\phi}}$ with $a\times1^{st}\psi=b\times1^{st}\phi$ and $a\times2^{nd}\psi=b\times2^{nd}\phi$. By the symmetry of Definition {\it 10.5}, we then immedately have that $\phi\in\overline{\underline\psi}$. Suppose $\zeta\in\overline{\underline\phi}$, with $c\times1^{st}\zeta=d\times1^{st}\phi$ and $c\times2^{nd}\zeta=d\times2^{nd}\phi$.  We then have that $$d\times a\times1^{st}\psi=d\times b\times1^{st}\phi=b\times d\times1^{st}\phi=b\times c\times1^{st}\zeta,$$ and $$d\times a\times2^{nd}\psi=d\times b\times2^{nd}\phi=b\times d\times2^{nd}\phi=b\times c\times2^{nd}\zeta,$$ thus $\zeta\in\overline{\underline\psi}$. Since $\zeta$ was arbitrary in $\overline{\underline\phi}$, this completes the proof that $\overline{\underline\phi}\subseteq\overline{\underline\psi}$. The proof if $\gamma\in\overline{\underline\psi}$ is entirely symmetric, thus $\overline{\underline\psi}\subseteq\overline{\underline\phi}$; consequently, $\overline{\underline\phi}=\overline{\underline\psi}$. This completes the proof. \\
\end{proof}

Although the classes $\overline{\underline\psi}$ are equivalence classes, we will not prove this formally. The following two theorems will elucidate all of the structure we care about for each $\overline{\underline\psi}$. \\

\begin{customthm}{10.7}
For all $\psi$, there exists exactly one $q$ such that $q\in\overline{\underline \psi}$.
\end{customthm}
\begin{proof}
We first note that for all $a\in\mathbb{Z}_\infty$, $1\in\mathcal{F}_a$. If $\mathcal{F}_\psi=\{1\}$, then $\psi\in\mathbb{Q}_\infty$ and we are satisfied. Suppose $\{1\}\neq\mathcal{F}_\psi\subset\mathbb{Z}_\infty$. Since any polynomial with a finite number of terms, coefficients in $\mathbb{Z}$, and exponents in $\{0\}\times O_n$ (under Hessenberg addition) can only have a finite number of unique factors, $\mathcal{F}_\psi$ is a finite set. Since $\mathbb{Z}_\infty$ is totally ordered and $\mathcal{F}_\psi$ is a finite set, we may choose the greatest member of $\mathcal{F}_\psi$, call it $f^\uparrow$. We then have that $\exists\phi\in\mathbb{Z}_\infty\times\mathbb{Z}_\infty^+\sim\{0\}$ such that $$\frac{1^{st}\psi}{2^{nd}\psi}=\frac{f^\uparrow\times1^{st}\phi}{f^\uparrow\times2^{nd}\phi},$$ thus $\phi\in\overline{\underline{\psi}}$. Further, suppose that $\mathcal{F}_\phi\neq\{1\}$, and let $g^\uparrow>1$ be the greatest member of $\mathcal{F}_\phi$. We then have that $\exists\zeta\in\mathbb{Z}_\infty\times\mathbb{Z}_\infty^+\sim\{0\}$ such that $$\frac{1^{st}\psi}{2^{nd}\psi}=\frac{f^\uparrow\times1^{st}\phi}{f^\uparrow\times2^{nd}\phi}=\frac{f^\uparrow\times g^\uparrow\times1^{st}\zeta}{f^\uparrow\times g^\uparrow\times2^{nd}\zeta},$$ thus $f^\uparrow\times g^\uparrow\in\mathcal{F}_\psi$ and $f^\uparrow<f^\uparrow\times g^\uparrow$, a contradiction since $f^\uparrow$ is the greatest member of $\mathcal{F}_\psi$. We thusly conclude that $\mathcal{F}_\phi=\{1\}$, thus $\phi\in\mathbb{Q}_\infty$. Now, suppose that $q\in\overline{\underline\psi}$, so $\overline{\underline\psi}=\overline{\underline\phi}=\overline{\underline q}$ by Theorem 10.6 {\it (2)} and we thusly have that there exist some $a,b$ such that $a\times1^{st}\psi=b\times1^{st}q$ and $a\times2^{nd}\psi=b\times2^{nd}q$. If $a\neq1$ or $b\neq1$ we have a contradiction, since $a\in\mathcal{F}_\psi$ and $b\in\mathcal{F}_q$. We thusly conclude that $a=1=b$, and consequently we have that $$\frac{1^{st}\phi}{2^{nd}\phi}=\frac{1\times1^{st}\phi}{1\times2^{nd}\phi}=\frac{1\times1^{st}q}{1\times2^{nd}q}=\frac{1^{st}q}{2^{nd}q},$$ thus $\phi=q$ as desired. This completes the proof. \\
\end{proof}

We will use these theorems shortly to define the binary operations on $\mathbb{Q}_\infty$; we are now prepared to first prove one last theorem exactly defining the relationship between $\mathbb{Z}_\infty$ and $\mathbb{Q}_\infty$. \\

\begin{customcor}{10.8}
$$\bigcup_{q\in\mathbb{Q}_\infty}\overline{\underline{q}}=\Big(\mathbb{Z}_\infty\times\mathbb{Z}_\infty^+\sim\{0\}\Big).$$
\end{customcor}
\begin{proof}
It follows directly from Definition {\it 10.5} that $\overline{\underline q}\subseteq\Big(\mathbb{Z}_\infty\times\mathbb{Z}_\infty^+\sim\{0\}\Big)$ for all $q$, thus $$\bigcup_{q\in\mathbb{Q}_\infty}\overline{\underline{q}}\subseteq\Big(\mathbb{Z}_\infty\times\mathbb{Z}_\infty^+\sim\{0\}\Big).$$ By the previous theorem, for all $\psi$ there exists exactly one $q\in\overline{\underline\psi}$, and by Theorem {\it 10.6 (1-2)} we then have that $\psi\in\overline{\underline{q}}$, thus $$\Big(\mathbb{Z}_\infty\times\mathbb{Z}_\infty^+\sim\{0\}\Big)\subseteq\bigcup_{q\in\mathbb{Q}_\infty}\overline{\underline{q}}.$$ These two facts together establish that $$\bigcup_{q\in\mathbb{Q}_\infty}\overline{\underline{q}}=\Big(\mathbb{Z}_\infty\times\mathbb{Z}_\infty^+\sim\{0\}\Big),$$ completing the proof. \\
\end{proof}

Thus we see that the classes $\overline{\underline\psi}$ are equivalence classes that partition $\mathbb{Z}_\infty\times\mathbb{Z}_\infty^+\sim\{0\}$, each containing exactly one member of $\mathbb{Q}_\infty$; by definition, these are all the members of $\mathbb{Q}_\infty$. We are now prepared to define the additive structure in $\mathbb{Q}_\infty$. \\

\begin{customdefn}{10.9}
For all $p,q$, we define $p\breve+q$ to be the unique member of $\mathbb{Q}_\infty$ lying in $$\overline{\underline{(1^{st}p2^{nd}q+2^{nd}p1^{st}q,2^{nd}p2^{nd}q)}}=\overline{\underline{\frac{1^{st}p2^{nd}q+2^{nd}p1^{st}q}{2^{nd}p2^{nd}q}}}.$$ We then define $$\breve+=\{\big((p,q),p\breve+q):p\in\mathbb{Q}_\infty\wedge q\in\mathbb{Q}_\infty\}.$$ $\breve+$ will be called {\bf Surrational addition}, or simply addition, and we will write $p\breve+q=r$ iff $\big((p,q),r\big)\in\breve+$. Further, $p\breve+q$ will be called the {\bf Surrational sum} of $p$ and $q$, or simply the {\bf sum} of $p$ and $q$. This definition is justified by Theorem 10.7 and Corollary 10.8. \\
\end{customdefn}

\begin{customthm}{10.10}
$\breve+$ is a function.
\end{customthm}
\begin{proof}
Suppose that $p\breve+q=r$ and $p\breve+q=s$. We then have that $r,s\in\overline{\underline{(1^{st}p2^{nd}q+2^{nd}p1^{st}q,2^{nd}p2^{nd}q)}}$ and $r,s\in\mathbb{Q}$, thus $r=s$ by Theorem {\it 10.7}. This completes the proof. \\
\end{proof}

Thus we see that adding two surrational numbers simply amounts to adding them as though they were rational functions, then reducing the exponent and coefficient sequences as much as possible in the classical sense.  We now define negation in $\mathbb{Q}_\infty$. \\

\begin{customdefn}{10.11}
$$\breve-=\{(p,q):q=(-1^{st}p,2^{nd}p)\}.$$ $\breve-$ will be called {\bf Surrational negation} or simply {\bf negation}. We will write $-p=q$ iff $(p,q)\in\breve-$, and we will write $p-q=r$ iff $p\breve+(\breve-q)=r$. Further, we will call $p\breve-q$ the {\bf difference between $p$ and $q$}. \\
\end{customdefn}

Thus we see that negating a surrational number simply amounts to negating its numerator, just as with classical rational numbers built out of integers. \\

\begin{customthm}{10.12}
$\breve-$ is a function.
\end{customthm}
\begin{proof}
Suppose $-p=q$ and $-p=r$; then $q=(-1^{st}p,2^{nd}p)=r$. This completes the proof. \\
\end{proof}

We are now prepared to define the ordering on $\mathbb{Q}_\infty$; recall that we defined a set of positive surrational numbers in Definition {\it 10.3 (1)}, and that $\hat0\in\mathbb{Q}_\infty^+$. \\

\begin{customdefn}{10.13}
$$\preceq=\{(p,q):q\breve-p\in\mathbb{Q}_\infty^+\}.$$ We will write $p\preceq q$ iff $(p,q)\in\preceq$. \\
\end{customdefn}

The reader hopefully recognizes this definition from nigh every discussion of the Rationals as an ordered field. It is, up to isomorpism, the only ordering on $\mathbb{Q}_\infty$ that will make it an ordered field under the binary operations we will define.  We now prove that this forms a total ordering on $\mathbb{Q}_\infty$. \\

\begin{customthm}{10.14}
$\preceq$ is a total ordering on $\mathbb{Q}_\infty$.
\end{customthm}
\begin{proof}
Suppose $p\preceq q$ and $q\preceq r$. We then have that $q\breve-p\in\mathbb{Q}_\infty^+$ and $r\breve-q\in\mathbb{Q}_\infty^+$, thus 
\begin{enumerate}
\item $0\leq1^{st}(q\breve-p)=1^{st}q2^{nd}p-2^{nd}q1^{st}p\implies2^{nd}q1^{st}p\leq1^{st}q2^{nd}p$,
\item $0\leq1^{st}(r\breve-q)=1^{st}r2^{nd}q-2^{nd}r1^{st}q\implies 2^{nd}r1^{st}q\leq1^{st}r2^{nd}q$.
\end{enumerate}
Since $\mathbb{Z}_\infty$ is an ordered ring, we note that all denominator ($2^{nd}$) elements are positive and consequently if all numerator elements are also positive we can simply take the product of the right-hand terms above to obtain $$2^{nd}q1^{st}p2^{nd}r1^{st}q\leq1^{st}q2^{nd}p1^{st}r2^{nd}q\implies1^{st}p2^{nd}r\leq2^{nd}p1^{st}r$$ $$\implies 0\leq2^{nd}p1^{st}r-1^{st}p2^{nd}r=1^{st}(r\breve-p),$$ thus $p\preceq r$. If $1^{st}q$ negative then {\it (1)} above gives that $1^{st}p$ must also be negative. Suppose that $1^{st}r$ is positive; we then have that $1^{st}p2^{nd}r$ is negative and $1^{st}r2^{nd}p$ is positive, thus $1^{st}p2^{nd}r\leq2^{nd}p1^{st}r$ once again, thus $0\leq1^{st}r2^{nd}p-2^{nd}r1^{st}p=1^{st}(r\breve-p)$, thus $p\preceq r$. The case if $1^{st}p$ is negative and $1^{st}q,1^{st}r$ are positive is entirely identical. \\ Suppose $1^{st}r$ is negative, so by {\it (2)} above $1^{st}q$ is also negative and thusly by {\it (1)} $1^{st}p$ is negative as well. We then have that $2^{nd}q1^{st}p,2^{nd}r1^{st}q,1^{st}q2^{nd}p$, and $1^{st}r2^{nd}q$ are all negative, thus {\it (1)} with {\it (2)} above then give $$0<1^{st}q2^{nd}p1^{st}r2^{nd}q\leq2^{nd}q1^{st}p2^{nd}r1^{st}q,$$ where we have reversed the inequality since $2^{nd}q1^{st}p2^{nd}r1^{st}q$ is positive and $1^{st}q2^{nd}p1^{st}r2^{nd}q$ is positive. Canceling like terms then yields $$0>1^{st}r2^{nd}p\geq1^{st}p2^{nd}r,$$ where we switched the direction of the inequality once again because we canceled one negative term on each side, $1^{st}q$. We thusly have once again that $1^{st}p2^{nd}r\leq1^{st}r2^{nd}p$, thus $0\leq1^{st}r2^{nd}p-2^{nd}r1^{st}p=1^{st}(r\breve-p)$, thus $p\preceq r$. This completes the proof that $\preceq$ is transitive. \\ Now suppose that $p\preceq q$ and $q\preceq p$, thus $$0\leq1^{st}(q\breve-p)=1^{st}q2^{nd}p-2^{nd}q1^{st}p\implies2^{nd}q1^{st}p\leq1^{st}q2^{nd}p,$$ $$0\leq1^{st}(p\breve-q)=1^{st}p2^{nd}q-2^{nd}p1^{st}q\implies2^{nd}p1^{st}q\leq1^{st}p2^{nd}q,$$ and consequently we have that $$2^{nd}q1^{st}p=1^{st}q2^{nd}p\implies\exists m,n\big[m\times(1^{st}p,2^{nd}p)=n\times(1^{st}q,2^{nd}q)\big]\implies\overline{\underline{p}}=\overline{\underline{q}}.$$ Since each $\overline{\underline{\psi}}$ contains exactly one member of $\mathbb{Q}_\infty$ by Theorem {\it 10.7} and $p,q\in\mathbb{Q}_\infty$, we thusly have that $p=q$. This completes the proof that $\preceq$ is antisymmetric, completing the proof that it is an ordering. \\ Now, suppose $p,q\in\mathbb{Q}_\infty$. Since $\leq$ is a total ordering on $\mathbb{Z}_\infty$ we have that $1^{st}p2^{nd}q\leq2^{nd}p1^{st}q$ or $1^{st}q2^{nd}p\leq2^{nd}q1^{st}p$. If the former holds then $$1^{st}p2^{nd}q\leq2^{nd}p1^{st}q\implies0\leq2^{nd}p1^{st}q-1^{st}p2^{nd}q=1^{st}(q\breve-p)\implies p\preceq q,$$ and if the latter holds then $$1^{st}q2^{nd}p\leq2^{nd}q1^{st}p\implies0\leq2^{nd}q1^{st}p-1^{st}q2^{nd}p=1^{st}(p\breve-q)\implies q\preceq p.$$ Since $p$ and $q$ were arbitrary in $\mathbb{Q}_\infty$, $\preceq$ is a total ordering on $\mathbb{Q}_\infty$. This completes the proof. \\
\end{proof}

We thusly see that $\mathbb{Q}_\infty$ can be totally ordered in the fashion that we know orders all ordered fields, up to isomorphism. We now show that this ordering could have been defined immedately in terms of products of members of $\mathbb{Z}_\infty$. \\

\begin{customthm}{10.15}
$$\preceq=\{(p,q):2^{nd}q1^{st}p\leq1^{st}q2^{nd}p\}.$$
\end{customthm}
\begin{proof}
$$p\preceq q\iff 0\leq1^{st}(q\breve-p)=1^{st}q2^{nd}p-2^{nd}q1^{st}p\iff2^{nd}p1^{st}q\leq1^{st}q2^{nd}p.$$ \\
\end{proof}

Thusly we see that ordering the rationals can also be viewed as requiring that the mixed products from our ordered pairs of surintegers, treated as ratios, obey an inequality in the Surintegers.  We now define the multiplicative structure in $\mathbb{Q}_\infty$. \\

\begin{customdefn}{10.16}
For all $p,q,$ we define $p\breve\times q$ to be the unique member of $\mathbb{Q}_\infty$ lying in $$\overline{\underline{(1^{st}p1^{st}q,2^{nd}p2^{nd}q)}}=\overline{\underline{\frac{1^{st}p1^{st}q}{2^{nd}p2^{nd}q}}}.$$ We then define $$\breve\times=\{\big((p,q)p\breve\times q\big):p\in\mathbb{Q}_\infty\wedge q\in\mathbb{Q}_\infty\}.$$ $\breve\times$ will be called {\bf Surrational multiplication}, or simply multiplication, and we will write $p\breve\times q=r$ iff $\big((p,q),r\big)\in\breve\times$. Further, $p\breve\times q$ will be called the {\bf Surrational product} of $p$ and $q$, or simply the {\bf product} of $p$ and $q$. This definition is justified by Theorem 10.7 and Corollary 10.8. \\
\end{customdefn}

\begin{customthm}{10.17}
$\breve\times$ is a function.
\end{customthm}
\begin{proof}
Suppose $p\breve\times q=r$ and $p\breve\times q=s$. We then have that $r,s\in\overline{\underline{(1^{st}p1^{st}q,2^{nd}q2^{nd}q)}}$ and $r,s\in\mathbb{Q}_\infty$, thus $r=s$ by Theorem {\it 10.7}. This completes the proof. \\
\end{proof}

We thusly see that multiplication in $\mathbb{Q}_\infty$ behaves exactly like multiplication of rational numbers, up to the selection of a member in an equivalence class after we execute the operations. We now define the notion of inversion in $\mathbb{Q}_\infty$, something which will have no analogous operation in $\mathbb{Z}_\infty$. \\

\begin{customdefn}{10.18}
$$\breve\div=\{(p,q):\big[q=(2^{nd}p,1^{st}p)\iff p\in\mathbb{Q}_\infty^+\sim\{0\}\big]\wedge\big[q=(-2^{nd}p,-1^{st}p)\iff p\in\mathbb{Q}_\infty^-\big]\wedge\big[q=\hat0\iff p=\hat0\big]\}.$$ $\breve\div$ will be called {\bf Surrational inversion}, or simply {\bf inversion}. We will write $q=\frac{\hat1}{p}$ iff $(p,q)\in\breve\div$ and refer to $q$ as the {\bf inverse} of $p$, and we will write $p\breve\div q=r$ iff $p\breve\times(\breve\div q)=r$. Further, we will write $p\breve\div q$ as $\frac{p}{q}$, and refer to $\frac{p}{q}$ as the {\bf quotient of $p$ with respect to $q$}. \\
\end{customdefn}

This definition is a bit technical, however this is because we required our denominators to be strictly positive, so if $p\in\mathbb{Q}_\infty^-$ and consequently $1^{st}p<0$ we must negate both the numerator and denominator when inverting to legitimately obtain another member of $\mathbb{Q}_\infty$. A similar situation holds if $p=0$, hence the special case given for this situation. \\

\begin{customthm}{10.19}
$\breve\div$ is a function.
\end{customthm}
\begin{proof}
Suppose $\frac{p}{q}=r$ and $\frac{p}{q}=s$. Since $\leq$ is a total ordering on $\mathbb{Z}_\infty$, we have that $0\leq 1^{st}p$ or $1^{st}p<0$. If $0<1^{st}p$ then $p\in\mathbb{Q}_\infty^+$, thus $r=(2^{nd}p,1^{st}p)=s$. If $0=1^{st}p$ then $p=\hat0$, thus $r=\hat0=s$. If $1^{st}p<0$ then $p\in\mathbb{Q}_\infty^-$, thus $r=(-2^{nd}p,-1^{st}p)=s$ once again. This completes the proof. \\
\end{proof}

\begin{customdefn}{10.20}
For all $\overline{\underline\psi}$, we define $\overline{\overline\psi}$ to be the unique member of $\overline{\underline\psi}$ lying in $\mathbb{Q}_\infty$. This definition is justified by Theorem {\it 10.7} and Corollary {\it 10.8}. \\
\end{customdefn}

We have now constructed all of the essential components of $\mathbb{Q}_\infty$ in a convenient and concise form.  We now give an exact definition for a field in the context of MK class theory. \\

\begin{customdefn}{10.21}
A set $\mathfrak{F}=\langle\mathbb{F},\leq_\mathbb{F},+_\mathbb{F},-_\mathbb{F},\times_\mathbb{F},\div_\mathbb{F},0_\mathbb{F},1_\mathbb{F}\rangle$ is a {\bf ordered field} iff $\mathfrak{F}\sim\div_\mathbb{F}$ is an ordered commutative ring, $\div_\mathbb{F}:\mathbb{F}\rightarrow\mathbb{F}$ is a function, and for all $a\in\mathbb{F}$ there exists some $b\in\mathbb{F}$ such that $a\times_\mathbb{F}b=1_\mathbb{F}$. We will write this $b$ as $\div_\mathbb{F}(a)$, or $\frac{1_\mathbb{F}}{a}$. \\
\end{customdefn}

\begin{customdefn}{10.22}
A proper class $\mathfrak{F}=\mathbb{F}\cup\leq_\mathbb{F}\cup+_\mathbb{F}\cup-_\mathbb{F}\cup\times_\mathbb{F}\cup\div_\mathbb{F}$ with $0_\mathbb{F},1_\mathbb{F}\in\mathbb{F}$ is a {\bf ordered Field} iff $\mathfrak{F}\sim\div_\mathbb{F}$ is an ordered commutative Ring, $\div_\mathbb{F}:\mathbb{F}\rightarrow\mathbb{F}$ is a function, and for all $a\in\mathbb{F}$ there exists some $b\in\mathbb{F}$ such that $a\times_\mathbb{F}b=1_\mathbb{F}$. We will write this $b$ as $\div_\mathbb{F}(a)$, or $\frac{1_\mathbb{F}}{a}$. \\
\end{customdefn}

Thus we see that an ordered f/Field is simply an ordered r/Ring in which every element has a multiplicative inverse. \\

\begin{customdefn}{10.23}
$$\mathcal{Q}_\infty=\mathbb{Q}_\infty\cup\preceq\cup\breve+\cup\breve-\cup\breve\times\cup\breve\div.$$ $\mathcal{Q}_\infty$ will be called the Surrationals$_+$. \\
\end{customdefn}

We now prove that $\mathcal{Q}_\infty$ forms an ordered Field. \\

\begin{customthm}{10.24 -- Main Theorem 4}
$\mathcal{Q}_\infty$ is an ordered Field, with $0_{\mathbb{Q}_\infty}=\hat0$ and $1_{\mathbb{Q}_\infty}=\hat1$.
\end{customthm}
\begin{proof}
We will begin by establishing that $\mathbb{Q}_\infty$ satisfies the items on the list in Definition {\it 9.28}. For {\it (1)}, we observe that $$p\breve+(q\breve+r)=p\breve+\overline{\overline{(1^{st}q2^{nd}r+2^{nd}q1^{st}r,2^{nd}q2^{nd}r)}}=\overline{\overline{\big(1^{st}p(2^{nd}q2^{nd}r)+2^{nd}p(1^{st}q2^{nd}r+2^{nd}q1^{st}r),2^{nd}p(2^{nd}q2^{nd}r)\big)}}$$ $$=\overline{\overline{\big((1^{st}p2^{nd}q+2^{nd}p1^{st}q)2^{nd}r+(2^{nd}p2^{nd}q)1^{st}r,(2^{nd}p2^{nd}q)2^{nd}r\big)}}=\overline{\overline{(1^{st}p2^{nd}q+2^{nd}p1^{st}q,2^{nd}p2^{nd}q)}}\breve+r=(p\breve+q)\breve+r,$$ since $\mathcal{Z}_\infty$ forms an ordered Ring. This completes the proof of {\it (1)}. \\ For {\it (2)}, we simply observe that $$p\breve+q=\overline{\overline{(1^{st}p2^{nd}q+2^{nd}p1^{st}q,2^{nd}p2^{nd}q)}}=\overline{\overline{(1^{st}q2^{nd}p+2^{nd}q1^{st}p,2^{nd}q2^{nd}p)}}=q\breve+p,$$ since $\times$ and $+$ are commutative in $\mathbb{Z}_\infty$. This completes the proof of {\it (2)}. \\ For {\it (3)}, we observe that $$p\breve+\hat0=(1^{st}p\times 1+2^{nd}p\times0,2^{nd}p\times1)=(1^{st}p,2^{nd}p)=p.$$ This completes the proof of {\it (3)}. \\ For {\it (4)}, we observe that $$p\breve-p=\overline{\overline{(1^{st}p2^{nd}p-2^{nd}p1^{st}p,2^{nd}p2^{nd}p)}}=(0,1)=\hat0.$$ This completes the proof of {\it (4)}, completing the proof that $\mathcal{Q}_\infty$ is an Abelian group under $\breve+$ and $\breve-$. \\ For {\it (5)}, we observe that $$p\breve\times (q\breve\times r)=p\breve\times(1^{st}q1^{st}r,2^{nd}q2^{nd}r)=\big(1^{st}p(1^{st}q1^{st}r),2^{nd}p(2^{nd}q2^{nd}r)\big)$$ $$=\big((1^{st}p1^{st}q)1^{st}r,(2^{nd}p2^{nd})q2^{nd}r\big)=(1^{st}p1^{st}q,2^{nd}p2^{nd}q)\breve\times r=(p\breve\times q)\breve\times r.$$ This completes the proof of {\it (5)}. \\ For {\it (6)}, we observe that $$p\breve\times q=(1^{st}p1^{st}q,2^{nd}p2^{nd}q)=(1^{st}q1^{st}p,2^{nd}q2^{nd}p)=q\breve\times p.$$ This completes the proof of {\it (6)}. \\ For {\it (7)}, we observe that $$p\breve\times\hat1=(1^{st}p\times1,2^{nd}p\times1)=(1^{st}p,2^{nd}p)=p.$$ This completes the proof of {\it (7)}. \\ Prior to proving {\it (8)}, we will now prove that all elements in $\mathbb{Q}_\infty$ have multiplicative inverses. We observe that if $p\in\mathbb{Q}_\infty^+\sim\{\hat0\}$, then $$p\breve\times\frac{\hat1}{p}=\overline{\overline{(1^{st}p2^{nd}p,2^{nd}p1^{st}p)}}=(1,1)=\hat1.$$ This completes the proof that all elements of $\mathbb{Q}_\infty$ have multiplicative inverses, completing the proof that $\mathbb{Q}_\infty$ is an Abelian group under $\breve\times$ and $\breve\div$. \\ For {\it (8)}, we observe that $$p\breve\times(q\breve+r)=p\breve\times\overline{\overline{(1^{st}q2^{nd}r+2^{nd}q1^{st}r,2^{nd}q2^{nd}r)}}=\overline{\overline{\big(1^{st}p(1^{st}q2^{nd}r+2^{nd}q1^{st}r),2^{nd}p(2^{nd}q2^{nd}r)\big)}}$$ $$=\overline{\overline{\big(1^{st}p1^{st}q2^{nd}r+2^{nd}q1^{st}p1^{st}r,2^{nd}q2^{nd}p2^{nd}r\big)}},$$ and $$(p\breve\times q)\breve+(p\breve\times r)=(1^{st}p1^{st}q,2^{nd}p2^{nd}q)\breve+(1^{st}p1^{st}r,2^{nd}p2^{nd}r)$$ $$=\overline{\overline{\big((1^{st}p1^{st}q)(2^{nd}p2^{nd}r)+(2^{nd}p2^{nd}q)(1^{st}p1^{st}r),2^{nd}p2^{nd}q2^{nd}p2^{nd}r\big)}}$$ $$=\overline{\overline{\Big(2^{nd}p\big((1^{st}p1^{st}q)(2^{nd}r)+(2^{nd}q)(1^{st}p1^{st}r)\big),2^{nd}p(2^{nd}q2^{nd}p2^{nd}r)\Big)}}$$ $$=\overline{\overline{\big(1^{st}p1^{st}q2^{nd}r+2^{nd}q1^{st}p1^{st}r,2^{nd}q2^{nd}p2^{nd}r\big)}},$$ thus $p\breve\times(q\breve+r)=(p\breve\times q)\breve+(p\breve\times r)$ as desired. This completes the proof of {\it (8)}, completing the proof that $\mathcal{Q}_\infty$ is a f/Field under $\breve+,\breve-,\breve\times,$ and $\breve\div$. \\ For {\it (9-10)} we could appeal to a well known theorem in Field theory concerning the ordering we have chosen, however we will prove both directly instead as they are quite concise. For {\it (9)}, we observe that if $p\prec q$ and $r\prec s$, then $1^{st}p2^{nd}q<2^{nd}p1^{st}q$ and $1^{st}r2^{nd}s<2^{nd}r1^{st}s$ by Theorem {\it 10.15}. Consequently, we have $$1^{st}(p\breve+r)2^{nd}(q\breve+s)=(1^{st}p2^{nd}r+2^{nd}p1^{st}r)(2^{nd}q2^{nd}s)=(1^{st}p2^{nd}q)2^{nd}r2^{nd}s+2^{nd}p2^{nd}r(1^{st}r2^{nd}s)$$ $$<(2^{nd}p1^{st}q)2^{nd}r2^{nd}s+2^{nd}p2^{nd}q(2^{nd}r1^{st}s)=(2^{nd}p2^{nd}r)(1^{st}q2^{nd}s+2^{nd}q1^{st}s)=2^{nd}(p\breve+r)1^{st}(q\breve+s),$$ thus $p\breve+r\prec q\breve+s$ by Theorem {\it 10.15} once again. This completes the proof of {\it (9)}. \\ Finally, for {\it (10)} we simply observe that if $\hat0\prec p$ and $\hat0\prec q$, then $0<1^{st}p$ and $0<1^{st}q$, thus $$0\prec(1^{st}p1^{st}q,2^{nd}p2^{nd}q)=p\breve\times q$$ since $0\leq1^{st}p1^{st}q$.  This completes the proof of {\it (10)}, completing the proof that $\mathcal{Q}_\infty$ is an ordered f/Field. \\ To see that $\mathbb{Q}_\infty$ is a proper class, we simply observe that $dmn\mathbb{Q}_\infty=\mathbb{Z}_\infty$ and $\mathbb{Z}_\infty$ is a proper class, so if $\mathbb{Q}_\infty$ was a set then $dmn\mathbb{Q}$ would be a set and we would have a contradiction. This completes the proof that $\mathbb{Q}_\infty$ is a proper class, completing the proof that $\mathcal{Q}_\infty$ is an ordered Field. This completes the proof. \\
\end{proof}

Having shown rigorously that $\mathcal{Q}_\infty$ has all the structure of a Field, we now take some notational convieniences common to field-theoretic discussions for the sake of conciseness. We make these changes explicit with the next definition. \\

\begin{customdefn}{10.25}
From now on, lowercase italic letters from the middle of the alphabet $p,q,r,\dots$ will denote surintegers, while $-p,-q,-r,\dots$ will denote their negations and $\frac{1}{p},\frac{1}{q},\frac{1}{r},\dots$ will denote their inverses. We will also write $\preceq$ as $\leq$, $p\breve+q$ as $p+q$, $p\breve\times q$ as $pq$, $p\breve\times\frac{1}{q}$ as $\frac{p}{q}$, and we will identify surintegers $a$ in general with their corresponding surrational number $(a,1)$, and ordinals $\zeta$ in general with their corresponding surrational number $\big((0,\zeta),(0,1)\big)$. \\ Further, for $\times$-numbers $\alpha$ and $\beta$ such that $\alpha=1$ or $\alpha=\omega^{\omega^\zeta}$ and $\beta=\omega^{\omega^\gamma}$, we define $$\frac{\alpha}{\beta}=\frac{1}{\omega^{\omega^\gamma}}\equiv\omega^{-\omega^\zeta},$$ or $$\frac{\alpha}{\beta}=\frac{\omega^{\omega^\zeta}}{\omega^{\omega^\gamma}}\equiv\omega^{\omega^\gamma-\omega^\zeta}.$$ \\
\end{customdefn}

Accordingly $p,q,r,\dots$ will denote surrational numbers, $a,b,c,\dots$ will denote surintegers, and $i,j,k,l,m,n,\dots$ will denote natural numbers, in general. If we run into a situation where the alphabets would overlap due to a profusion of variables, we will simply begin priming at the beginning of the list again. Further, we now have that expressions like $\omega^{-\omega^2+\omega}$ can be legitimately interpreted as $\frac{\omega^\omega}{\omega^{\omega^2}}$, etc. \\ We now identify a unique Archimedean subfield of $\mathbb{Q}_\infty$ and identify it with the Rational numbers. \\

\begin{customdefn}{10.26}
$$\mathbb{Q}=\mathbb{Q}_\infty\upharpoonleft\upharpoonright \mathbb{Z}.$$ $\mathbb{Q}$ will be called the {\bf Rational numbers}. Further, we define
\begin{enumerate}
\item $\leq_\omega\ =\ \leq\upharpoonleft\upharpoonright\mathbb{Q}$.
\item $+_\omega=+\upharpoonleft(\mathbb{Q}\times\mathbb{Q})$.
\item $-_\omega=-\upharpoonleft\mathbb{Q}$.
\item $\times_\omega=\times\upharpoonleft(\mathbb{Q}\times\mathbb{Q})$.
\item $\mathcal{Q}=\langle\mathbb{Q},\leq_\omega,+_\omega,-_\omega,\times_\omega,0,1\rangle$. \\
\end{enumerate}
\end{customdefn}

We now proceed to generalize the process used above to form the Rational numbers in order to define a proper class of non-isomorphic ordered fields of somewhat arbitrary transfinite order-type and cardinality. $\mathbb{Q}$ will be the smallest in a cardinality sense and the simplest in an order-type sense, while $\mathbb{Q}_\infty$ will be (by far) the largest and most complicated in cardinality and order-type senses. \\

\begin{customdefn}{10.27}
For any transfinite $\times-number$ $\lambda$ as defined in {\it 8.7}, we define 
\begin{enumerate}
\item $\mathbb{Q}_\lambda=\mathbb{Q}_\infty\upharpoonleft\upharpoonright\mathbb{Z}_\lambda$. $\mathbb{Q}_\lambda$ is called the {\bf $\lambda$-Rational numbers}.
\item $\leq_\lambda=\ \leq\upharpoonleft\upharpoonright\mathbb{Q}_\lambda$.
\item $+_\lambda=+_\infty\upharpoonleft\mathbb{Q}_\lambda\times\mathbb{Q}_\lambda$.
\item $-_\lambda=-_\infty\upharpoonleft\mathbb{Q}_\lambda$.
\item $\times_\lambda=\times_\infty\upharpoonleft\mathbb{Q}_\lambda\times\mathbb{Q}_\lambda$.
\item $\mathcal{Q}_\lambda=\langle\mathbb{Q}_\lambda,\leq_\lambda,+_\lambda,-_\lambda,\times_\lambda,0,1\rangle$.
\item $\mathcal{O}_\mathcal{Q}=\{\mathcal{Q}_\lambda:\lambda \ \text{is a}\ \times\text{-number}\}$. \\
\end{enumerate}
\end{customdefn}

\begin{customthm}{10.28 -- Main Theorem 5}
For any transfinite $\times$-number $\lambda$:
\begin{enumerate}
\item $\mathbb{Q}_\lambda$ is a set.
\item $\leq_\lambda$ is a total ordering on $\mathbb{Q}_\lambda$.
\item $+_\lambda$ is a function.
\item $-_\lambda$ is a function.
\item $\times_\lambda$ is a function.
\item $\div_\lambda$ is a function.
\item $\mathcal{Q}_\lambda$ is a field.
\item $\mathcal{O}_\mathcal{Q}$ is a proper class.
\end{enumerate}
\end{customthm}
\begin{proof}
To see that $\mathbb{Q}_\lambda$ is a set, we consider that $\mathbb{Q}_\lambda\subset(\lambda\times\lambda\times\lambda\times\lambda)$ as a consequence of Theorem {\it 10.4}, and $\lambda\times\lambda\times\lambda\times\lambda\subset\mathcal{P}\lambda$, and since $\lambda$ is a set by Theorem {\it 2.12} and $\mathcal{P}\lambda$ is then a set by the power set axiom, $\mathbb{Q}_\lambda$ is also a set. This completes the proof of {\it (1)}. {\it (2-6)} follow immedately from Theorems {\it 10.14}, {\it 10.10}, {\it 10.12}, {\it 10.17}, and {\it 10.19}, with closure following from the fact that $\lambda$ is a $\times$-number together with Definition 8.6 and Theorem 8.7. {\it (7)} then immedately follows from {\it (1-6)} together with Main Theorem 4 and Definition {\it 10.21}. To prove {\it (8)}, we simply observe that every transfinite cardinal number is a $\times$-number, thus $\mathbb{H}=<\alpha:\mathbb{Q}_\alpha\in \mathcal{O}_\mathcal{Q}>$ is a function with $dmn\mathbb{H}=\mathcal{O}_\mathcal{Q}$ and $Card\subseteq rng\mathbb{H}$. Since $Card$ is a proper class, if $\mathcal{O}_\mathcal{Q}$ was a set we would have a contradiction by the axiom of substitution, since $rng\mathbb{H}$ would be a set, thus we conclude that $\mathcal{O}_\mathcal{Q}$ is a proper class.  This completes the proof.  \\
\end{proof}

Accordingly, we see that every $\mathbb{Q}_\lambda$ is a unique field that is \underline{not} field-isomorphic to any other $\mathbb{Q}_{\lambda'}$, with $\mathbb{Q}_\infty$ being the only Field and consequently not a member of $\mathcal{O}_\mathcal{Q}$. Note that $\omega$ is the smallest transfinite $\times$-number, while by Theorem {\it 7.7 (3)} we have that the next option for 'closing out' a field is at $\omega^\omega$ -- all of the other comments after Main Theorem {\it 2} also apply here. \\ Prior to demonstrating the unique nature of $\mathbb{Q}$, we must first define the notion of a finite sum in a field, and an Archimedean field. \\

\begin{customdefn}{10.29}
Let $\mathfrak{F}=\langle\mathbb{F},\leq_\mathbb{F},+_\mathbb{F},-_\mathbb{F},\times_\mathbb{F},\div_\mathbb{F},0_\mathbb{R},1_\mathbb{F}\rangle$ be a field. If $\mu\in^\omega\mathbb{F}$, we define $$\Sigma(\mu,n)=\mu_0+\mu_1+\dots+\mu_{n-1}.$$ We will write $\Sigma_{i<n}\mu_i$ instead of $\Sigma(\mu,n)$. \\
\end{customdefn}

\begin{customdefn}{10.30}
For any element $f$ in an ordered field $\mathbb{F}\in\mathfrak{F}$, we define $$|f|=\begin{cases}f,&\ 0_\mathbb{F}\leq_\mathbb{F}f, \\ -f,&\ f<_\mathbb{F}0_\mathbb{F}.\end{cases}$$
A field $\mathfrak{F}$ is {\bf Archimedean} iff for any $p,q\in\mathbb{F}$ there exists some $n$ such that $$|q|\leq_\mathbb{F}\Sigma_{i<n}|p|,$$ $$|p|\leq_\mathbb{F}\Sigma_{i<n}|q|.$$ \\
\end{customdefn}

Note that we have not defined the notion of being Archimedean for a Field -- this is because the only proper class fields we will construct, $\mathbb{Q}_\infty$ and $\mathbb{R}_\infty=N_0$, do not even come close to being Archimedean. \\

\begin{customthm}{10.31 -- Main Theorem 6}
$$\mathbb{Q}_\lambda\ \text{is Archimedean}\iff\lambda=\omega.$$
\end{customthm}
\begin{proof}
We first observe that $\mathbb{Q}_\omega=\mathbb{Q}$, and $$p,q\in\mathbb{Q}\implies\exists a,b,c,d\in\mathbb{Z}\big[|p|=\frac{a}{b}\wedge |q|=\frac{c}{d}\big],$$ with $0\leq a$ and $0\leq b$.  If $c=d$, then since $\mathbb{Z}$ is cyclic by Main Theorem {\it 3} we have that $a$ and $b$ are both finite sums of $1$s, thus if $a\leq b$ we may simply add more chunks of $a$ $1s$ to $a$ until we have more than $b$ $1's$. The same is obviously true if $b\leq a$.  If $c\neq d$, then we can make an identical argument about $ad$ and $cb$ since $0<c,d\in\mathbb{Z}$, thus $\mathbb{Q}$ is Archimedean. This completes the proof one way. \\ Now suppose $\lambda\neq\omega$; we then have that $\omega\in\mathbb{Q}_\lambda$ and $|\omega|$ is greater than $\Sigma_{i<n}|1|$ for all $n$, thus $\mathbb{Q}_\lambda$ is not Archimedean. This completes the proof. \\
\end{proof}

We thusly see that $\mathbb{Q}$ forms the only Archimedean ordered field of fractions, and that this is a nigh direct result of the fact that $\mathbb{Z}$ is the only cyclic ring in our construction. We are now prepared to define the Surreal numbers, and explicit examples of all of Hausdorffs $\eta_\varepsilon$-fields -- prior to doing so, we will prove one last quick Lemma regarding $\mathbb{Q}_\infty$, showing that it does not contain any of the 'irrational' numbers in $\mathbb{R}$ as one might expect. \\

\begin{customlem}{10.32} \
\begin{enumerate}
\item $\mathbb{Q}_\infty$ is dense in itself.
\item $\mathbb{Q}$ is not dense in any $\mathbb{Q}_\lambda$ for $\lambda>\omega$. 
\end{enumerate}
\end{customlem}
\begin{proof}
Suppose $a,b\in\mathbb{Q}_\infty$ such that $a<b$. We then have that $a<\frac{a+b}{2}<b$, and since $a$ and $b$ were arbitrary in $\mathbb{Q}_\infty$ this completes the proof of {\it (1)}. \\ For {\it (2)}, we first observe that $\frac{1}{\omega}\in\mathbb{Q}_{\omega^\omega}$ and $\mathbb{Q}_{\omega^\omega}$ is the smallest field containing $\mathbb{Q}$ by Theorem {\it 7.7}. We then have that for any $q\in\mathbb{Q}\subset\mathbb{Q}_{\omega^\omega}$, there does not exist any $p\in\mathbb{Q}$ such that $q<p<q+\frac{1}{\omega}$, thus $\mathbb{Q}$ is not dense in $\mathbb{Q}_{\omega^\omega}$. Since $\mathbb{Q}_{\omega^\omega}$ is the smallest field containing $\mathbb{Q}$ and consequently every larger field will also contain $\mathbb{Q}_{\omega^\omega}$ and $\mathbb{Q}$, this completes the proof.
\end{proof}

It is a well known fact that the rational numbers are dense in the irrational numbers -- an immedate consequence of this fact, together with the above theorem, is that $\mathbb{Q}_\lambda$ for $\lambda>\omega$ does not contain any irrational numbers and thusly $\mathbb{Q}_\infty$ does not either, as it is the union of all $\mathbb{Q}_\lambda$. Note that $\mathbb{Z}_\infty$ is also the union of all $\mathbb{Z}_\lambda$ -- we will utilize this fact in reverse to define the Surreal numbers in the next section.  \\

\vspace{50mm}

\section{The Surreal Numbers}

This section will be brief compared to sections {\bf 9} and {\bf 10}, as we have already done the large bulk of work necessary to make the definitions we wish to make here.  I believe that the following development could also be done in terms of Hahn series in $\mathbb{Q}_\infty[[T^{\mathbb{Z}_\infty}]]$ or $\mathbb{Q}_\infty[[T^{\mathbb{Q}_\infty}]]$, however a construction in terms of Dedekind cuts seems more natural to me and more closely analogous to Conways original construction. \\  We will begin by constructing set-sized saturated fields, as we cannot directly 'cut up' $\mathbb{Q}_\infty$ since each left and right class would be a proper class. Note that since each $\mathbb{Q}_\lambda$ is a set, cuts in $\mathbb{Q}_\lambda$ will also be sets. \\

\begin{customdefn}{11.1}
A {\bf $\lambda$-cut} is a class $L$ such that
\begin{enumerate}
\item $L\subset\mathbb{Q}_\lambda$, and we define $R=\mathbb{Q}_\lambda\sim L$.
\item $L$ has no greatest element.
\item $\big[l\in L\wedge r\in R\big]\implies l<r$. We may equivalently write that $L<R$.
\end{enumerate}
$L$ may be called a {\bf left set}, and $R$ may be called a {\bf right set}. If $R$ has a least element, $L$ is called a {\bf surrational} $\lambda$-cut.  If $R$ has no least element, then $L$ will be called an {\bf irrational} $\lambda$-cut. \\
\end{customdefn}

Note that an immedate consequence of {\it (1)} is that $L\cup R=\mathbb{Q}_\lambda$, and {\it (3)} implies that $L\cap R$ is empty. \\

\begin{customdefn}{11.2}
$$\mathbb{R}_\lambda=\{L:L\ \text{is a}\ \lambda-\text{cut}\}.$$ $\mathbb{R}_\lambda$ will be called the {\bf $\lambda$-Real numbers}. Lowercase italic letters from the end of the alphabet $x,y,z$ will denote real numbers unless otherwise specified, and if $x=L$ we will refer to $L_x=L$ as the {\bf left set of $x$} and $R_x=R=\mathbb{Q}_\lambda\sim L_x$ as the {\bf right set of $x$}. We wll also write $l_x$ for an arbitrary member of $L_x$, and $r_x$ for an arbitrary member of $R_x$. \\
\end{customdefn}

We thusly see that the members of $\mathbb{R}_\lambda$ are cuts in $\mathbb{Q}_\lambda$, in the classical sense -- we are simply cutting a much 'longer' line than we typically do to create the Reals. We now identify particular members of $\mathbb{R}_\lambda$ that will correspond to cuts at surrational numbers. \\

\begin{customthm}{11.3}
For all $q\in\mathbb{Q}_\lambda$, the subset $L_q$ defined by $$L_q=\{\frac{a}{b}\in\mathbb{Q}_\lambda:a\times2^{nd}q<b\times1^{st}q\}$$ is a surrational $\lambda$-cut.
\end{customthm}
\begin{proof}
We first observe that $q<p\implies2^{nd}p1^{st}q\leq1^{st}p2^{nd}q\implies p\notin L_q$, thus $L_q\subset\mathbb{Q}_\lambda$, and $L_q\cup R_q=L_q\cup\big(\mathbb{Q}_\lambda\sim L_q\big)=\mathbb{Q}_\lambda$ by definition.  For {\it (2)}, we observe that $$2^{nd}q\times a<1^{st}q\times b\iff \frac{a}{b}<q\iff \frac{a}{b}<\frac{a\times2^{nd}q+b\times1^{st}q}{2\times b\times2^{nd}q}<q,$$ thus $L_q$ has no greatest element. For {\it (3)} we observe that $R_q=\{\frac{a}{b}\in\mathbb{Q}_\lambda:b\times1^{st}q\leq a\times2^{nd}q\}$, and suppose that $l\in L_q$ and $r\in R_q$. We then have that $1^{st}l2^{nd}q<2^{nd}l1^{st}q$ and $1^{st}q2^{nd}r\leq2^{nd}q1^{st}r$, thus $$1^{st}l2^{nd}q1^{st}q2^{nd}r<2^{nd}l1^{st}q2^{nd}q1^{st}r\implies 1^{st}l2^{nd}r<2^{nd}l1^{st}r\implies l<r,$$ completing the proof that $L_q$ is a $\lambda$-cut. Finally, we observe that $q\in R_q$ and $$p<q\implies 2^{nd}q1^{st}p<1^{st}q2^{nd}p\implies p\in L_q,$$ thus $q$ is the least member of $R_q$. This completes the proof that $L_q$ is a surrational $\lambda$-cut for all $q\in\mathbb{Q}_\lambda$. \\
\end{proof}

Thus wee see that each surrational number can be identified with a surrational $\lambda$-cut. We will now identify cuts that will act as $n^{th}$ roots for all surrational numbers -- if the surrational number is not already the product of some other surrational number multiplied by itself $n$ times, it will correspond to an irrational cut. \\

\begin{customthm}{11.4}
For all surrational numbers $q$, the cut $L_{\sqrt[n]{q}}$ defined by $$L_{\sqrt[n]{q}}=\{\frac{a}{b}\in\mathbb{Q}_\lambda:a^n\times2^{nd}q<b^n\times1^{st}q\}$$ is a surrational $\lambda$-cut if $q=p^n$ for some $p\in\mathbb{Q}_\lambda$, and an irrational $\lambda$-cut otherwise.
\end{customthm}
\begin{proof}
We first observe that $q<p^n\implies2^{nd}p^n1^{st}q\leq1^{st}p^n2^{nd}q\implies p\notin L_{\sqrt[n]{q}}$, thus $L_{\sqrt[n]{q}}\subset\mathbb{Q}_\lambda$, and $L_{\sqrt[n]{q}}\cup R_{\sqrt[n]{q}}=L_{\sqrt[n]{q}}\cup\big(\mathbb{Q}_\lambda\sim L_{\sqrt[n]{q}}\big)=\mathbb{Q}_\lambda$ by definition. For {\it (2)}, we observe that $$2^{nd}q\times a^n<1^{st}q\times b^n\implies \frac{a^n}{b^n}<q\implies \frac{a^n}{b^n}<\frac{a^n\times2^{nd}q+b^n\times1^{st}q}{2\times b^n\times2^{nd}q}<q,$$ thus $L_{\sqrt[n]{q}}$ has no greatest element. For {\it (3)} we observe that $R_q=\{\frac{a}{b}\in\mathbb{Q}_\lambda:b^n\times1^{st}q\leq a^n\times2^{nd}q\}$, and suppose that $l\in L_{\sqrt[n]{q}}$ and $r\in R_{\sqrt[n]{q}}$. We then have that $1^{st}l^n2^{nd}q<2^{nd}l^n1^{st}q$ and $1^{st}q2^{nd}r^n\leq2^{nd}q1^{st}r^n$, thus $$1^{st}l^n2^{nd}q1^{st}q2^{nd}r^n<2^{nd}l^n1^{st}q2^{nd}q1^{st}r^n\implies 1^{st}l^n2^{nd}r^n<2^{nd}l^n1^{st}r^n\implies l^n<r^n,$$ completing the proof that $L_{\sqrt[n]{q}}$ is a $\lambda$-cut. Finally, suppose $q=p^n$; we then have that $p\in R_{\sqrt[n]{q}}$ and $$r<p\implies 2^{nd}p1^{st}r<1^{st}p2^{nd}r\implies (2^{nd}p1^{st}r)^n<(1^{st}p2^{nd}r)^n\implies 2^{nd}p^n1^{st}r^n<1^{st}p^n2^{nd}r^n$$ $$\implies 2^{nd}q1^{st}r^n<1^{st}q2^{nd}r^n\implies r\in L_{\sqrt[n]{q}},$$ thus $p$ is the least member of $R_{\sqrt[n]{q}}$ and $L_{\sqrt[n]{q}}$ is consequently a surrational $\lambda$-cut. If there is no $p$ such that $p^n=q$, then $R_{\sqrt[n]{q}}$ has no least element since the equality is never satisfied, thus $L_{\sqrt[n]{q}}$ is an irrational $\lambda$-cut.  This completes the proof. \\
\end{proof}

Thusly we see that all surrational cuts will have $n^{th}$ roots, and further that the $n^{th}$ root of any cut can be defined in an analogous fashion; accordingly, $\mathbb{R}_\lambda$ will be a real-closed field. For example, we now have $$\sqrt{2}=L_{\sqrt{2}}=\{\frac{a}{b}:a^2\leq b^2\times2\},$$ $$\sqrt{\omega}=L_{\sqrt{\omega}}=\{\frac{a}{b}:a^2\leq b^2\times\omega\},$$ etc. From now on, we will associate surrational numbers in general with their corresponding cuts. We now define the ordering we will use in $\mathbb{R}_\lambda$. \\

\begin{customdefn}{11.5}
$$\dot\preceq_\lambda=\{(x,y):L_x\subseteq L_y\}.$$ We will write $x\dot\preceq y$ iff $(x,y)\in\dot\preceq_\lambda$. \\
\end{customdefn}

Thus we see that for $\lambda$-real numbers $x$ and $y$, $x\dot\preceq y$ if and only if $x\subseteq y$. \\

\vspace{20mm}

\begin{customthm}{11.6}
$\dot\preceq_\lambda$ is a total ordering on $\mathbb{R}_\lambda$. 
\end{customthm}
\begin{proof}
Suppose $x\dot\preceq y$ and $y\dot\preceq z$. We then have that $L_x\subseteq L_y\subseteq L_z\implies L_x\subseteq L_z$, thus $x\dot\preceq z$. This completes the proof that $\dot\preceq_\lambda$ is transitive. \\ Now suppose that $x\dot\preceq y$ and $y\dot\preceq x$, so $L_x\subseteq L_y$ and $L_y\subseteq L_x$, thus $L_x=L_y\implies x=y$. This completes the proof that $\dot\preceq_\lambda$ is antisymmetric, completing the proof that it is an ordering. \\ Now suppose $x,y\in\mathbb{R}_\lambda$. If there exists $q\in L_x$ such that $q\notin L_y$, then $q\in R_y\implies\forall p\big[p\in L_y\Rightarrow p<q\big]\implies L_y\subseteq L_x$, thus $y\dot\preceq x$. If there does not exist any such $q$, then $L_x\subseteq L_y$ thus $x\dot\preceq y$. Since $x$ and $y$ were arbitrary in $\mathbb{R}_\lambda$, this completes the proof that $\dot\preceq_\lambda$ is a total ordering on $\mathbb{R}_\lambda$. This completes the proof.\\
\end{proof}

We now define the binary operations we will use in $\mathbb{R}_\lambda$. \\

\begin{customdefn}{11.7}
For all $U,V\subseteq\mathbb{Q}_\infty$, we define $U\tilde+V=\{u+v\}$, where $u\in U$ and $v\in V$. We then define $$\tilde+=\{\big((x,y),z\big):L_z=L_x\tilde+L_y\}.$$ $\tilde+$ will be called {\bf Surreal addition}, and we will write $x\tilde+y=z$ iff $\big((x,y),z\big)\in\tilde+$. \\
\end{customdefn}

\begin{customthm}{11.8}
$\tilde+$ is a function.
\end{customthm}
\begin{proof}
Suppose that $x\tilde+y=z$ and $x\tilde+y=z'$. We then have that $L_z=L_x\tilde+L_y=L_{z'}\implies z=z'$. This completes the proof.  \\
\end{proof}

\begin{customdefn}{11.9}
For all $U\subset\mathbb{Q}_\infty$ such that $U$ has a minimum element, we define $U'=U\sim\{\min{U}\}$, and $\tilde-U=\{-u'\}$, where $u'\in U'$. For all $U\subset\mathbb{Q}_\infty$ such that $U$ has no minimum element, we define $\tilde-U=\{-u\}$. We then define $$\tilde-=\{(x,y):L_y=\tilde-R_x\}.$$ $\tilde-$ will be called {\bf Surreal negation}, and we will write $\tilde-x=y$ iff $(x,y)\in\tilde-$. Further, we will write $x\tilde-y=z$ iff $x\tilde+(\tilde-y)=z$. \\
\end{customdefn}

Note that we have removed the minimum element from any $R_x$ that had one, by definition. \\

\begin{customthm}{11.10}
$\tilde-$ is a function.
\end{customthm}
\begin{proof}
Suppose $\tilde-x=y$ and $\tilde-x=z$. We then have that $y=\tilde-R_x=z$. This completes the proof. \\
\end{proof}

We will now define the expected additive and multiplicative identities in $\mathbb{R}_\lambda$. \\

\begin{customdefn}{11.11}
\begin{enumerate}
\item $\tilde0=\{q:q<0\}=\mathbb{Q}_\lambda^-$.
\item $\tilde1=\{q:q<1\}$. \\
\end{enumerate}
\end{customdefn}

We thusly see that these identities in $\mathbb{R}_\lambda$ simply amount to cuts made at the identities in $\mathbb{Q}_\lambda$. \\

\vspace{20mm}

\begin{customdefn}{11.12}
We define $x\tilde\times y=L_{x\tilde\times y}$ for all $x$ and $y$. \\ If $0\dot\preceq x$ and $0\dot\preceq y$, we define $$L_{x\tilde\times y}=\{q:\exists l_x\exists l_y\big[0\leq l_x\wedge0\leq l_y\wedge q\leq l_xl_y\big]\}.$$ If $0\dot\preceq x$ and $y\dot\prec 0$, then we define $$L_{x\tilde\times y}=\{q:\exists l_x\exists l_y\big[0\leq l_x\wedge l_y<0\wedge q\leq l_xl_y\big]\}.$$ If $0\dot\preceq y$ and $x\dot\prec 0$, then we define $$L_{x\tilde\times y}=\{q:\exists l_x\exists l_y\big[l_x<0\wedge 0\leq l_y\wedge q\leq l_xl_y\big]\}.$$ If $x\dot\prec0$ and $y\dot\prec0$, we define $$L_{x\tilde\times y}=\{q:\exists l_x\exists l_y\big[l_x<0\wedge l_y<0 \wedge q\leq l_xl_y\big]\}.$$ We then define $$\tilde\times=\{\big((x,y),z\big):z=x\tilde\times y\}.$$ $\tilde\times$ will be called {\bf Surreal multiplication}, or simply {\bf multiplication}, and we will write $x\tilde\times y=z$ iff $\big((x,y),z\big)\in\tilde\times$. \\
\end{customdefn}

\begin{customthm}{11.13}
$\tilde\times$ is a function.
\end{customthm}
\begin{proof}
Suppose $x\tilde\times y=z$ and $x\tilde\times y=z'$. For any placement of $x$ and $y$, we then have $z=L_{x\tilde\times y}=z'$. This completes the proof. \\
\end{proof}

\begin{customdefn}{11.14}
$$\tilde\div=\{(x,y):y=\{q:\forall l_x\big[ql_x<1\big]\}\}.$$ $\tilde\div$ will be called {\bf Surreal inversion}, or simply {\bf inversion}. We will write $\frac{1}{x}=y$ iff $(x,y)\in\tilde\div$. \\
\end{customdefn}

\begin{customthm}{11.15}
$\tilde\div$ is a function.
\end{customthm}
\begin{proof}
Suppose $\frac{1}{x}=y$ and $\frac{1}{x}=z$. We then have that $y=\{q:\forall p\big[p\in L_x\implies pq<1\big]\}=z$. This completes the proof. \\
\end{proof}

We have now esablished all of the essential components of $\mathbb{R}_\lambda$ in a convienient and concise form -- we will now prove that $\mathbb{R}_\lambda$ together with the above ordering and binary operations forms a field for all $\lambda$. \\

\begin{customdefn}{11.16}
$$\mathcal{R}_\lambda=\langle\mathbb{R}_\lambda,\dot\preceq_\lambda,\tilde+,\tilde-,\tilde\times,\tilde\div,\tilde0,\tilde1\rangle.$$ $\mathcal{R}$ will be called the {\bf $\lambda$-Real numbers$_+$}. \\
\end{customdefn}

\begin{customthm}{11.17 -- Main Theorem 7}
If $\lambda$ is a $\times$-number, then $\mathcal{R}_\lambda$ is an ordered field.
\end{customthm}
\begin{proof}
Suppose $\lambda$ is a $\times$-number. To see that $\mathbb{R}_\lambda$ is a set, we simply observe that $\mathbb{Q}_\lambda$ is a set by Main Theorem {\it 6} and that $\mathcal{P}\mathbb{Q}_\lambda$ is then a set by the power set axiom, and $\mathbb{R}_\lambda\subset\mathcal{P}\mathbb{Q}_\lambda$, thus $\mathbb{R}_\lambda$ is also a set. \\ For {\it (1)}, we simply observe that $$x\tilde+(y\tilde+z)=x\tilde+\{l_y+l_z\}=\{l_x+(l_y+l_z)\}=\{(l_x+l_y)+l_z\}=\{l_x+l_y\}\tilde+z=(x\tilde+y)\tilde+z,$$ since $+$ is associative in $\mathbb{Q}_\lambda$. This completes the proof of {\it (1)}. \\ For {\it (2)}, we simply observe that $$x\tilde+y=\{l_x+l_y\}=\{l_y+l_x\}=y\tilde+x,$$ since $+$ is commutative in $\mathbb{Q}_\lambda$. This completes the proof {\it (2)}. \\ For {\it (3)}, we first observe that $x\tilde+\tilde0=\{l_x+l_0\}$. Suppose that $q\in L_x$, thus there exists some $p\in L_x$ such that $q<p$ since $L_x$ has no greatest element. We then have that $p+s=p$ for some $s\in\mathbb{Q}_\lambda^-$, thus $q\in\{l_x+l_0\}=x\tilde+\tilde0$, and since $q$ was arbitrary in $L_x$ we thuslt have that $x\subseteq x\tilde+\tilde0$. Now suppose that $r\in L_{x\tilde+0}$, so $r=l_x+l_0$ for some $l_x$ and $l_0$, and consequently we have that $r<l_x\implies r\in L_x$. Since $r$ was arbitrary in $L_{x\tilde+0}$, we thusly have that $x\tilde+\tilde0\subseteq x$. These facts together establish that $x\tilde+\tilde0=x$, completing the proof of {\it (3)}. \\ For {\it (4)}, we simply observe that $$x\tilde+(\tilde-x)=L_x\tilde+(\tilde-R_x)=\mathbb{Q}_\lambda^-=\tilde0.$$ This completes the proof of {\it (4)}, completing the proof that $\mathcal{R}_\lambda$ is an Abelian group under $\tilde+$ and $\tilde-$. \\ For {\it (5)} we simply observe that $p(rs)=(pr)s$ for all $p,r,s\in\mathbb{Q}_\lambda$, thus $q<p(rs)\iff q<(pr)s$ and consequently $x\tilde\times(y\tilde\times z)=(x\tilde\times y)\tilde\times z$ for any placements of $x$, $y$ and $z$ since $\mathbb{Q}_\lambda$ is an ordered field. This completes the proof of {\it (5)}. \\ For {\it (6)}, we observe that $rs=sr$ for all $r,s\in\mathbb{Q}_\lambda$ and thusly $q<rs\iff q<sr$, thus $x\tilde\times y=y\tilde\times x$ for all placements of $x$ and $y$ since $\mathbb{Q}_\lambda$ is an ordered field. This completes the proof of {\it (6)}. \\ For {\it (7)}, we first suppose that $\tilde0\dot\preceq x$ and observe that $\tilde0\dot\preceq1$, thus $$x\tilde\times1=L_{x\tilde\times1}=\{q:\exists l_x\exists l_1\big[0\leq l_x\wedge 0\leq l_1\wedge q<l_xl_1\big]\}.$$ Suppose $q\in L_x$, thus there exists some $l_x$ such that $q<l_x$ since $L_x$ has no greatest element. Consequently there exists some $l_1$ such that $0<l_1<1$ and $q=l_1l_x$, thus we may simply observe that $$l_1<\frac{l_1+1}{2}<1\implies\frac{l_1+1}{2}\in L_1$$ and that $q<\frac{l_1+1}{2}l_x \implies q\in L_{x\tilde\times1}$. Since $q$ was arbirtary in $L_x$, we thusly have that $L_x\subseteq L_{x\tilde\times1}$. Now suppose that $s\in L_{x\tilde\times1}$, thus there exist $0\leq l_x$ and $0\leq l_1$ such that $s<l_xl_1$, thus $s<l_x\implies s\in L_x$. Since $s$ was arbitrary in $L_{x\tilde\times1}$, we have that $L_{x\tilde\times1}\subseteq L_x$ and consequently $L_x=L_{x\tilde\times1}\implies x=x\tilde\times1$. The proof if $x\dot\prec0$ is completely symmetric. This completes the proof of {\it (7)}. \\ We now prove that multiplicative inverses exist in $\mathbb{R}_\lambda$. Suppose $\tilde0\dot\prec x$, so $\exists l_x\big[0<l_x\big]$, and since there exists some $l'_x>l_x>0$ in $L_x$ and thusly $0<l_x\times\frac{1}{l'_x}=\frac{l_x}{l'_x}<1\implies0<\frac{1}{l'_x}\in L_{\frac{1}{x}}$, thus $$\tilde0\dot\prec L_{\frac{1}{x}}=\frac{1}{x}.$$ We consequently have that $$x\tilde\times\frac{1}{x}=\{q:\exists l_x\exists l_{\frac{1}{x}}\big[0\leq l_x\wedge 0\leq l_{\frac{1}{x}}\wedge q<l_xl_{\frac{1}{x}}\big]\}=\{q:q<1\}=\tilde1.$$ The proof if $x\dot\prec\tilde0$ is entirely symmetric, thus we have that inverses exist for all $x\in\mathbb{R}_\lambda$. This completes the proof that $\mathcal{R}_\lambda$ forms an Abelian group under $\tilde\times$ and $\tilde\div$. \\ For {\it (8)}, we first observe that $$x\tilde\times(y\tilde+z)=L_x\tilde\times(L_y\tilde+L_z)=L_x\tilde\times\{l_y+l_z\}.$$ Suppose that $\tilde0\dot\preceq x,y,z$, and $\tilde0\dot\preceq(y\tilde+z)=\{l_y+l_z\}$. We then have that $$L_x\tilde\times\{l_y+l_z\}=\{q:\exists l_x\exists l_{y\tilde+z}\big[0\leq l_x\wedge0\leq l_{y\tilde+z}\wedge q<l_xl_{y\tilde+z}\big]\},$$ $$q<l_xl_{y\tilde+z}\iff\exists l_y\exists l_z\big[q<l_x(l_y+l_z)=l_xl_y+l_xl_z\big].$$ Further, we have that $$(x\tilde\times y)\tilde+(x\tilde\times z)=\{q:\exists l_x\exists l_y\big[0\leq l_x\wedge0\leq l_y\wedge q<l_xl_y\big]\}\tilde+\{q:\exists l'_x\exists l_z\big[0\leq l'_x\wedge0\leq l_z\wedge q<l'_xl_z\big]\},$$ and letting $l'_x=l_x$ we have that $q<l_xl_y$ together with $q<l_xl_z$ yields $q<l_xl_y+l_xl_z$. We thusly see that the cuts defined by $x\tilde\times(y\tilde+z)$ and $(x\tilde+y)\tilde\times(x\tilde+z)$ are identical, thus $x\tilde\times(y\tilde+z)=(x\tilde+y)\tilde\times(x\tilde+z)$ as desired. The proof if the various elements lie before $\tilde0$ follows by symmetry, since $\mathbb{Q}_\lambda$ forms an ordered field. This completes the proof of {\it (8)}, completing the proof that $\mathcal{R}_\lambda$ forms a field under $\tilde+,\tilde-,\tilde\times,$ and $\tilde\div$. \\ For {\it (9)}, suppose that $x\dot\prec y$ and $x'\dot\prec y'$, so $L_x\subset L_y$ and $L_{x'}\subset L_{y'}$, and consequently $$x\tilde+x'=L_x\tilde+L_{x'}=\{l_x+l_{x'}\}\subset\{l_y+l_{y'}\}=L_y\tilde+L_{y'}=y\tilde+y'$$ $$\implies x\tilde+x'\dot\prec y+y'.$$ This completes the proof of {\it (9)}. \\ For {\it (10)}, suppose that $\tilde0\dot\prec x$ and $\tilde0\dot\prec y$, so $0\in L_x$ and $0\in L_y$. Since $L_x$ and $L_y$ have no greatest elements, we thusly have that there exist $l_x>0$ and $l_y>0$, thus $0<l_xl_y$ and consequently $L_0\subset L_{x\tilde\times y}\implies \tilde0\dot\prec x\tilde\times y$. This completes the proof of {\it (10)}, completing the proof that $\mathbb{R}_\lambda$ forms an ordered field. Since $\lambda$ was an arbitrary $\times$-number, this completes the proof that $\mathbb{R}_\lambda$ forms an ordered field for all $\times$-numbers $\lambda$. \\
\end{proof}

We have now constructed explicit examples of all of Hausdorrffs $\eta_\varepsilon$-fields, each being isomorphic to some $\mathbb{R}_\lambda$. We now prove that there are a proper class of such fields, prior to proposing a new definition for the Surreal numbers. \\

\begin{customdefn}{11.18}
$$\mathbb{O}_\mathbb{R}=\{\mathbb{R}_\lambda:\lambda\ \text{is a}\ \times\text{- number}\}.$$
$$\mathcal{O}_\mathcal{R}=\{\mathcal{R}_\lambda:\lambda\ \text{is a}\ \times\text{- number}\}.$$ \\
\end{customdefn}

\begin{customthm}{11.19}
$\mathbb{O}_\mathbb{R}$ and $\mathcal{O}_\mathcal{R}$ are proper classes. 
\end{customthm}
\begin{proof}
We observe that $\mathbb{G}=\langle \lambda:\mathbb{R}_\lambda\in\mathbb{O}_\mathbb{R}\rangle$ and $\mathbb{H}=\langle \lambda:\mathcal{R}_\lambda\in\mathcal{O}_\mathcal{R}\rangle$ are functions with $dmn\mathbb{H}=\mathcal{O}_\mathcal{R}$, $dmn\mathbb{G}=\mathbb{O}_\mathbb{R}$, $Card\subset rng\mathbb{G}$, and $Card\subset rng\mathbb{H}$, thus if $\mathbb{O}_\mathbb{R}$ or $\mathcal{O}_\mathcal{R}$ were sets then $rng\mathbb{H}$ or $rng\mathbb{G}$ would also be a set by the substitution axiom and $Card$ would thusly also be a set, a contradiction since $Card$ is a proper class. This completes the proof. \\
\end{proof}

\begin{customdefn}{11.20}
$$\mathbb{R}=\mathbb{R}_\omega.$$ $\mathbb{R}$ will be called the {\bf Real numbers}. \\
\end{customdefn}

We thusly see that we can obtain the standard real numbers by making the smallest choice possible, $\lambda=\omega$. I propose that many definitions in the Real numbers that involve notions of finiteness, like being compact or bounded, are relics of this choice. Below, I offer an extension of these notions to $\mathbb{R}_{\omega^\omega}$. \\

\begin{customdefn}{11.21}
A class $\mathbb{A}\subseteq\mathbb{R}_{\omega^\omega}$ is {\bf $\omega$-compact} iff for any family of open sets $\mathbb{X}=\{\mathbb{X}_\alpha\}_{\alpha<\varepsilon}\subseteq\mathcal{P}\mathbb{R}_{\omega^\omega}$ such that $\mathbb{A}\subseteq\bigcup_{\alpha<\varepsilon}\mathbb{X}_\alpha$, there exists some $\{\mathbb{X}_\beta\}_{\beta<\omega}\subseteq\{\mathbb{X}_\alpha\}_{\alpha<\varepsilon}$ such that $\mathbb{A}\subseteq\bigcup_{\beta<\omega}\mathbb{X}_\beta$. \\
\end{customdefn}

\begin{customdefn}{11.22}
A class $\mathbb{A}\subseteq\mathbb{R}_{\omega^\omega}$ is {\bf $\omega$-bounded} iff there exist some $x,y\in\mathbb{R}_{\omega^\omega}$ such that $\{x\}<\mathbb{A}<\{y\}$. \\
\end{customdefn}

We are now prepared to propose a new definition for the Surreal numbers. \\

\begin{customdefn}{11.22 -- Main Definition 1}
$$\mathbb{R}_\infty=\bigcup\mathbb{O}_\mathbb{R}.$$ $\mathbb{R}_\infty$ will be called the {\bf Surreal numbers}. \\
\end{customdefn}

Thus we see that the Surreal numbers are composed from cuts of all sizes in every $\mathbb{Q}_\lambda$. Note that every member of $\mathbb{R}_\infty$ is a proper subset of some cut in $\mathbb{Q}_\infty$ -- I believe that these proper class-sized cuts are the equivalent of the 'gaps' that appear in the Surreal numbers. I will have to retool our binary operations and ordering definition slightly, however the intuition behind these definitions will be identical to that behind the above definitions. Here we will only prove that $\mathbb{R}_\infty$ is a proper class -- I am not yet certain that this definition will produce the structure I'm looking for under modified binary operations and a modified ordering, and will accordingly explore that question and publish a subsequent paper when I have resolved my mind on the subject. \\

\begin{customthm}{11.24}
$\mathbb{R}_\infty$ is a proper class.
\end{customthm}
\begin{proof}
If $\mathbb{R}_\infty$ were a set, then we would have that $\mathcal{P}\mathbb{R}_\infty$ is a set by the power set axiom and $\mathbb{O}_\mathbb{R}\subset\mathcal{P}\mathbb{R}_\infty$, thus $\mathbb{O}_\mathbb{R}$ would be a set, a contradiction by Theorem {\it 11.19} -- we thusly conclude that $\mathbb{R}_\infty$ is a proper class. This completes the proof. \\
\end{proof}

Another potentially interesting option is to truncate $\mathbb{Z}_\infty$ at some inaccessible cardinal number $\Omega$, then create a field of fractions and cut it up directly (since the cuts will be sets), then explore the unique properties of each 'inaccessible field' defined by choosing different inaccessible cardinal numbers. I will define such a structure below, and explore it in greater depth in a subsequent paper. \\

\begin{customdefn}{11.25 -- Main Definition 2}
Let $\Omega$ be an inaccessible cardinal number. We then define 
\begin{enumerate}
\item $\mathbb{Z}_\Omega=\mathbb{Z}_\infty\upharpoonleft\upharpoonright\Omega=\{a:1^{st}a<\Omega\wedge2^{nd}a<\Omega\}$. $\mathbb{Z}_\Omega$ will be called the {\bf $\Omega$-inaccessible Integers}.
\item $\mathcal{Z}_\Omega=\langle\mathbb{Z}_\Omega,\leq_\Omega,+_\Omega,-_\Omega,\times_\Omega,0,1\rangle$. $\mathcal{Z}_\Omega$ will be called the {\bf $\Omega$-inaccessible Integers$_+$}.
\item $\mathbb{Q}_\Omega=\mathbb{Q}_\infty\upharpoonleft\upharpoonright \mathbb{Z}_\Omega$. $\mathbb{Q}_\Omega$ will be called the {\bf $\Omega$-inaccessible Rational numbers}.
\item $\mathcal{Q}_\Omega=\langle\mathbb{Q}_\Omega,\leq_\Omega,+_\Omega,-_\Omega,\times_\Omega,\div_\Omega,0,1\rangle$. $\mathbb{Q}_\Omega$ will be called the {\bf $\Omega$-inaccessible Rational numbers$_+$}.
\item $\mathbb{R}_\Omega=\{L:L\ \text{is an}\ \Omega\text{-cut}\}$. $\mathbb{R}_\Omega$ will be called the {\bf $\Omega$-inaccessible Real numbers}.
\item $\mathcal{R}_\Omega=\langle\mathbb{R}_\Omega,\leq_\Omega,+_\Omega,-_\Omega,\times_\Omega,\div_\Omega,0,1\rangle$. $\mathbb{R}_\Omega$ will be called the {\bf $\Omega$-inaccessible Real numbers$_+$}. \\
\end{enumerate}
\end{customdefn}

This concludes our discussion of real-closed fields -- we now briefly outline how to define $\lambda$-Complex numbers, $\Omega$-inaccessible Complex numbers, and the Surcomplex numbers (assuming that the above definition for the Surreal numbers is valid). \\

\section{The Surcomplex Numbers}

Here we provide an outline for how to define an algebraic closure for each $\mathbb{R}_\lambda$, and for $\mathbb{R}_\infty$ if it is isomorphic to the Surreal numbers. \\

\begin{customdefn}{12.1}
$$\mathbb{C}_\lambda=\mathbb{R}_\lambda\times\mathbb{R}_\lambda.$$ $\mathbb{C}_\lambda$ will be called the {\bf $\lambda$-Complex numbers}. Lowercase roman letters $\mathfrak{a},\mathfrak{b},\mathfrak{c},\dots$ will be used to denote arbitrary $\lambda$-complex numbers from now on unless otherwise specified. \\
\end{customdefn}

Thus we see that the $\lambda$-Complex numbers simply amount to $\mathbb{R}_\lambda^2$, with some extra structure. We will provide all of this structure with the next definition. \\

\begin{customdefn}{12.2} \
\begin{enumerate}
\item $\tilde{\hat{+}}=\{\big((\mathfrak{a},\mathfrak{b}),\mathfrak{c}\big):\mathfrak{c}=(1^{st}\mathfrak{a}+1^{st}\mathfrak{b},2^{nd}\mathfrak{a}+2^{nd}\mathfrak{b})\}$.
\item $\tilde{\hat{-}}=\{(\mathfrak{a},\mathfrak{b}):\mathfrak{b}=(-1^{st}\mathfrak{a},-2^{nd}\mathfrak{a})\}$.
\item $\tilde{\hat{\times}}=\{\big((\mathfrak{a},\mathfrak{b}),\mathfrak{c}\big):\mathfrak{c}=(1^{st}\mathfrak{a}1^{st}\mathfrak{b}-2^{nd}\mathfrak{a}2^{nd}\mathfrak{b},1^{st}\mathfrak{a}2^{nd}\mathfrak{b}+2^{nd}\mathfrak{a}1^{st}\mathfrak{b})\}$.
\item $\tilde{\hat{\div}}=\{(\mathfrak{a},\mathfrak{b}):\mathfrak{b}=\Big(\frac{1^{st}\mathfrak{a}}{1^{st}\mathfrak{a}^2+2^{nd}\mathfrak{a}^2},\frac{-2^{nd}\mathfrak{a}}{1^{st}\mathfrak{a}^2+2^{nd}\mathfrak{a}^2}\Big)\}.$
\item $\tilde{\hat{1}}=(1,0)$.
\item $\tilde{\hat{0}}=(0,0)$.
\item $i=(0,1)$. \\
\end{enumerate}
\end{customdefn}

\begin{customdefn}{12.3 -- Main Definition 3}
$$\mathcal{C}_\lambda=\langle\mathbb{C}_\lambda,\tilde{\hat{+}},\tilde{\hat{-}},\tilde{\hat{\times}},\tilde{\hat{\div}},\tilde{\hat{0}},\tilde{\hat{1}},i\rangle.$$ $\mathcal{C}_\lambda$ will be called the {\bf $\lambda$-Complex numbers$_+$}. \\
\end{customdefn}

$\mathcal{C}_\lambda$ thusly defined forms an Algebraically closed field, and we identify $(\mathfrak{a},\mathfrak{b})$ with $\mathfrak{a}+\mathfrak{b}i$, and we can form the standard Complex numbers by choosing $\lambda=\omega$. We now close out the paper by defining classes of inaccessible Complex numbers, and the Surcomplex numbers. \\

\begin{customdefn}{12.4 -- Main Definition 4}
For an inaccessible cardinal number $\Omega$, we define $$\mathbb{C}_\Omega=\mathbb{R}_\Omega\times\mathbb{R}_\Omega.$$ $\mathbb{C}_\Omega$ will be called the {\bf $\Omega$-inaccessible Complex numbers}. Further, we define $$\mathcal{C}_\Omega=\langle\mathbb{C}_\Omega,\tilde{\hat{+}},\tilde{\hat{-}},\tilde{\hat{\times}},\tilde{\hat{\div}},\tilde{\hat{0}},\tilde{\hat{1}},i\rangle.$$ $\mathcal{C}_\Omega$ will be called the {\bf $\Omega$-inaccessible Complex numbers$_+$}. \\
\end{customdefn}

\begin{customdefn}{12.5 -- Main Definition 5}
$$\mathbb{C}_\infty=\mathbb{R}_\infty\times\mathbb{R}_\infty.$$ $\mathbb{C}_\infty$ will be called the {\bf Surcomplex numbers}. Further, we define $$\mathcal{C}_\infty=\langle\mathbb{C}_\infty,\tilde{\hat{+}},\tilde{\hat{-}},\tilde{\hat{\times}},\tilde{\hat{\div}},\tilde{\hat{0}},\tilde{\hat{1}},i\rangle.$$ $\mathcal{C}_\infty$ will be called the {\bf Surcomplex numbers$_+$}. \\
\end{customdefn}

We have thusly shown that the Natural numbers, Integers, Rational numbers, Real numbers, and Complex numbers are all the smallest and simplest versions of constructions that are built on top of each other in the fashion demonstrated in this paper, with the Ordinals, Surintegers, Surrational numbers, Surreal numbers, and Surcomplex numbers representing the largest and most complicated versions of the constructions given in this paper.  For the reader, this hopefully sheds some light on the nature of the vast difference between $\mathbb{R}$ and the Surreal numbers, and provides us with some additional context with which we can attempt to develop a robust theory of analysis over the Surreal numbers. \\

\section{References}
\begin{enumerate}
\item {\it Introduction to Set Theory, J. Donald Monk}, 1969.
\item {\it Cardinal and Ordinal Numbers, W. Sierpinski}, 1958.
\item {\it Foundations of Analysis over Surreal Number Fields, N. Alling} 1987.
\item {\it Principles of Mathematical Analysis, W. Rudin} 1964.
\end{enumerate}

\end{document}